\documentclass[a4paper]{amsart}

\usepackage{txfonts, amsmath,amstext,amsthm,amscd,amsopn,verbatim,amssymb, amsfonts}
\usepackage{fullpage, todonotes}

\usepackage[bbgreekl]{mathbbol}

\usepackage{tikz}
\usepackage{tikz-cd}
\usetikzlibrary{matrix}
\usetikzlibrary{shapes}
\usetikzlibrary{arrows}
\usetikzlibrary{calc,3d}
\usetikzlibrary{decorations,decorations.pathmorphing}
\usetikzlibrary{through}
\tikzset{ext/.style={circle, draw,inner sep=1pt},int/.style={circle,draw,fill,inner sep=1pt},nil/.style={inner sep=1pt}}
\tikzset{exte/.style={circle, draw,inner sep=3pt},inte/.style={circle,draw,fill,inner sep=3pt}}
\tikzset{diagram/.style={matrix of math nodes, row sep=3em, column sep=2.5em, text height=1.5ex, text depth=0.25ex}}
\tikzset{diagram2/.style={matrix of math nodes, row sep=0.5em, column sep=0.5em, text height=1.5ex, text depth=0.25ex}}
\tikzset{every picture/.append style={baseline=-.65ex}}
\tikzset{every loop/.style={draw}}

\usepackage{hyperref}

\theoremstyle{plain}
  \newtheorem{thm}{Theorem}[section]
   
  \newtheorem{defi}[thm]{Definition}
  \newtheorem{prop}[thm]{Proposition}
  
  \newtheorem{cor}[thm]{Corollary}
  \newtheorem{lemma}[thm]{Lemma}
   
\theoremstyle{definition}
  \newtheorem{ex}[thm]{Example}
  \newtheorem{rem}[thm]{Remark}
  \newtheorem{cons}[thm]{Construction}

\newcommand{\sslash}{\mkern-3mu\mathbin{
  \mathchoice{/\mkern-4mu/}
    {/\mkern-4mu/}
    {/\mkern-3mu/}
    {/\mkern-3mu/}}\mkern-3mu}

\newcommand{\alg}[1]{\mathfrak{{#1}}}

\newcommand{\ad}{{\text{ad}}}

\newcommand{\R}{{\mathbb{R}}}
\newcommand{\Z}{{\mathbb{Z}}}
\newcommand{\K}{{\mathbb{K}}}
\newcommand{\Q}{{\mathbb{Q}}}

\newcommand{\fGCc}{{\mathsf{fcGC}}}

\newcommand{\TCG}{{\mathsf{TCG}}}
\newcommand{\Graphs}{{\mathsf{Graphs}}}
\newcommand{\fGraphs}{{\mathsf{fGraphs}}}
\newcommand{\dgHAlgc}{{\mathit{dg}\mathcal{HA}\mathit{lg}}} 
\newcommand{\dgHOpc}{{\mathit{dg}\mathcal{HO}\mathit{p}^c}}

\newcommand{\FreeLie}{\mathrm{Free}_{\Lie}}

\newcommand{\Poiss}{{\mathsf{Poiss}}}

\newcommand{\ICG}{\mathsf{ICG}}
\newcommand{\ICGF}{\ICG^{\rm fr}}
\newcommand{\wICGF}{\widetilde \ICG^{\rm fr}}
\newcommand{\FICG}{\ICGF}
\newcommand{\TCGF}{\TCG^{\rm fr}}
\newcommand{\hotimes}{\hat\otimes}
\newcommand{\tm}{Z}

\newcommand{\Exp}{\mathrm{Exp}}

\newcommand{\hGra}{{\mathsf{hGra}}}

\newcommand{\sGrp}{\mathit{s}\mathcal G\mathit{rp}}

\newcommand{\Aut}{{\mathrm{Aut}}}

\newcommand{\fG}{\mathsf{fG}}
\newcommand{\sG}{\mathsf{G}}

\newcommand{\Def}{\mathrm{Def}}

\newcommand{\op}{\mathcal}

\newcommand{\Lie}{\mathsf{Lie}}

\newcommand{\hoLie}{\mathsf{hoLie}}

\newcommand{\FM}{\mathsf{FM}}

\newcommand{\cT}{\mathcal{T}}

\newcommand{\bpm}{\begin{pmatrix}}
\newcommand{\epm}{\end{pmatrix}}

\newcommand{\GC}{\mathsf{GC}}

\newcommand{\fGC}{\mathsf{fGC}}

\newcommand{\POp}{\op P}
\newcommand{\QOp}{\op Q}
\newcommand{\COp}{\op C}
\newcommand{\DOp}{\op D}
\DeclareMathOperator{\Map}{\mathrm{Map}}
\newcommand{\fg}{\mathfrak g}
\newcommand{\fh}{\mathfrak h}
\newcommand{\MC}{\mathrm{MC}}
\DeclareMathOperator{\Mor}{\mathtt{Mor}}

\newcommand{\mU}{\mathcal{U}}

\DeclareMathOperator{\sgn}{sgn}

\newcommand{\BGC}{\mathsf{BGC}}
\newcommand{\BGraphs}{\mathsf{BGraphs}}
\newcommand{\BstG}{\mathsf{BGraphs}}

\newcommand{\gra}{\mathrm{gra}}

\newcommand{\stG}{\Graphs}
\newcommand{\SO}{\mathit{SO}}

\newcommand{\lD}{\mathsf{D}}
\newcommand{\flD}{\lD^{\mathrm{fr}}}

\newcommand{\Conf}{\mathrm{Conf}}
\newcommand{\dgca}{\mathcal{D}\mathit{gca}}

\newcommand{\Ompoly}{\Omega(\Delta^\bullet)}

\newcommand{\Op}{{\mathcal{O}\mathit{p}}}
\newcommand{\sOp}{{\mathit{s}\Op}}

\newcommand{\catB}{{\mathcal{B}}}
\newcommand{\catC}{{\mathcal{C}}}
\newcommand{\catD}{{\mathcal{D}}}

\newcommand{\Grp}{{\mathcal{G}\mathit{rp}}}

\newcommand{\mF}{\mathcal{F}}
\newcommand{\dgVect}{\mathrm{dgVect}}

\DeclareMathOperator{\DGG}{\mathtt{G}}
\newcommand{\G}{\DGG}
\DeclareMathOperator{\Ind}{\mathrm{Ind}}
\DeclareMathOperator{\diag}{\mathrm{diag}}

\newcommand{\Pp}{P}

\DeclareMathOperator{\sS}{\mathtt{S}}

\newcommand{\DK}{{\alg p}}
\newcommand{\FreeOp}{{\mathbb F}}
\newcommand{\DKF}{\DK^{\mathrm{fr}}}
\newcommand{\stGC}{{}^*\GC}
\newcommand{\tadpole}{
\begin{tikzpicture}[baseline=-.65ex]
\node[int] (v) at (0,0) {};
\draw (v) edge[loop] (v);
\end{tikzpicture}
}
\newcommand{\thetagr}{
\begin{tikzpicture}[baseline=-.65ex]
\node[int] (v) at (0,0) {};
\node[int] (w) at (.7,0) {};
\draw (v) edge[bend left] (w) edge[bend right] (w)edge (w);
\end{tikzpicture}
}
\newcommand{\sSet}{\mathit{s}\mathcal{S}\mathit{et}}
\newcommand{\Set}{\mathcal{S}\mathit{et}}

\newcommand{\beq}[1]{
\begin{equation}\label{#1}
}
\newcommand{\eeq}{
\end{equation}
}

\usepackage{color}

\begin{document}
\title{Real models for the framed little $n$-disks operads}

\author{Anton Khoroshkin}
\address{	
Department of Mathematics, University of Haifa, Mount Carmel, 3498838, Haifa, Israel
}
\email{khoroshkin@gmail.com}

\author{Thomas Willwacher}
\address{Department of Mathematics \\ ETH Zurich \\  
R\"amistrasse 101 \\
8092 Zurich, Switzerland}
\email{thomas.willwacher@math.ethz.ch}

\thanks{T.W. acknowledges partial support by the Swiss National Science Foundation (grant 200021\_150012 and the SwissMap NCCR). This work has been partially funded by the European Research Council, ERC StG 678156--GRAPHCPX}


\begin{abstract}
We study the action of the orthogonal group on the little $n$-disks operads.
As an application we provide small models (over the reals) for the framed little $n$-disks operads.
It follows in particular that the framed little $n$-disks operads are formal (over the reals) for $n$ even and coformal for all $n$.
\end{abstract}

\maketitle

\setcounter{tocdepth}{2}
\tableofcontents

\section{Introduction}

The framed little $n$-disks operads $\flD_n$ are operads of embeddings of ``small'' $n$-dimensional disks in the $n$-dimensional unit disk. 
These operads are of fundamental importance in algebraic topology and homological algebra. In particular, in recent years they saw a surging interest due to applications in the manifold calculus of Goodwillie-Weiss \cite{G,GW}, and, relatedly, in the study of factorization algebras in homotopy theory \cite{AF}.

Surprisingly, the rational homotopy type of the operads $\flD_n$ is currently not understood very well.
This is in sharp contrast to the rational homotopy type of the non-framed sub-operads $\lD_n\subset \flD_n$, which is well understood due to work of Kontsevich \cite{K2}, Tamarkin (for $n=2$) \cite{Tam}, Lambrechts-Voli\'c \cite{LV}, Petersen \cite{Petersen} and Fresse-Willwacher \cite{FW} (see also \cite{BoavidaHorel} for recent results in positive characteristic).
Furthermore, it is known that the operad $\flD_2$ is rationally formal \cite{pavolfr, GS}.
The goal of this paper is to study the real homotopy type of the topological operads $\flD_n$ for $n\geq 3$.

To this end we will study the real homotopy type of the action of the orthogonal groups on the operads $\lD_n$, from which the real homotopy type of $\flD_n$ may be deduced.
It has been shown in \cite{FWAut} that the homotopy automorphisms of the rationalized operad $\lD_n^\Q$ are
\begin{equation}\label{equ:FW aut}
\Aut^h(\lD_n^\Q) \simeq \Q^\times \ltimes \Exp_\bullet(\GC_n^2),
\end{equation}
where $\GC_2^2$ is the Kontsevich graph complex, a dg Lie algebra to be discussed in detail below, and $\Exp_\bullet(-)$ its exponential simplicial group.

The action of the rationalization $\SO(n)^\Q$ on $\lD_n^\Q$ is determined by a morphism from $B\SO(n)^\Q$ to $B\Aut^h(\lD_n^\Q)$.
Since $\SO(n)$ is connected this map actually lands in the classifying space of the connected component of the identity, which is in turn contained in $\Exp_\bullet(\GC_n^2)$ since $\Q^\times$ is discrete.
Hence our map takes the form 
\[
 B\SO(n)^\Q \to B\Exp_\bullet(\GC_n^2) \cong \MC_\bullet(\GC_n^2),
\]
and any such morphism is equivalent to providing a Maurer-Cartan element 
\[
 m_\Q\in 
  \GC_n^2\hat \otimes H(B\SO(n);\Q)
\]
in the dg Lie algebra $\GC_n\hat \otimes H(B\SO(n);\Q)$.
The main result of this paper is a computation of the (gauge equivalence class) of the real version 
$m_\Q\hotimes_\Q \R$ of this Maurer-Cartan element.

To state the result precisely we need some more preparation.
First, there is an action of $\GC_n^2$ on an extended version $\Graphs_n^2$ of Kontsevich's dg Hopf cooperad model $\Graphs_n$ for $\lD_n$, and this action underlies the identification \eqref{equ:FW aut}. 
Any Maurer-Cartan element $m\in\GC_n^2\hat \otimes H(B\SO(n);\R)$ then 
defines a coaction of a dg Hopf coalgebra model $A$ for $\SO(n)$ on $\Graphs_n^2$. In our case the relevant Maurer-Cartan elements in fact live in a smaller dg Lie subalgebra $\GC_n^+\subset \GC_n^2$, which acts directly on Kontsevich's model $\Graphs_n$, so that the extension to $\Graphs_n^2$ is not necessary.

Furthermore, the arity-wise application of the geometric realization functor $\G: \dgca \to \sSet$ from differential graded commutative algebras to simplicial sets (see \eqref{equ:rht adj} below) takes dg commutative Hopf algebras to simplicial groups, dg Hopf cooperads to simplicial operads and coactions to action.
In particular, we obtain from $m$ a pair 
\[
  ( \G A,  \G \Graphs_n)_m,
\]
consisting of a simplicial group acting on a simplicial operad, and where we have marked the dependence of the action on $m$ in the notation.

Now define the following family of Maurer-Cartan elements 
\begin{equation}\label{equ:mn intro}
  m_n :=
  \begin{cases}
    -E \tadpole & \text{for $n$ even} \\
    \sum_{j\geq 1}
    \frac {\Pp_{2n-2}^{j}}{4^j}
   \frac{1}{2(2j+1)!} 
   \begin{tikzpicture}[baseline=-.65ex]
    \node[int] (v) at (0,.5) {};
    \node[int] (w) at (0,-0.5) {};
    \draw (v) edge[bend left=50] (w) edge[bend right=50] (w) edge[bend left=30] (w) edge[bend right=30] (w);
    \node at (2,0) {($2j+1$ edges)};
    \node at (0,0) {$\scriptstyle\cdots$};
   \end{tikzpicture}
   & \text{for $n$ odd}
  \end{cases}
  \in 
  \GC_n^+\hat \otimes H(B\SO(n);\R),
\end{equation}
where $E\in H(B\SO(n);\R)$ is the Euler class and $\Pp_{2n-2}\in H(B\SO(n);\R)$ is the top Pontryagin class.
Let us denote by $\sS(\SO(n), \lD_n)$ some simplicial model for the pair $(\SO(n), \lD_n)$.
Then our main result reads:

\begin{thm}\label{thm:main pair}
  For each $n\geq 3$ there is a zigzag of real homotopy equivalences 
  \[
   \sS(\SO(n), \lD_n) \dasharrow (\G A, \G\Graphs_n)_{m_n} 
  \]
  of pairs $(G,\POp)$ consisting of a simplicial group $G$ acting on a simplicial operad $\POp$, that connects the little disks operad with its natural $\SO(n)$-action, to the pair constructed above. 
\end{thm}
In other words, we have that 
\[
  (\SO(n)^\R, \lD_n^\R) \simeq (\G A, \G\Graphs_n)_{m_n}.
\]
Or yet alternatively, the result states that the pair $(A, \Graphs_n)_{m_n}$ consisting of the commutative Hopf algebra $A$ coacting on the dg Hopf cooperad $\Graphs_n$ via $m_n$ constitutes a real model in the sense of rational homotopy theory for the pair $(\SO(n), \lD_n)$.

We also note that the case $n=2$ of Theorem \ref{thm:main pair} also holds, but this is well-known by prior work.

There are several Corollaries to Theorem \ref{thm:main pair}.
First, the framed-operad construction 
\[
(G,\POp) \to \POp \rtimes G  
\]
is functorial in pairs consisting of simplicial group and an operad acted upon by the group, and preserves weak ($\R$-)homotopy equivalences.
Furthermore, one has that 
\[
(\G\Graphs_n) \rtimes_{m_n} (\G A) = \G (\Graphs_n \rtimes_{m_n} A)   
\]
where on the right-hand side we used a natural version of the framing construction $\rtimes$ defined for dg commutative Hopf algebras acting on dg Hopf cooperads, see section \ref{sec:framed operads} below, and on both sides we indicate the dependence on $m_n$ in the notation.
One then immediately obtains from Theorem \ref{thm:main pair}:

\begin{cor}\label{cor:flD model}
For $n\geq 3$ there is a zigzag of real homotopy equivalences of simplicial operads 
\[
\flD_n \dashrightarrow \G (\Graphs_n \rtimes_{m_n} A)
\]
connecting the framed little disks operad $\flD_n$ to the simplicial operad $\G (\Graphs_n \rtimes_{m_n} A)$.
\end{cor}
In other words, the dg Hopf cooperad $\Graphs_n \rtimes_{m_n} A$ is a real dg Hopf cooperad model for $\flD_2$ in the sense of rational homotopy theory.
Again, the statement of Corollary \ref{cor:flD model} also holds for $n=2$, but that case is well-known.

From the explicit model of Corollary \ref{cor:flD model} one can read of various further facts about the framed little disks operads.

\begin{cor}\label{cor:partial framed formality}
Let $n\geq 4$. The $\SO(n)$-framed little $n$-disks operads are formal over $\R$ if $n$ is even.
\end{cor}
For the precise meaning of formality we refer to section \ref{sec:models formality} below. The case $n=2$ is again well known and has been shown in \cite{pavolfr, GS}.

\begin{cor}\label{cor:odd nonformality}
 The operads of real chains of the $\SO(n)$-framed little $n$-disks operads are not formal for $n\geq 3$ odd. 
\end{cor}
While this work was under preparation, the case $n\geq 5$ has also been shown in \cite{Mo} by an explicit obstruction computation.

\begin{cor}\label{cor:FE3coformal}
The operads $\flD_n$ are coformal over $\R$ for all $n\geq 2$.
\end{cor}

The explicit minimal (Quillen) graded Lie algebra model of $\flD_n$ is constructed in section \ref{sec:quillen} below.

We will derive Theorem \ref{thm:main pair} from a more technical result.
We will see that to describe $\lD_n$ as an operad with $\SO(n)$-action, we can equivalently study the homotopy quotient $\lD_n\sslash\SO(n)$ as an operad in simplicial sets over $B\SO(n)$.
A real model for that object is given by the equivariant differential forms on $\lD_n$. These form a cooperad only up to homotopy. Thus we have to work with homotopy dg Hopf cooperads, and can formulate the following result. 
\begin{thm}\label{thm:main equiv}
There is a quasi-isomorphism of homotopy dg Hopf cooperads under $H(B\SO(n);\R)$
\[
  (\Graphs_n \hotimes H(B\SO(n);\R))^{m_n}
  \to 
  \Omega^{PA}_{\SO(n)}(\FM_n),
\]
where $\Omega^{PA}_{\SO(n)}(\FM_n)$ is a (quasi-isomorphic) version of the equivariant differential forms on the Fulton-MacPherson operad, introduced in section \ref{sec:equiv forms} below.
\end{thm}

We remark that all operads appearing in this paper are concentrated in arities $\geq 1$, i.e., we do not consider operations of arity zero.

\subsection*{Structure of the paper}
In section \ref{sec:basic notation} we introduce our notation and recall several well-known facts in operad theory and homotopy theory.
In section \ref{sec:rht pairs} we recall the version of homotopy operads we use and discuss their real homotopy theory. 
Section \ref{sec:k graphs} contains a discussion of the Kontsevich graph complex and the cooperad $\Graphs_n$, as well as the action of $\SO(n)$.
Section \ref{sec:main derivations} contains the proofs of Theorem \ref{thm:main pair} above and its corollaries, provided Theorem \ref{thm:main equiv}. Finally, the remainder of the paper is devoted to the proof of Theorem \ref{thm:main equiv}.
In section \ref{sec:graphs} we first show that the Theorem holds for some Maurer-Cartan element $Z^n_{\SO(n)}$ for which we provide an explicit configuration space integral formula.
Using an equivariant localization technique, we then show in sections \ref{sec:MC} and \ref{sec:auxthmproof} that this Maurer-Cartan element is gauge equivalent to the element $m_n$ of \eqref{equ:mn intro} above, and hence derive Theorem \ref{thm:main equiv}. 
More precisely, the argument of sections 8 and 9 is an induction, with the main new idea of the induction step being the following:
We may restrict the the group from $\SO(n)$ to $\SO(n-2)\times \SO(2)$ to obtain a Maurer-Cartan element $Z^n_{\SO(n-2)\times \SO(2)}$. After inverting the Euler class of $H(B\SO(2))$, $Z^n_{\SO(n-2)\times \SO(2)}$ essentially becomes gauge equivalent to the Maurer-Cartan element $Z^{n-2}_{\SO(n-2)}$ for dimension $n-2$, as is shown in section \ref{sec:auxthmproof}. This is reminiscent of the equivariant localization result that the fixed point inclusion 
\[
\lD_{n-2} \cong \lD_{n}^{\SO(2)} \hookrightarrow \lD_n
\]
induces an isomorphism on equivariant cohomology after inverting the orthogonal Euler class, see \cite[Theorem III.1]{Hsiang}. Surprisingly, the restriction of the group and the localization at the Euler class can be undone, as shown in section 8, and the gauge equivalence class of $Z^n_{\SO(n)}$ can thus be recovered from that of $Z^{n-2}_{\SO(n-2)}$.

\subsection*{Remark}
This paper is a thoroughly revised and extended version of our arXiv preprint \cite{KWarxiv} from 2017.
It repairs several shortcomings of the original preprint, that used an ad-hoc real homotopy theory for operads. Here we embed the statements of our preprint in a more standard homotopy theory of operads, to make them more useful for forthcoming works.
We note that this has partially also been done in the concluding remarks of Fresse's paper \cite[section 5]{Frextended}, where he develops  a rational homotopy theory of operads.
We remark on the relation to Fresse's work in sections \ref{sec:Fresse rht} and \ref{sec:models formality} below.
Unfortunately, due to significant changes, the notation of the current manuscript is not fully compatible with that of the original 2017 preprint.

We also remark that very recently a version of $O(n)$-equivariant formality for all $\lD_n$ were shown by Boavida et al. \cite{BoavidaCiriciHorel}. Their formality statements concern a representation of the $O(n)$-operad $\lD_n$ in a category of dendroidal objects, and are in particular not in contradiction to the non-formality statements for odd $n$ presented above. Rather, they should be viewed as related to formality of the homotopy quotient $\lD_n\sslash O(n)$.

\subsection*{Acknowledgements}
We are very grateful for discussions with Pedro Boavida de Brito, Joana Cirici, Benoit Fresse, Geoffroy Horel, Alexander Kupers and Victor Turchin.
Victor Turchin in particular contributed to parts of (the original version of) section \ref{sec:auxthmproof}.


\section{Recollections and notation}\label{sec:basic notation}

\subsection{Vector spaces, complexes, dgcas}\label{sec:vector spaces complexes}

We generally work over a ground field $\K$ of characteristic zero.
Our algebraic constructions work for $\K=\Q$. To show the main results we will however use transcendental methods (integrals) and eventually restrict to $\K=\R$.

As usual, we abbreviate the phrase \emph{differential graded} by dg. We mostly use cohomological conventions, so that our differentials have degree $+1$. (If needed, we convert homological degrees to cohomological by changing their sign.) 
We denote the category of differential $\mathbb Z$-graded vector spaces by $\dgVect$, and the subcategory consisting of such objects with non-negative grading by $\dgVect_{\geq 0}$. 
For a dg vector space $V$ we will denote its cohomology by $H(V)$.
For $X$ a topological space we will denote the cohomology with $\K$-coefficients by $H(X)$ or $H^\bullet(X)$, and the homology by $H_\bullet(X)$.
We shall denote by $\dgca$ the category of differential non-negatively graded commutative algebras.

\subsection{Model category structures on functor categories}

\begin{thm}[{Existence of projective model structure, \cite[Theorems 11.6.1 and 11.7.3]{Hirschhorn}, \cite[Corollary 1.54]{AdamekRosicky}, \cite[Remark A.2.8.4]{LurieHTT}}]
  \label{thm:proj exists}
  Let $\catC$ be a cofibrantly generated model category and let $\catD$ be a small category. Then the diagram category $\catC^{\catD}$ carries a cofibrantly generated model category structure such that 
  \begin{itemize}
    \item The weak equivalences (resp. fibrations) in $\catC^{\catD}$ are those natural transformations that are weak equivalences (resp. fibrations) in $\catC$ objectwise.
    \item The cofibrations in $\catC^{\catD}$ are the morphisms that have the left-lifting property with respect to the acyclic fibrations.
  \end{itemize}
  If furthermore $\catC$ is a combinatorial (resp. simplicial or left proper) model category  then $\catC^{\catD}$ is a combinatorial (resp. simplicial or left proper) model category as well.
\end{thm}

This model structure is called the \emph{projective} model structure on $\catC^{\catD}$. Dually, there is also the injective model structure, that is however more involved.
\begin{thm}[{Existence of injective model structure \cite[Propositions A.2.8.2 and A.3.3.2]{LurieHTT}}]
  \label{thm:inj exists}
  Let $\catC$ be a cofibrantly generated combinatorial model category and let $\catD$ be a small category. Then the diagram category $\catC^{\catD}$ carries a model category structure such that 
  \begin{itemize}
    \item The weak equivalences (resp. cofibrations) in $\catC^{\catD}$ are those natural transformations that are weak equivalences (resp. cofibrations) in $\catC$ objectwise.
    \item The fibrations in $\catC^{\catD}$ are the morphisms that have the right-lifting property with respect to the acyclic fibrations.
  \end{itemize}
\end{thm}

The constructions of these model category structure are functorial in $\catC$ in the following sence.
\begin{thm}[{\cite[Theorem 11.6.5]{Hirschhorn}, \cite[Remark A.2.8.6]{LurieHTT}}]
  \label{thm:proj functorial}
Let $\catB$, $\catC$ be cofibrantly generated model categories and let $\catD$ be a small category. Let 
\begin{equation}\label{equ:LR adj func}
L \colon \catB \rightleftarrows \catC \colon R  
\end{equation}
be a Quillen adjunction. Then post-composition induces an adjunction on diagram categories 
\begin{equation}\label{equ:LR adj func2}
  L \colon \catB^\catD \rightleftarrows \catC^\catD \colon R. 
\end{equation}
that is a Quillen adjunction with respect to either the projective or injective model category structures, if the latter exist.
If furthermore \eqref{equ:LR adj func} is a Quillen equivalence, so is \eqref{equ:LR adj func2}.
\end{thm}

\subsection{Rational (real) homotopy theory}
\label{sec:std rht}
Let $\sSet$ be the category of simplicial sets equipped with the standard Quillen model category structure. That is, the weak equivalences of $\sSet$ are the weak homotopy equivalences, the fibrations are the Kan fibrations and the cofibrations are the degree-wise injective morphisms of simplicial sets.

Let $\dgca$ be the category of differential non-negatively graded commutative algebras.
This category is equipped with a cofibrantly generated model structure obtained by transfer along the forgetful functor (see e.g. \cite[II.6.2]{F})
\[
\dgca \to \dgVect,  
\] 
or in other words:
\begin{itemize}
\item The weak equivalences of $\dgca$ are the quasi-isomorphisms.
\item The fibrations are the surjective maps.
\item The cofibrations are those morphisms that have the left-lifting property with respect to the acyclic fibrations. 
\end{itemize}

The central element of rational (or real) homotopy theory is the Quillen adjunction 
\begin{equation}\label{equ:rht adj} 
  \Omega \colon \sSet \rightleftarrows (\dgca)^{op} \colon \G
\end{equation}
with 
\begin{align*}
 \G A &= \Mor_{\dgca}(A, \Ompoly)
 &
 \Omega(X) &= \Mor_{\sSet}(X, \Ompoly)
\end{align*}
and $\Ompoly$ the simplicial dg commutative algebra of polynomial differential forms on the simplicies.
For $X$ a(n automatically cofibrant) simplicial set, we define its rationalization as
\[
X^\Q = \G^h(\Omega(X)) := \G(A),
\]
for $A$ a cofibrant replacement of $\Omega(X)$ in $\dgca$.


\begin{rem}\label{rem:G pullback}
We will use below that by adjunction the functor $\G$ sends pushouts in $\dgca$ to pullbacks in $\sSet$. This means that for $A\leftarrow B\to C$ a diagram of dg commutative algebras we have that 
\[
 \G(A\otimes_B C) = \G A \times_{\G B} \G C.   
\]
\end{rem}

\subsection{Rational homotopy theory of simplicial groups}\label{sec:sGrp rht}
Let $\sSet_0\subset \sSet$ be the full sub-category of reduced simplicial sets, that is, simplicial sets $X$ with $X_0=*$.
By \cite[Proposition V.6.2]{GoerssJardine} $\sSet_0$ is equipped with a model category structure such that 
\begin{itemize}
\item The weak equivalences (resp. cofibrations) are the morphisms that are weak equivalences (resp. cofibrations) in $\sSet$. (I.e., weak homotopy equivalences and monomorphisms respectively.)
\item The fibrations are the morphisms that have the right-lifting property with respect to the acyclic cofibrations.
\end{itemize}
This model category structure can be thought of as the left-transferred model structure along the adjunction
\begin{equation}\label{equ:iota E1 adj}
  \iota \colon \sSet_0 \rightleftarrows \sSet^{*/} \colon E_1,
\end{equation}
which is hence Quillen.
Here $\iota$ is the obvious inclusion and $E_1$ the Eilenberg subcomplex construction in degree 1. 
Then the following result is classical.
\begin{thm}
    [{see \cite[Proposition V.6.3]{GoerssJardine}}]
    \label{thm:sGrp equivalence}
The simplicial classifying space functor $\bar W$ and the simplicial loop space functor $\ell$ form a Quillen equivalence 
\begin{equation}\label{equ:ell bar W adj}
\ell \colon \sSet_0 \rightleftarrows \sGrp \colon \bar W   
\end{equation}
between the categories of simplicial sets and simplicial groups.
\end{thm}

Finally, one has the Quillen adjunction obtained by slicing \eqref{equ:rht adj},
\[
  \Omega \colon \sSet^{*/} \rightleftarrows (\dgca_{/\K})^{op}\colon \G
\]
with the right-hand category the category of augmented dg commutative algebras.

Now let $G\in \sGrp$ be a connected simplicial group, i.e., $\pi_0(G)=*$.
Then we define the simplicial group 
\[
G^\K \simeq  \ell^h \circ E_1^h \circ \G^h \circ \Omega^h \circ \iota^h \circ \bar W^h (G)
\] 
by composing the derived functors $(-)^h$ of the functors of the adjunctions above.
This expression through derived functors shows that $G^\K$ is well-defined up to weak equivalence. Furthermore, $G$ is connected to $G^\K$ be a zigzag of morphisms using the adjunction (co-)units.
Finally, there is a much simpler expression for $G^\K$: Suppose that $G$ is fibrant (a Kan complex), and let $A\xrightarrow{\sim} \Omega(\bar WG)$ be a cofibrant replacement in $\dgca$ such that the subspace of cohomological degree zero $A^0\subset A$ satisfies $A^0=\K$.
For example, we may take the Sullivan minimal model for $A$.
Then we have that 
\[
  G^\K \simeq \ell\circ \G(A).
\] 
To see this note that all objects of $\sSet^{*/}$ and $\sSet_0$ are cofibrant, hence we can replace the respective derived functors by the non-derived functors. Furthermore, from our connectivity condition on $A$ it follows that $\G A$ is already a reduced simplicial set, so that $E_1^h\circ \G A =\G A$. Then, the derived morphism $G\dashrightarrow G^\K$ is given by the zigzag
\[
G \xleftarrow{\sim} \ell(\bar W G) 
\to
\ell( \G (A) ).
\]
\begin{prop}\label{prop:G rht equiv}
For $G$ a connected simplicial group with degree-wise finite dimensional rational cohomology the morphism above is a $\K$-homotopy equivalence, in the sense that it induces an isomorphism on $\pi_0$ (i.e., $G^{\K}$ is connected) and isomorphisms 
\[
\pi_k(G) \otimes_{\mathbb{Z}} \K \cong 
\pi_k(G^{\K}).
\]
\end{prop}
\begin{proof}
Since $G$ is connected $\bar WG$ is simply connected.
Using \cite[Theorem II.7.3.5]{F} we then have that the derived adjunction counit 
\[
  \bar WG \to \G (A)
\]
is a $\K$-homotopy equivalence. Since $\ell$ just shifts the homotopy groups down in degree the map 
\[
  \ell(\bar W G) \to \ell( \G (A))
\]
is also a $\K$-homotopy equivalence.
Finally that map $G \xleftarrow{\sim} \ell(\bar W G) $ is a weak equivalence by Theorem \ref{thm:sGrp equivalence}.
\end{proof}

\subsection{\texorpdfstring{$G$}{G}-spaces}
Let $G$ be a simplicial group and $EG\to BG$ the universal principal $G$-bundle.
We consider the category $G\sSet$ of simplicial sets with a $G$-action. It can be equipped with a cofibrantly generated model category structure by model categorial transfer along the forgetful functor 
\[
  G\sSet \to \sSet.
\]
The weak equivalences (resp. fibrations) in $G\sSet$ are those $G$-equivariant morphisms of simplicial sets that are weak homotopy equivalences (resp. Kan fibrations) of simplicial sets.
The following result can be extracted from the literature.
\begin{prop}\label{prop:GsSet}
The model category structure on $G\sSet$ as above is well defined.
It furthermore has the following properties:
\begin{itemize}
\item $G\sSet$ is a left and right proper combinatorial and simplicial model category.
\item The forgetful functor $F:G\sSet \to \sSet$ is a left and right Quillen adjoint. It creates weak equivalences and preserves fibrations and cofibrations.
\end{itemize}
\end{prop}
\begin{proof}
  We refer to \cite{DDK} or \cite[Theorem V.2.3]{GoerssJardine} for the validity of the construction of the cofibrantly generated simplicial model category structure.
  For combinatoriality we then have to verify that $G\sSet$ is locally presentable -- this follows since any $G$-simplicial set is a union of its finitely generated sub-objects.

  The last statement of the Proposition is also well-known. It clear from the construction of the model category structure by transfer that $F$ is a right adjoint and creates weak equivalences.
  On the other hand one has an adjunction 
  \[
    F \colon G\sSet \rightleftarrows \sSet \colon \Map_{\sSet}(G, -).
  \]
  The right-adjoint preserves (acyclic) fibrations, which are just acyclic fibrations of simplicial sets. Hence $F$ is also left Quillen and the remaining statements follow. 
\end{proof}

The over-category $\sSet_{/\bar WG}$ of $\bar WG$ is naturally equipped with the slice model structure. That is, a morphism is a weak equivalence (resp. fibration, cofibration) if the underlying morphism of simplicial sets is a weak equivalence (resp. fibration, cofibration).
It is well-known that the two model categories are Quillen equivalent.
\begin{prop}[{Dror-Dwyer-Kan \cite[Proposition 2.3]{DDK}}]
\label{prop:DDK}
  Let $G$ be a simplicial group. Then there is a Quillen equivalence 
  \begin{equation}\label{equ:times WG adj}
  (-) \times_{\bar WG} WG \colon
  \sSet_{/\bar WG} \rightleftarrows 
  G\sSet
  \colon (-)\sslash G .
  \end{equation}
\end{prop}

\subsection{Dg Lie algebras, Maurer-Cartan elements and exponential group}
Let $\fg$ be a dg Lie algebra. Then a Maurer-Cartan element is a degree 1 element $\alpha\in \fg$ satisfying the Maurer-Cartan equation 
\[
d\alpha +\frac12 [\alpha,\alpha]=0.
\]
We write $\MC(\fg)$ for the set of Maurer-Cartan elements of $\fg$.
Next, suppose that $\fg$ is equipped with a descending complete filtration 
\[
\fg = \mF^1\fg \supset \mF^1\fg \supset \cdots
\]
that is compatible with the dg Lie structure in the sense that $d\mF^p\fg\subset \mF^p\fg$ and $[\mF^p\fg,\mF^q\fg]\subset \mF^{p+q}\fg$.
Then we define the Maurer-Cartan space of $\fg$ to be the simplicial set 
\[
\MC_\bullet(\fg) = \MC(\fg\hotimes \Ompoly),
\]
where we use the filtration to complete the tensor product.

Furthermore, we define the exponential group of $\fg$ to be the simplicial group
\[
\Exp_\bullet(\fg) := \{x \in (\fg\hotimes \Ompoly)^0 \mid dx=0\}
\]
consisting of the closed elements of degree zero.
The group multiplication is given by the Baker-Campbell-Hausdorff (BCH) formula 
\[
x*_{BCH} y = x+y+\frac12 [x,y] + \cdots.
\]
Note that the BCH series on the right-hand side converges due to our assumption that the filtration is complete. Note that as simplicial sets we have 
\[
  \Exp_\bullet(\fg) = \MC_\bullet(\fg[-1]),
\]
where we consider the degree shifted dg vector space $\fg[-1]$ as an abelian dg Lie algebra. 
For more details on the above constructions we refer the reader to \cite{BerglundMC}.

\subsection{Equivariant cohomology and the Cartan model}\label{sec:compactGrecollection}
Now suppose that $G$ is a compact connected Lie group, $T\subset G$ a maximal torus, $N=N(T)\subset G$ the normalizer of $T$ in $G$ and $W=N/T$ the Weyl group of $G$.
Denote by $EG\to BG$ a(ny) model of the classifying fibration.
For $X$ a $G$-space there are morphisms  
\[
    (EG \times X)/T/W = (EG \times X)/N
\to (EG \times X)/G =: X\sslash G. 
\]
It is known (see \cite[Proposition III.1.1]{Hsiang} or \cite[Proposition 1]{BrionEquivariant}) that the above maps induce isomorphisms on cohomology
\[
 H(X\sslash G) \cong H(X\sslash N) \cong H(X\sslash T)^W.
\] 

Next suppose that $X$ is a smooth manifold and $G$ acts smoothly on $X$. Then we may compute the equivariant cohomology with the Cartan model.
We refer to \cite{Meinrenken} for a comprehensive overview of this construction.
For us, it will be sufficient to consider the toric variant. Let $r$ be the rank of $G$ so that $T=(S^1)^{\times r}$. Let $e_1,\dots,e_r$ be the generating vector fields on $X$ of the $r$ circle actions. 
Then we define 
\[
\Omega_T^{Cartan}(X):=
\R[u_1,\dots,u_r]\otimes \Omega_{dR}(X)^T
\]
with $u_1,\dots,u_r$ formal variables of degree $+2$, and with the differential 
\begin{equation}\label{equ:dudef}
d_u = d - \sum_{j} u_j \iota_{e_j}.
\end{equation}
Here $\iota_{e_j}$ is the contraction operator with the vector field $e_j$ on differential forms on $X$. Also note that there is a natural action of the Weyl group $W$ on $\Omega_T^{Cartan}(X)$.
We then define 
\[
\Omega_G^{Cartan}(X):=
(\Omega_T^{Cartan}(X))^W
=
(\R[u_1,\dots,u_r]\otimes \Omega_{dR}(X))^N.
\]

It is known that (see \cite[Theorems 6.1 and 6.7]{Meinrenken}) 
\begin{align*}
H( \Omega_T^{Cartan}(X)) &\cong H(X\sslash T)
& &
\text{and}
&
H( \Omega_G^{Cartan}(X)) &\cong H(X\sslash G).
\end{align*}

\subsection{Models for compact connected Lie groups and \texorpdfstring{$\SO(n)$}{SO(n)}}
\label{sec:group models}
It is well known that for $G$ a compact connected Lie group the cohomology of $BG$ is a free graded commutative algebra in generators of even degree
$$
H(BG) \cong \K[X_1,\dots,X_r] 
$$
with $r$ the rank of $G$ and the degrees of the generators determined by the exponents of the Lie algebra of $G$.
Since $H(BG)$ is free it follows that $BG$ is formal.
Furthermore, freeness implies that $H(BG)$ is cofibrant in $\dgca$ and hence we have that 
\[
BG^\K \simeq \G(H(BG)).
\]
Consider an abelian graded Lie algebra 
$$
\fg = \mathrm{span}(x_1,\dots,x_r)
$$
in generators of (cohomological) degrees $|x_j|=1-|X_j|$.
Then we have that 
\[
  B(G^\K) \simeq (BG)^\K \simeq \G(H(BG)) = \MC_\bullet(\fg) \simeq B\Exp_\bullet(\fg),
\]
where for the last equality we used \cite[Theorem 5.2]{BerglundMC}.
Taking loop spaces we hence conclude that 
\[
  G^\K \simeq \Exp_\bullet(\fg),
\]
so that the exponential group of $\fg$ is a model for $G^\K$. 


Finally, we consider the Weil algebra, which is the Koszul complex of $H(G)$:
\[
W_G = (H(G)\otimes H(BG), D).
\]
The differential $D$ is defined on dgca generators as $Dx_i^*=X_i$, $DX_i=0$, denoting the generators of $H(G)=\K[x_1,^*,\dots,x_r^*]$ (the cohomology of $G$) by $x_i^*$.
It is clear that $W_G$ is an acyclic dgca. Furthermore, we have the obvious cofibration 
\[
  H(BG) \to W_G,
\]
and $W_G$ inherits an $H(G)$-comodule structure.
It follows by right adjointness of $\G$ that 
\begin{equation}\label{equ:uni fib model}
\G(W_G)\to \G(H(BG))
\end{equation}
is a Kan fibration. The space $\G(W_G)$ inherits the $G^\K$-action, which is principal, and by acyclicity of $W_G$ we have that $\G(W_G)$ is weakly contractible. Hence \eqref{equ:uni fib model} is a model for the classifying fibration $EG^\K\to BG^\K$. 

\begin{ex}
  The most relevant case for our present paper is $G=\SO(n)$.
  For $n=2k+1$ odd we have $r=k$, $W=S_r \wr \mathbb Z_2$ and $H(B\SO(n))=\R[\Pp_{4},\Pp_8,\dots, \Pp_{4k}]$ is a polynomial algebra generated by the Pontryagin classes $\Pp_{4j}$ in degree $4j$. To be explicit, we use the following conventions.
  Computing $H(BG)=H(*\sslash G)$ using the Cartan model we find that 
  \[
    H(B\SO(n)) = \R[u_1,\dots,u_r]^W,
  \]
  and we define the Pontryagin classes to be the elementary symmetric polynomials 
  \[
    \Pp_{4j} = \frac{1}{(2\pi)^{2j}}\sum_{1\leq i_1 < \cdots < i_j \leq r}
  u_{i_1}^2u_{i_2}^2\cdots u_{i_j}^2.
  \]
  
  The action of $\pi_0(O(n))=\Z_2$ on $H(B\SO(n))$ is trivial, so that in particular $H(B\SO(n))=H(BO(n))$.
  Also note that the rational homotopy groups of $\SO(n)$ are generated by classes $p_{4j-1}$ in (homological) degree $4j-1$, for $j=1,2,\dots, k$.
  
  Next consider the $\SO(n)$ for even $n=2k$. Then $r=k$ and the Weyl group is an index two subgroup 
  $$
  W=\{(\sigma, \tau_1,\dots,\tau_r) \mid \prod_i\sgn(\tau_i)=1\}\subset S_r\wr \mathbb Z_2
  $$
  of the wreath product. We have that
$$
H(B\SO(n))=\R[\Pp_{4},\Pp_8,\dots, \Pp_{4k-4}, E],
$$
where the Pontrygin classes $\Pp_{4j}$ are defined as before and $E=\frac{1}{(2\pi)^k} u_1\cdots u_k$ is the Euler class of degree $n$.
The action of $\pi_0(O(n))=\Z_2$ on $H(B\SO(n))$ is trivial on the $\Pp_{4j}$, but by sign on $E$. Hence $H(BO(n))=H(B\SO(n))^{\Z_2}=\K[\Pp_{4},\Pp_8,\dots, \Pp_{4k-4}, E^2]$.
  The rational homotopy groups of $\SO(n)$ are generated by classes $p_{4j-1}$ in homological degree $4j-1$, for $j=1,2,\dots, k-1$, and the Euler class $e$ in homological degree $n-1$.
  
  \end{ex}

  \begin{rem}
Note that we consider only connected groups $G$, excluding the case of $O(n)$. 
We could however re-introduce $O(n)$ by hand as follows.
The action of $\Z_2$ on $H(B\SO(n))$ (by acting trivially on the Pontryagin classes and by sign on the Euler class) descends to a $\Z_2$-action on the group $\SO(n)^\K$. We can then define
\[
  O(n)^\K := \Z_2\ltimes \SO(n)^\K.
\]
\end{rem}

\subsection{Framed operads and the framing product}\label{sec:framed operads}
Let $\op T$ be a topological operad with an action of a topological group (or monoid) $G$. Then one may build the corresponding framed operad $\op T\rtimes G$ defined such that the $\op T\rtimes G$-algebras are the algebras of $\op T$ in the category of $G$-spaces \cite[section 2]{SW}.
More explicitly, we have that 
\[
(\op T\rtimes G)(r)=(\op T\circ G)(r) = \op T(r) \times G^{\times r},
\]
with the natural composition structure defined using the $G$-action. Concretely, the composition is such that for $t\in \op T(r)$, $t'\in \op T(s)$, $g_1,\dots,g_r,g_1',\dots,g_s'\in G$, $j=1,\dots,r$ we have
\[
 (t,g_1,\dots, g_r) \circ_j (t',g_1',\dots, g_s')
:=
(t\circ_j (g\cdot t'), g_1,\dots,g_{j-1}, g_jg_1',\dots, g_j g_s', g_{j+1},\dots, g_r).
\]

We call the operation ``$\rtimes$'' which associates to an operad in $G$-spaces the corresponding framed operad the \emph{framing product}.
On the underlying symmetric sequences of spaces it is the same as the plethysm $\circ$.

Similarly, let $\op C$ be a Hopf cooperad with a Hopf coaction of the Hopf algebra $A$. 
Then we may build the corresponding framed Hopf cooperad $\op C\rtimes A$ such that 
\[
(\op C \rtimes A)(r) = (\op C \circ A)(r) = \op C(r) \otimes A^{\otimes r},
\]
with the cocomposition naturally defined using the Hopf coaction of $A$ on $\op C$.

\section{Real homotopy theory of group-operad-pairs}
\label{sec:rht pairs}
As mentioned in the introduction we need to work with homotopy operads and cooperads. There are several approaches to model homotopy operads, and two of them would be practical for our purposes. Either one can use dendroidal objects as developed by Moerdijk and Cisinski, or one can use algebraic (Lawvere) theories, as in the work of Bergner \cite{Bergner}, generalizing earlier work by Badzioch \cite{Badzioch}.
Here we follow the latter approach through algebraic theories, which is slightly simpler than but essentially equivalent to using dendroidal sets.
We focus solely on the multi-sorted algebraic theory of operads, although a significant part of the exposition would also work for general algebraic theories.

Finally, let us remark that equivariant homotopy theories for operads have also been developed in \cite{BergnerEquivariant} and \cite{BonventrePereira}, however their technical setups are not directly usable for us, unfortunately.

\subsection{Homotopy operads, following \texorpdfstring{\cite{Bergner}}{Bergner}}
Let $\cT$ be the following category:
\begin{itemize}
\item The objects of $\cT$ are symbols $C_{\underline k}$ for $\underline k=(k_0,k_1,\dots)$ a sequence of numbers $k_j\in \mathbb Z_{\geq 0}$ such that $k_r=0$ for $r\gg 0$.
We write $C_r$ for the object corresponding to the sequence $\underline k$ such that $k_r=1$, and $k_j=0$ for $j\neq r$.
We also denote 
\[
C_{\emptyset} := C_{(0,0,\dots)}.    
\]
\item Denote by $\FreeOp_{\underline k}$ the free symmetric operad in the category of sets with $k_r$ generators of arity $r$ for $r=1,2,\dots$. Then we set 
\[
  \Mor_{\cT}(C_{\underline j}, C_{\underline k})
  =
  \Mor_{\Set\Op}(\FreeOp_{\underline k}, \FreeOp_{\underline j}).
\]
\end{itemize}
In other words the opposite category of $\cT$ is the full subcategory of $\Set\Op$ on the finitely generated free operads.
Every object $C_{\underline k}$ of $\cT$ can be written as a product 
\[
  C_{\underline k} = \prod_{r\geq 0} C_{r}^{k_r}
\]
of the $C_{r}$.

One can check \cite{Bergner} that a simplicial operad is the same data as a product preserving functor 
\[
\cT\to \sSet.  
\]
Concretely, to a simplicial operad $\POp$ one associates the product preserving functor $\POp:\cT \to \sSet$ such that 
\[
\POp(C_{\underline k}) =
\prod_{r\geq 0} \POp(r)^{k_r}
= 
\Mor_{\Set\Op}(\FreeOp_{\underline k}, \POp).
\]
This motivates the following definition of a homotopy operad as a functor in $\sSet^\cT$ that preserves products up to weak equivalence.
\begin{defi}
A simplicial \emph{homotopy operad} is a functor $X:\cT\to \sSet$ such that the natural morphisms 
\[
X(C_{\underline k}) \to \prod_{r\geq 0} X(C_r)^{k_r}
\]
are weak homotopy equivalences,
\end{defi}

We denote the inclusion functor from simplicial operads to $\sSet^\cT$ by $N:\sOp \to \sSet^\cT$. This functor fits into an adjunction (see \cite[Proposition 5.9]{Bergner}) 
\begin{equation}\label{equ:tau N adj}
\tau \colon \sSet^\cT \rightleftarrows \sOp \colon N,
\end{equation}
thus realizing $\sOp$ as a reflective subcategory of $\sSet^\cT$.
The left-adjoint $\tau$ has the following explicit description.
For $X\in \sSet^\cT$ the simplicial operad $\tau(X)$ is generated by the simplicial sets $X(C_r)$, $r=0,1,2,\dots$ with the following relations:
\begin{itemize}
\item The $S_r$-action on the generators $X(C_r)$ agrees with the images under $X$ of the $S_r$-action on $C_r$ in $\cT$.
\item Let $f_j: C_r\times C_s\to C_{r+s-1}$ be the composition morphism in $\cT$.
Then we impose that the following diagram commutes
\[
\begin{tikzcd}
    X(C_r\times C_s) \ar{d} \ar{dr}{X(f_j)} &  \\
    X(C_r) \times X(C_s) \ar{r}{\circ_j} & X(C_{r+s-1}),
\end{tikzcd}
\]
where $\circ_j$ is the operadic composition in $\tau(X)$.
\item Let $\epsilon: C_\emptyset\to C_1$ be the unit morphism.  
Then we impose that the image of $X(\epsilon):X(C_\emptyset)\to X(C_1)$ is identified with the operadic unit in $\tau(X)$. 
\end{itemize}

We can equip $\sSet^\cT$ with the projective model structure, see Theorem \ref{thm:proj exists}. Then the adjunction \eqref{equ:tau N adj} is a Quillen adjunction since $N$ obviously preserves weak equivalences and fibrations.

\subsubsection{Bousfield localization and Quillen equivalence}
Consider the $\Set$-operad $\FreeOp_{\underline k}$ as a simplicial operad, using the natural inclusion $\Set \to \sSet$. Then consider the object $\tilde C_{\underline k}:=N(\FreeOp_{\underline k})$ of $\sSet^{\cT}$. 
Let $L$ be the set of morphisms of $\sSet^{\cT}$
\[
\tilde C_{\underline k} \to \prod_{r\geq 0} \tilde C_r^{k_r}
\]
Then one may define a second model structure on $\sSet^{\cT}$ by left Bousfield localization at $L$. We denote the resulting model category by $L\sSet^{\cT}$. Note that $\tau(L)$ consists of isomorphisms, and hence Bousfield-localizing $\sOp$ at $\tau(L)$ leaves its model category structure unchanged. But since Bousfield localization preserves Quillen adjunctions the Quillen adjunction \eqref{equ:tau N adj} lifts to a Quillen adjunction (see \cite[Proposition 5.11]{Bergner})
\begin{equation}\label{equ:tau N adj l}
    \tau \colon L\sSet^\cT \rightleftarrows \sOp \colon N.
\end{equation}

\begin{thm}[{Bergner \cite[Theorem 5.13]{Bergner}}]
The adjunction \eqref{equ:tau N adj l} is a Quillen equivalence.
\end{thm}

We will only need the following consequence:
\begin{cor}\label{cor:bergner}
Let $\POp$ be a fibrant simplicial operad. Then the derived adjunction counit 
\[
\tau^h(N\POp) \to \POp    
\]
is a weak equivalence of simplicial operads.
\end{cor}
We note that it is clear by reflectivity that the non-derived adjunction counit is an isomorphism. Also, note that the cofibrations in $\sSet^\cT$ and $L\sSet^\cT$ are the same, so that the corollary holds for either of the adjunctions \eqref{equ:tau N adj} or \eqref{equ:tau N adj l}.

\subsection{Five model categories of pairs}
\subsubsection{Groups and operads}
Let $\sGrp\sOp$ be the category of pairs $(G,\POp)$ consisting of a simplicial group $\POp$ and an operad in $G$-spaces $\POp\in G\Op$.
The morphisms $(G,\POp)\to (H,\QOp)$ between two such pairs are pairs $(\phi,f)$ consisting of a morphism of simplicial groups $\phi:G\to H$ and a morphism $f:\POp\to \QOp=\phi^*\QOp$ of $G$-operads, equipping $\QOp$ with a $G$-action through $\phi$. Via the induction/restriction adjunction the datum $f$ is equivalent to a morphism of $H$-operads $f^\sharp: \phi_! \POp\to \QOp$.
\begin{prop}\label{prop:sGrpsOp model cat}
The category $\sGrp\sOp$ is equipped with a cofibrantly generated model structure with the following distinguished classes of morphisms:
\begin{itemize}
    \item The weak equivalences (resp. fibrations) are the morphisms that are object- and aritywise weak equivalences (resp. fibrations) of simplicial sets.
    \item The cofibrations are the morphisms that have the left-lifting property with respect to the acyclic fibrations.
\end{itemize}
\end{prop}
\begin{proof}
The stated model category structure is the one obtained by transfer along the free/forgetful adjunction 
\[
  \sSet \times \prod_{r\geq 0} \sSet \rightleftarrows  \sGrp\sOp.
\]
To see that the transfer yields a well-defined model structure we extend the analogous constructions of \cite[section II.8.2]{F} or \cite{BMAxiomatic} of model category structures on $\sOp$. The forgetful functor preserves filtered colimits and hence it suffices to show that $\sGrp\sOp$ has a fibrant replacement functor 
and functorial path objects, see \cite[sections 2.5, 2.6]{BMAxiomatic}.
For the fibrant replacement functor we may take the object- and aritywise application of a symmetric monoidal fibrant replacement functor for $\sSet$.
Also note that $\sGrp\sOp$ is naturally powered over $\sSet$, and hence we may take the powering with an interval $(-)^I$ as our functorial path object.
\end{proof}

\subsubsection{Groups and homotopy operads}
Similarly, let $\sGrp\sSet^\cT$ be the category of pairs $(G,X)$, with a $G$ a simplicial group and $X\in G\sSet^{\cT}$.
Morphisms $(G,X)\to (H,Y)$ are then pairs $(\phi,f)$ consisting of a morphism of simplicial groups $\phi:G\to H$ and a morphism $f:X\to Y=\phi^*Y$ in $G\sSet^{\cT}$, or equivalently a morphism $f^\sharp: \phi_! \POp\to \QOp$ in $H\sSet^{\cT}$.

\begin{prop}\label{prop:sGrpsSetT model cat}
    The category $\sGrp\sSet^\cT$ is equipped with a cofibrantly generated model structure with the following distinguished classes of morphisms:
    \begin{itemize}
        \item The weak equivalences (resp. fibrations) are the morphisms that are objectwise weak equivalences (resp. fibrations) of simplicial sets.
        \item The cofibrations are the morphisms that have the left-lifting property with respect to the acyclic fibrations.
    \end{itemize}
\end{prop}
\begin{proof}
The proof is analogous to that of Proposition \ref{prop:sGrpsOp model cat}, except that we transfer the model category structure along the free/forgetful adjunction 
\[
\sSet \times \prod_{\underline k} \sSet \rightleftarrows  \sGrp\sSet^\cT.
\]
\end{proof}

We have an adjunction 
\begin{equation}\label{equ:id tau adj}
  (id, \tau) \colon \sGrp\sSet^\cT \rightleftarrows \sGrp\Op \colon (id, N)
\end{equation}
by applying the adjunction $(\tau, N)$ of \eqref{equ:tau N adj} to the second member of the pairs.
The adjunction \eqref{equ:id tau adj} is evidently a Quillen adjunction since the right-adjoint preserves weak equivalences and fibrations, by definition of the model structures.

\begin{prop}\label{prop:adj we}
The derived adjunction counit of the Quillen adjunction \eqref{equ:id tau adj} is a weak equivalence.
\end{prop}
For the proof we need:
\begin{lemma}\label{lem:pre prop adj we}
Let $(G,X)$ be a cofibrant object of $\sGrp\sSet^\cT$.
Then the object $X$ is cofibrant in $\sSet^\cT$.
\end{lemma}
\begin{proof}
Consider the over-category $\sGrp\sSet^\cT_{/(G, *)}$ with the slice model structure.
The canonical morphism $(G,X)\to (G,*)$ is cofibrant in this category by definition of the  slice model structure.
Next, consider the category $G\sSet^{\cT}$ that we equip with the projective model structure, using the model structure on $G\sSet$ of Proposition \ref{prop:GsSet}.
We have an adjunction
\[
   L \colon  \sGrp\sSet^\cT_{/(G, *)} \rightleftarrows G\sSet^{\cT} \colon R,
\] 
such that the left-adjoint sends a morphism $(H,Y)\to (G,*)$ to the induced object $\Ind_H^G Y$ and the right-adjoint sends an object $X$ to the pair $(G,X)$.
It is clear that $(L,R)$ is a Quillen adjunction since the right-adjoint preserves weak equivalences and fibrations.
But then $L$ preserves cofibrant objects so that $X$ is cofibrant in $G\sSet^{\cT}$.
Next we have a Quillen adjunction 
\[
  F \colon  G\sSet \rightleftarrows \sSet \colon \Map_{\sSet}(G,-)
\]
with $F$ the forgetful functor.
Hence by Theorem \ref{thm:proj functorial} we obtain a Quillen adjunction between the functor categories
\[
  F \colon  G\sSet^\cT \rightleftarrows \sSet^\cT \colon \Map_{\sSet}(G,-).
\]
This means that $X$ is cofibrant as an object of $\sSet^\cT$.
\end{proof}

\begin{proof}[Proof of Proposition \ref{prop:adj we}]
Let $(G,\POp)$ be a fibrant object of $\sGrp\sOp$.
Let $(H, X)$ be a cofibrant replacement of $(G,N(\POp))$ in $\sGrp\sSet^\cT$.
Then by Lemma \ref{lem:pre prop adj we} we have that $X$ is also a cofibrant replacement of $N(\POp)$ in $\sSet^\cT$.
Hence by Corollary \ref{cor:bergner} the derived adjunction counit 
\[
  \tau(X)\to \POp  
\] 
is a weak equivalence.
\end{proof}

\subsubsection{Reduced simplicial sets and homotopy operads}
Next consider the category $\sSet_0\sSet_{/}^\cT$ whose objects are pairs $(X,Y)$ with $X\in \sSet_0$ a reduced simplicial set and $Y\in \sSet_{/X}^\cT$ an $\cT$-diagram in the overcategory.
Morphisms $(\phi, f): (X,Y)\to (Z,W)$ are pairs consisting of a morphism $\phi:X\to Z$ in $\sSet_0$ and a morphism $f:\phi_*Y\to W$ in $\sSet_{/Z}^\cT$. Note that by adjunction the morphism $f$ is equivalent data to a morphism $f^\sharp:Y\to \phi^*W$.

\begin{prop}\label{prop:sSet0sSetT model cat}
    The category $\sSet_0\sSet_{/}^\cT$ is equipped with a cofibrantly generated model structure with the following distinguished classes of morphisms:
    \begin{itemize}
        \item The weak equivalences (resp. fibrations) are the morphisms that are objectwise weak equivalences (resp. fibrations) in $\sSet_0$ or $\sSet$.
        \item The cofibrations are the morphisms that have the left-lifting property with respect to the acyclic fibrations.
    \end{itemize}
\end{prop}
\begin{proof}
We transfer the model structure along the adjunction 
\[
\sSet_0 \times \prod_{\underline k} \sSet \rightleftarrows  \sSet_0\sSet_{/}^\cT.
\]
Here the right-adjoint preserves all colimits. Hence by \cite[Theorem 7.4.4 and adjacent remark]{HeutsMoerdijk} it suffices to check that the images under the adjunction unit of the generating cofibrations on the left-hand side are again acyclic cofibrations. But this is clear.
\end{proof}

The adjunctions \eqref{equ:ell bar W adj} and \eqref{equ:times WG adj} combine into an adjunction of the category of pairs 
\begin{equation}\label{equ:ell times adj}
    (\ell, \times) \colon \sSet_0\sSet_{/}^\cT \rightleftarrows \sGrp\sSet^\cT \colon (\bar W, \sslash ).
\end{equation}
Here the left-adjoint sends a pair $(G, X)$ of a simplicial group $G$ and an object $X$ of $G\sSet^{\cT}$ to the pair $(\bar WG, X\sslash G)$.
The right-adjoint sends a pair $(B, Y)$ to the pair $(\ell B, Y\times_B B_\eta)$, with $B_\eta$ the principal $\ell B$-fibration over $B$, defined as the pullback via the adjunction unit $\eta:\mathit{id} \Rightarrow \bar W\circ \ell$ 
\[
\begin{tikzcd}
B_\eta \ar{d}\ar{r} & W\ell B \ar{d} \\
B \ar{r}{\eta} & \bar W\ell B
\end{tikzcd}.
\]   

\begin{prop}\label{prop:pair sset sset equiv}
The pair of functors \eqref{equ:ell times adj} is a Quillen equivalence.
\end{prop}
\begin{proof}
The right adjoint of \eqref{equ:ell times adj} is formed by applying the right adjoints of the Quillen pairs \eqref{equ:ell bar W adj} and \eqref{equ:times WG adj} objectwise.
Hence it also preserves (acyclic) fibrations and \eqref{equ:ell times adj} is a Quillen adjunction.

To check that \eqref{equ:ell times adj} is a Quillen equivalence we check that the derived unit and counit are weak equivlaences.
Let $(B, X)$ be a fibrant and cofibrant object of $\sSet_0\sSet_{/}^\cT$.
Then $(\ell B, X \times_B B_\eta)$ is again fibrant in $\sGrp\sSet^\cT$, since any object of $\sGrp$ is fibrant and pullbacks along fibrations preserve fibrant objects. (Here we use that $B$ is also fibrant in $\sSet$ by \cite[Corollary V.6.8]{GoerssJardine}.)
Hence the derived adjunction agrees with the non-derived adjunction unit and we need to check that
\begin{equation}\label{equ:prop ps proof 1}
(B, X) \to (\bar W\ell B, (X \times_B B_\eta)\sslash \ell B) 
\end{equation}
is a weak equivalence. But since $X \times_B B_\eta=X\times_{\bar W\ell B} W\ell B$ the morphism is obtained by applying objectwise the (derived) adjunction units of \eqref{equ:ell bar W adj} and \eqref{equ:times WG adj}. Hence, since \eqref{equ:ell bar W adj} and \eqref{equ:times WG adj} are Quillen equivalences, \eqref{equ:prop ps proof 1} is  a weak equivalence.

Finally, let $(G,Y)$ be a fibrant object of $\sGrp\sSet^\cT$. Since the left adjoint of \eqref{equ:ell times adj} preserves all weak equivalences, it is sufficient to check that the plain adjunction counit
\begin{equation}\label{equ:prop ps proof 2}
    (\ell \bar W G, (Y\sslash G) \times_{\bar WG} (\bar WG)_\eta ) \to  (G, Y) 
\end{equation}
is a weak equivalence. The morphism $\ell \bar W G\to G$ is a weak equivalence since \eqref{equ:ell bar W adj} is a Quillen equivalence.
Denoting $H=\ell\bar WG$, the second part of \eqref{equ:prop ps proof 2} factorizes as 
\[
    (Y\sslash G) \times_{\bar WG} (\bar WG)_\eta
    \to 
    (Y\sslash H) \times_{\bar WG} (\bar WG)_\eta
    =
    (Y\sslash H) \times_{\bar W H}
    W H
    \to 
    Y.
\]
Here the left-hand arrow is a weak equivalence since so is $Y\sslash G\to Y\sslash H$, and the 
right-hand arrow is a weak equivalence since \eqref{equ:times WG adj} is a Quillen equivalence.
Hence \eqref{equ:prop ps proof 2} is a weak equivalence.
\end{proof}

\subsubsection{Pointed simplicial sets and homotopy operads}
Let $\sSet_*\sSet_{/}^\cT$ be the category of pairs $(B, X)$, with $B$ a pointed simplicial set and $X\in \sSet_{/B}^\cT$.
Morphisms of $\sSet_*\sSet_{/}^\cT$ are pairs of maps $(\phi, f);(B, X)\to (C, Y)$ with $\phi:B\to C$ a morphism of pointed simplicial sets and $f:\phi_* X\to Y$ a morphism in $\sSet_{/C}^\cT$.

\begin{prop}\label{prop:sSetst sSet model cat}
    The category $\sSet_*\sSet_{/}^\cT$ is equipped with a cofibrantly generated model structure with the following distinguished classes of morphisms:
    \begin{itemize}
        \item The weak equivalences (resp. fibrations) are the morphisms that are objectwise weak equivalences (resp. fibrations) of simplicial sets.
        \item The cofibrations are the morphisms that have the left-lifting property with respect to the acyclic fibrations.
    \end{itemize}
\end{prop}
\begin{proof}
    We construct the model structure in 2 steps. First consider the category $\sSet\sSet_{/}^\cT$ of pairs $(B, X)$, with $B$ a plain (non-pointed) simplicial set and $X\in \sSet_{/B}^\cT$.
    This category is a diagram category, 
    \[
    \sSet\sSet_{/}^\cT \cong \sSet^{\cT^+},
    \]
    where $\cT^+$ is a category obtained by adding one object to $\cT$, and a unique arrow from any other object.
    As a diagram category, $\sSet\sSet_{/}^\cT$ is equipped with the projective model structure by Theorem \ref{thm:proj exists}. Finally, 
    \[
        \sSet_*\sSet_{/}^\cT = (\sSet\sSet_{/}^\cT)^{(*,\emptyset)/} 
    \]
    can be seen as an undercategory and hence is equipped with the coslice model structure.
\end{proof}

There is an obvious adjunction 
\begin{equation}\label{equ:iota id adj}
    (\iota, id) \colon \sSet_0\sSet_{/}^\cT \rightleftarrows \sSet_*\sSet_{/}^\cT \colon (E_1, \epsilon_*).
\end{equation}
by applying the adjunction \eqref{equ:iota E1 adj} to the first component $B$ of the pairs, and base-changing the second component accordingly.
The right adjoint preserves (acyclic) fibrations, since both components are right Quillen functors. Hence the adjunction \eqref{equ:iota id adj} is Quillen.

\subsubsection{Dg commutative algebras}
Finally, we consider the category $\dgca_*(\dgca^{/})^{\cT^{op}}$ of pairs $(A, B)$, where $A\in \dgca_*$ is an augmented dg commutative algebra and $B\in (\dgca^{A/})^{\cT^{op}}$ is a functor 
\[
  B:  \cT^{op} \to \dgca^{A/}
\]
into the under-category of $A$, considered as a plain (non-augmented) dg commutative algebra.
\begin{prop}\label{prop:dgcast dgca model cat}
    The category $\dgca_*(\dgca^{/})^{\cT^{op}}$ is equipped with a model category structure with the following distinguished classes of morphisms:
    \begin{itemize}
        \item The weak equivalences (resp. cofibrations) are the morphisms that are objectwise weak equivalences (resp. cofibrations) of dg commutative algebras.
        \item The fibrations are the morphisms that have the right-lifting property with respect to the acyclic cofibrations.
    \end{itemize}
\end{prop}
\begin{proof}
We proceed analogously, yet dually to the proof of Proposition \ref{prop:sSetst sSet model cat}.
First, let $\dgca(\dgca^{/})^{\cT^{op}}$ be defined as the category of pairs $(A, B)$, where $A\in \dgca$ is a plain dg commutative algebra and $B\in (\dgca^{A/})^{\cT^{op}}$,
Then we have that 
\[
    \dgca(\dgca^{/})^{\cT^{op}} = \dgca^{(\cT^+)^{op}}  
\]
is a diagram category, with $\cT^+$ defined as in the proof of Proposition \ref{prop:sSetst sSet model cat}.
We may hence use Theorem \ref{thm:inj exists} to equip $\dgca^{(\cT^+)^{op}}$ with the injective model structure.
Then 
\[
    \dgca_*(\dgca^{/})^{\cT^{op}} = \left(\dgca^{(\cT^+)^{op}} \right)_{/(\Q, 0)} 
\]
is naturally an over-category and hence equipped with the slice model structure, which has the distinguished classes of morphisms stated in the proposition.
\end{proof}

\subsection{Rational homotopy theory of pairs}
The categories of pairs above and the Quillen adjunctions between them are summarized in the following diagram, with the left-adjoints marked by dashed arrows.
\begin{equation}\label{equ:big adj}
\begin{tikzcd}
\sGrp\sOp \ar[shift left]{d}{N} \\
\sGrp\sSet^{\cT} \ar[shift left]{d}{(\bar W,\sslash )}\ar[shift left, dashed]{u}{\tau} \\
\sSet_0\sSet_{/} \ar[shift left, dashed]{d}{\iota}\ar[shift left, dashed]{u}{(\ell,\times)} \\
\sSet_*\sSet_{/} \ar[shift left, dashed]{d}{\Omega}\ar[shift left]{u}{E_1} \\
(\dgca_*(\dgca^{/})^{\cT^{op}})^{op}\ar[shift left]{u}{\G}
\end{tikzcd}
\end{equation}

Now let $(G,\POp)\in \sGrp\sOp$ be a fibrant pair consisting of a connected simplicial group $G$ and a simplicial operad $\POp$ acted upon by $G$.
Then we define the object $(G^\K,\POp^\K)\in \sGrp\sOp$ (the rationalization for $\K=\Q$) by passing through the derived functors of all the adjunctions of \eqref{equ:big adj}.
Concretely, we pick:
\begin{itemize}
\item A cofibrant replacement $(B,X)\to (\bar WG, N(\POp)\sslash G)$ in $\sSet_0\sSet_{/}^\cT$.
\item A cofibrant replacment $(A,Y) \to (\Omega(B), \Omega(X))$ in $\dgca_*(\dgca^{/})^{\cT^{op}}$.
\item A cofibrant replacement $(C, Z)\twoheadrightarrow (E_1\G A, \G Y)$ in $\sSet_0\sSet_{/}^\cT$.
\end{itemize}
Then we define 
\begin{align*}
    (G^\K, \POp^\K):= (\ell C, \tau(Z\times_{C} C_\eta ) ).
\end{align*}
From the adjunction (co)units we may obtain a zigzag connecting $(G,\POp)$ with $(G^\K,\POp^\K)$ in $\sGrp\sOp$.
To this end, fix a solution of the following lifting problem:
\[
     \begin{tikzcd}
        \emptyset \ar[hookrightarrow]{d} \ar{rr} & & (C, Z) \ar[twoheadrightarrow]{d}{\sim} \\
        (B,X) \ar{r} \ar[dashed]{urr} 
        & (E_1\G \Omega B, \G\Omega X) \ar{r}
        & (E_1\G A, \G Y)
     \end{tikzcd}.
\]
Then we have the zigzag 
\begin{equation}\label{equ:rht pair zigzag}
(G,\POp) 
\xleftarrow{\sim} (\ell B, \tau(X\times_B B_\eta))
\xrightarrow{r} (\ell C, \tau(Z\times_C C_\eta))
= (G^\K, \POp^\K),
\end{equation}
where the left-hand arrow is a weak equivalence by Propositions \ref{prop:adj we} and \ref{prop:pair sset sset equiv}.

In our specific setting, the above definition of $(G^\K, \POp^\K)$ can furthermore be slightly simplified.
Namely, the functors $\iota$ and $\Omega$ preserve all weak equivalences. Hence we have that 
\[
(\Omega B,\Omega X) \simeq (\Omega \bar W G, \Omega (N\POp \sslash  G)).
\]
Furthermore, as discussed in section \ref{sec:sGrp rht}, we may pick the cofibrant dgca $A$ above such that $A^0=\K$. Then $\G A$ is automatically a reduced simplicial set and the application of $E_1$ may be omitted.
Finally, the functor $(\ell,\times)$ also preserves all weak equivalences. 
Let us summarize this in the following Lemma.
\begin{lemma}\label{lem:rat simpl}
Suppose that $(A, Y)$ is a cofibrant object of $\dgca_*(\dgca^{/})^{\cT^{op}}$ with $A^0=\K$, weakly equivalent to $(\Omega\bar WG, \Omega (N\POp\sslash G))$.
Then if $(G',Z)$ is a cofibrant replacement of $(\ell\G A, \G Y\times_{\G A} (\G A)_\eta)$ in $\Grp\sSet_{/}^\cT$, we have that $(G', \tau Z)$ is weakly equivalent to $(G^\K, \POp^\K)$ in $\sGrp\sOp$.\hfill\qed
\end{lemma}

\subsection{Fresse's rational homotopy theory of operads}
\label{sec:Fresse rht}
We note that B. Fresse has developed a rational homotopy theory of (non-equivariant) operads \cite{Frextended, F} that we briefly sketch here. First, the right adjoint $\G$ of the rational homotopy adjunction \eqref{equ:rht adj} is symmetric monoidal. Hence the arity-wise application yields a functor $\G: \dgHOpc\to \sOp$ from the category of dg Hopf cooperads to the category of simplicial operads. 
\begin{thm}[{Fresse \cite[Theorems 1.6 and 1.7]{Frextended}}]
There is a cofibrantly generated model category structure on the category of dg Hopf cooperads $\dgHOpc$ with the following properties.
\begin{itemize}
\item The weak equivalences are the morphisms that are arity-wise quasi-isomorphisms. 
\item Let $f:\COp\to \DOp$ be a cofibration in $\dgHOpc$. Then the part of arity $r$, $\COp(r)\to \DOp(r)$ is a cofibration in the category $\dgca$ of dg commutative algebras for each $r=1,2,\dots$. 
\end{itemize}
\end{thm}

Fresse shows that the adjunction \eqref{equ:rht adj} extends to a Quillen adjunction 
\[
\Omega_\sharp \colon \sOp \rightleftarrows (\dgHOpc)^{op} \colon \G.
\]
Here the functor $\Omega_\sharp$ is defined as the left-adjoint of $\G$. It can be considered as an operadic upgrade of the functor $\Omega$ of \eqref{equ:rht adj}.

\begin{thm}[{Fresse \cite[Theorem 2.3]{Frextended}}]\label{thm:fresse comparison}
Let $\POp$ be a cofibrant simplicial operad such that $\POp(1)$ is connected and each $\POp(r)$ has finite dimensional $\K$-homology in each degree. Then the natural comparison morphism 
\[
\Omega_\sharp(\POp)(r) \to \Omega(\POp(r))   
\]
is a weak equivalence for each $r$.
\end{thm}

As a corollary we obtain:
\begin{cor}\label{cor:omega sharp comparison}
Let $\POp$ be a cofibrant simplicial operad. Let $\COp$ be a cofibrant dg Hopf cooperad such that $H^0(\COp(r))=\K$, $H^1(\COp(r))=0$ and $H^k(\COp(r))$ is finite dimensional for each $r$ and $k$.
Suppose that $\G(\COp)$ is connected to $\POp$ by a zigzag of rational homotopy equivalences 
\begin{equation}\label{equ:P C zigzag 0}
\POp \to \bullet \leftarrow \G(\COp).  
\end{equation}
Then $\Omega_\sharp(\POp)$ is connected to $\COp$ by a zigzag of quasi-isomorphisms of dg Hopf cooperads
\[
  \Omega_\sharp(\POp) \to \bullet \leftarrow \COp.  
\]
\end{cor}
\begin{proof}
  Cofibrancy of $\COp$ implies that each $\COp(r)$ is cofibrant.
Also note that the further assumptions on $\COp$ imply that each dg commutative algebra $\COp(r)$ is quasi-isomorphic to a minimal Sullivan model $(S(V),d)$ with the generating graded vector space $V$ degree-wise finite dimensional. In particular, $\COp(r)$ is weakly equivalent to a nilpotent cell complex of finite type in the sense of \cite[section II.7.3]{F}. The statement \cite[Proposition II.7.3.9]{F} then implies that the derived adjunction counit 
\[
  \Omega^h(\G(\COp(r)))\to \COp(r)
\]
is a weak equivalence.
This implies that a(ny) cofibrant resolution $\QOp\xrightarrow{\sim} \G(\COp)$ satisfies the assumptions of Theorem \ref{thm:fresse comparison}. Hence, using Theorem \ref{thm:fresse comparison} we also conclude that the derived adjunction counit 
\[
  \Omega_\sharp(\QOp)\to \COp
\]
is a weak equivalence.

Now consider the zigzag \eqref{equ:P C zigzag 0}. By passing to appropriate resolutions we may assume that the intermediate simplicial objects in our zigzag are all cofibrant, to obtain a zigzag of rational weak equivalences 
\begin{equation}\label{equ:P C zigzag}
  \POp \to \bullet \leftarrow \QOp \to \G(\COp),
\end{equation}
with $\QOp\xrightarrow{\sim} \G(\COp)$ again a cofibrant replacement of $\G(\COp)$.
Note that the assumptions of Theorem \eqref{thm:fresse comparison} are then satisfied for all operads in the zigzag except possibly $\G(\COp)$. Apply $\Omega_\sharp(-)$ to the zigzag \eqref{equ:P C zigzag} and compose with the adjunction counit to obtain a zigzag 
\[
  \Omega_\sharp(\POp)\to \bullet \leftarrow \Omega_\sharp(\QOp) \to \Omega_\sharp(\G(\COp)) \to \COp.
\]
By Theorem \eqref{thm:fresse comparison} and the assumption all the morphisms connecting $\Omega_\sharp(\POp)$ and $\Omega_\sharp(\QOp)$ are weak equivalences (quasi-isomorphisms) of dg Hopf cooperads, and by our initial considerations above so is the derived adjunction counit $\Omega_\sharp(\QOp)\to \COp$. Hence the desired zigzag of weak equivalences of dg Hopf cooperads is 
\[
  \Omega_\sharp(\POp)\to \bullet \leftarrow \Omega_\sharp(\QOp) \to \COp.
\]
\end{proof}

\section{Kontsevich's graph complex and graphs operad}
\label{sec:k graphs}
\subsection{Definition}\label{sec:GraDefinitions}
We recall here the definition of M. Kontsevich's graph complexes and graph cooperads. The original definitions may be found in \cite{K2, Knoncomm}.

For every $r\geq 1$ denote by 
\[
\gra_n(r) = \K[\omega_{ij}=(-1)^{n}\omega_{ji}\mid 1\leq i, j \leq n]    
\]
the free graded commutative algebra generated by symbols $\omega_{ij}=(-1)^n\omega_{ji}$ of cohomological degree $n-1$.
Every monomial in the $\omega_{ij}$ may be thought of as a graph on $r$ vertices, with each $\omega_{ij}$ representing an edge between vertices $i$ and $j$. 
These dg commutative algebras $\gra_n(r)$ carry an obvious $S_r$-action by permuting the (indices of the) generators. They furthermore assemble into a Hopf cooperad $\gra_n$.
The reduced cooperadic cocomposition
\[
\Delta_k : \gra_n(r) \to \gra_n(r-k+1)\otimes \gra_n(k) 
\]
is defined on the algebra generators $\omega_{ij}$ as 
\[
\Delta_k(\omega_{ij})=
\begin{cases}
  1\otimes \omega_{ij} + \omega_{11}\otimes 1 & \text{if $i,j\leq k$} \\
  \omega_{1(j-k+1)}\otimes 1 & \text{if $i\leq k$, $j> k$} \\
  \omega_{(i-k+1)(j-k+1)}\otimes 1 & \text{if $i,j> k$}
\end{cases}.
\]
Graphically, this corresponds to cutting out a subgraph on the vertices $1,\dots,k$.

We next define
\[
\fGraphs_n(r) := 
\bigoplus_{k\geq 0}  \left(\gra_n(r+k)\otimes \K[n]^{\otimes k}  \right)_{S_k}, 
\]
where the symmetric group acts diagonally on $\gra_n(r+k)$ and on $\K[n]^{\otimes k}$ by permuting factors, with Koszul signs.
Elements of the $k$-th summand of $\fGraphs_n(r)$ may be thought of as linear combinations of graphs with $r$ numbered "external" vertices and $k$ unnumbered "internal" vertices.
\[
\begin{tikzpicture}
\node[ext] (v1) at (0,0) {$\scriptstyle 1$};
\node[ext] (v2) at (0.7,0) {$\scriptstyle 2$};
\node[ext] (v3) at (1.4,0) {$\scriptstyle 3$};
\node[int] (w1) at (0.7,0.7) {};
\draw (w1) edge (v1) edge (v2) edge (v3) (v2) edge (v3)
(v3) edge[loop] (v3);
\end{tikzpicture} 
\]
The graded Hopf cooperad structure on $\gra_n$ naturally induces a Hopf cooperad structure on $\fGraphs_n$.

Also consider the total coinvariant space 

\[
\fG_n := \bigoplus_{k\geq 1} \left(\gra_n(k)\otimes \K[n]^{\otimes k}[-n] \right)_{S_k}.
\]
The total coinvariant space of any cooperad carries a graded Lie coalgebra structure. Instead of describing the Lie coalgebra sructure explicitly, it is notationally slightly easier to describe the Lie algebra structure on the dual space 
\[
\fGC_n := \fG_n^* \cong \prod_{k\geq 1} \left(\gra_n^*(k)\otimes \K[-n]^{\otimes k}[n] \right)_{S_k}.
\]
For two elements $p\in \left(\gra_n^*(k)\otimes \K[-n]^{\otimes k}[n] \right)_{S_k}$, $q\in \left(\gra_n^*(r)\otimes \K[-n]^{\otimes r}[n] \right)_{S_r}$ the Lie bracket is 
\[
[p,q] := \sum_{j=1}^k p\circ_j q - (-1)^{|p||q|} \sum_{j=1}^r q\circ_j p,
\]
where $\circ_j$ is the operation naturally induced from the operadic composition.

The elements of $\fGC_n$ can be identified combinatorially as $\K$-linear series of graphs.
It is furthermore not hard to see that the specific graph 
\[
L =\frac 12
\begin{tikzpicture}
    \node[int] (v) at (0,0){};
    \node[int] (w) at (0.7,0){};
    \draw (v) edge (w);
\end{tikzpicture}
\]
is a Maurer-Cartan element, i.e., it is of degree $+1$ and $[L,L]=0$. 
Hence we may equip $\fGC_n$ with the differential 
\[
\delta = [L,-],
\]
and dually $\fG_n$ with the differential $d=L\cdot$ given by the action of $L$. 

The dg Lie algebra $\fGC_n$ has several dg Lie subalgebras 
\[
    \GC_n \subset  \GC_n^+ \subset \GC_n^2 \subset \fGCc_n\subset \fGC_n.
\]
Here $\fGCc_n\subset \fGC_n$ is the dg Lie subalgebra of series of connected graphs, $\GC_n^2\subset \fGCc_n$ is the dg Lie subalgebra of series of connected and 1-vertex irreducible\footnote{A connected graph is called 1-vertex irreducible or biconnected if it remains connected after deleting any one vertex.} graphs whose vertices are at least bivalent and $\GC_n\subset \GC_n^2$ is the dg Lie subalgebra of series of graphs whose vertices are at least trivalent. The dg Lie algebra 
$$
\GC_n^+=  \K\tadpole \ltimes \GC_n
$$ 
is a one-dimensional extension of $\GC_n$ obtained allowing the special tadpole graph with one vertex and one edge in the series.
This is justified since the Lie bracket with this graph is combinatorially the action of adding one edge, and this does not invalidate the trivalence condition.

Dually, the dg Lie coalgebra $\fG_n$ has quotients 
\[
\fG_n \to \sG_n^2 \to \sG_n^+ \to \sG_n,
\]
where $\sG_n^2$ is obtained from $\fG_n$ by setting to zero all disconnected graphs, all non-1-vertex-irreducible graphs, and all graphs with $<2$-valent vertices. $\sG_n^+$ is further obtained by setting also graphs with two-valent vertices to zero, except for the tadpole graph, and finally in $\sG_n$ we also set the tadpole graph to zero.

Similarly, the graded Hopf cooperad $\fGraphs_n$ has quotients 
\begin{equation}\label{equ:graphs quotients}
    \fGraphs_n\twoheadrightarrow \Graphs_n^2 \twoheadrightarrow \Graphs_n,
\end{equation}
where $\Graphs_n^2$ is obtained by setting to zero all graphs with internal vertices of valence $\leq 1$, and all graphs with connected components without external vertices.
In turn, the Kontsevich cooperad $\Graphs_n$ is obtained by also setting to zero all graphs that have an internal vertex of valence 2. 

Furthermore, it has been observed in \cite{Will} (for a detailed account see \cite{DolWill}) that the graded Lie coalgebra $\fG_n$ coacts on the cooperad $\fGraphs_n$.
Dually, the graded Lie algebra $\fGC_n$ acts on $\fGraphs_n$ compatibly with the graded cooperad structure, and also on the dual operad $\fGraphs_n^*$.
We describe here the latter action of $\fGC_n$ on $\fGraphs_n^*$, from which the other actions can be deduced by duality. For $\gamma\in \fGC_n$ 
we denote by $\gamma_1\in\fGraphs_n^*(1)$ the element obtained by summing over all ways of making one vertex the external vertex.
Then for $\Gamma\in \fGraphs_n^*(r)$ with $k$ internal vertices the action is defined by the formula
\[
\gamma\cdot \Gamma = \gamma_1\circ_1 \Gamma
-(-1)^{|\gamma||\Gamma|} \sum_{j=1}^r \Gamma \circ_j \gamma - (-1)^{|\gamma||\Gamma|} \Gamma\bullet \gamma,
\]
with $\circ_j$ the operadic composition in $\fGraphs_n^*$ and where
\[
\Gamma\bullet \gamma = \sum_{j=r+1}^{r+k} \Gamma\circ_j \gamma,
\]
is obtained by inserting $\gamma$ at the internal vertices of $\Gamma$, using the underlying operadic composition on $\gra_n^*$.
One can check that this action is compatible with the operad structure, in the sense that $\Gamma\mapsto \gamma\cdot\Gamma$ is an operadic derivation.
Unfortunately, it is generally not compatible with the Hopf structure, i.e., it is not a coderivation with respect to the cocommutative coalgebra structure.
However, it is a coderivation if $\gamma$ is a linear combination of connected 1-vertex-irreducible graphs, and this is the reason for introducing this requirement above.

The action of the specific Maurer-Cartan element $L$ above defines the differential $d$ on $\fGraphs_n^*$ and by duality $\fGraphs_n$, i.e., $d=L\cdot$. Combinatorially, the differential $d$ on $\fGraphs_n$ is given by edges contraction. This differential descends to the quotients \eqref{equ:graphs quotients}.
Likewise, the action of the dg Lie algebra $\fGC_n$ on $\fGraphs_n$ descends to an action of the dg Lie algebra $\GC_n^2$ on $\Graphs_n^2$, and an action of $\GC_n^+$ and $\GC_n$ on $\Graphs_n$. The latter two actions are compatible with the dg Hopf cooperad structures, in the sense that $\gamma\cdot$ is a cooperadic coderivation and a commutative algebra derivation.
Furthermore, we note that all algebraic structures (differentials, Lie bracket, action and dg Hopf cooperad structure) preserve the number of edges minus the number of internal vertices in graphs.

There is a natural map $\Graphs_n \to \Poiss_n^c$ given by sending any graph with an internal (black) vertex to zero, and otherwise sending the algebra generator $\omega_{ij}$ to the corresponding algebra generator of $\Poiss_n^c$. We will use the following well known result:
\begin{prop}[{\cite[Section 3.3]{K2}},{\cite[Theorem 8.1]{LV}},{\cite[Propositions 3.4,3.8,3.9,Appendix F]{Will}, \cite{CGV}}]
The maps 
\[
\Graphs_n^2 \to\Graphs_n \to \Poiss_n^c
\]
are quasi-ismorphisms of dg Hopf cooperads.
The inclusion
\[
\GC_n^2 \subset \fGCc_n
\]
is a quasi-isomorphism of dg Lie algebras and the inclusions 
\[
\GC_n \subset  \GC_n^+ \subset \GC_n^2
\]
become quasi-isomorphisms upon restriction to the dg Lie subalgebras of (graphs of) loop orders $\geq 2$.
\end{prop}

\subsection{Hopf algebra coactions on \texorpdfstring{$\Graphs_n$}{Graphs n}}
\label{sec:gn action}
The action of $\SO(n)$ on $\lD_n$ can be encoded (as we will see below) by a homotopy action of the abelian Lie algebra $\fg_n=\pi^\K(\SO(n))$ on the dg Hopf cooperad model $\Graphs_n$ for $\lD_n$.
More precisely, let us consider $L_\infty$-morphisms 
\[
   \alpha:  \fg_n \to \GC_n^+.
\]
By definition, $\alpha$ is a morphism of dg cocommutative coalgebras 
\[
\alpha: C^+(\fg_n)=S^+(\fg_n[1]) \to C^+(\GC_n^+)  
\]
from the reduced Chevalley-Eilenberg complex of $\fg_n$ to that of $\GC_n^+$. 
Such a morphism in turn is an equivalent datum to a Maurer-Cartan element 
\[
    \alpha \in (C^+(\fg_n))^*\hotimes \GC_n^+ \cong \bar H(B\SO(n))\hotimes \GC_n^+
\]
in the dg Lie algebra that is the completed tensor product of the dg Lie algebra $\GC_n^+$ with the commutative algebra $\bar H(B\SO(n))$.

Yet alternatively, let 
\[
\hat \fg_n =  (\FreeLie(\bar H_\bullet(B\SO(n))), D)
\]
be the Lie algebraic bar-cobar resolution of $\fg_n$. Then the datum $\alpha$ is equivalent to a dg Lie algebra morphism 
\[
\alpha: \hat \fg_n \to \GC_n^+. 
\]

Finally, denote by $\hat \fg_n^c$ the dg Lie coalgebra (pre-)dual to $\hat \fg_n$, and let $\mU\hat \fg_n^c$ be its universal enveloping coalgebra. The latter object is a commutative dg Hopf algebra.
Then a dg Lie action of $\hat \fg_n$ is the equivalent datum to a coaction of $\mU\hat \fg_n^c$ .
Hence $\alpha$ encodes an object 
\[
 ( \mU\hat \fg_n^c,\Graphs_n) \in \dgHAlgc\dgHOpc
\]
in the category $\dgHAlgc\dgHOpc$ of pairs consisting of a dg commutative Hopf algebra coacting upon a dg Hopf cooperad.

Any dg Hopf cooperad can be considered as a homotopy Hopf cooperad by applying the Hopf cooperad version of the inclusion functor $N$ of \eqref{equ:tau N adj}. Hence we obtain an object 
\[
    ( \mU\hat \fg_n^c,N\Graphs_n) \in \dgHAlgc\dgca^{\cT^{op}}
\]
in the category of pairs consisting of a dg commutative Hopf algebra coacting on an object in $\dgca^{\cT^{op}}$, i.e., on a homotopy Hopf cooperad.

\subsection{Another model (for the homotopy quotient)}
\label{sec:another model}
Let us abbreviate $H:=H(B\SO(n))$.
Define 
\[
\BGraphs_n := \Graphs_n \otimes H
\]
to be the cooperad object in the category $\dgca^{H/}$ of dg commutative algebras under $H$, that is obtained by base change to $H$.
Since $\GC_n^+$ acts on $\Graphs_n$ the dg Lie algebra $\GC_n^+\hotimes H$ acts on $\BGraphs_n$, by just extending the action $H$-linearly.
Now suppose that $\alpha\in \GC_n^+\hotimes \bar H(B\SO(n))$ is a Maurer-Cartan element as in section \ref{sec:gn action} above.
Then we can form the twisted object 
\[
\BGraphs_n^\alpha := (\BGraphs_n, d+\alpha\cdot) .
\]
This is still a cooperad object in $\dgca^{H/}$. We may consider it as a homotopy cooperad 
\[
   N \BGraphs_n^\alpha \in (\dgca^{H/})^{\cT^{op}},
\]

Next consider the Koszul complex (or Weil algebra, see section \ref{sec:group models}) $K=W_{\SO(n)}=(\mU \fg_n^c\otimes H,D)$, 
which is compatibly a $\mU\fg_n^c$ comodule, an $H$-module, and satisfies $H( K)\cong \mathbb Q$.
Tensoring with $K$ over $H$ defines a functor from $H$-modules to $\mU\hat \fg_n^c$-comodules.
By tensoring with $K$ we hence obtain a pair
\[
  (\mU\fg_n^c, K\otimes_H  N \BGraphs_n^\alpha ) 
  \in \dgHAlgc\dgca^{\cT^{op}},
\]
consisting of a commutative Hopf algebra, coacting on an object (a homotopy cooperad) in $\dgca^{\cT^{op}}$.
Below we will also use the version
$\hat K=(\mU\hat \fg_n^c\otimes H,D)$ of the Koszul complex, obtained by replacing the Lie coalgebra $\fg_n^c$ by its bar-cobar resolution. This has analogous properties to $K$, and by tensoring with $\hat K$ we hence obtain a pair
\[
  (\mU\hat \fg_n^c, \hat K\otimes_H  N \BGraphs_n^\alpha ) 
  \in \dgHAlgc\dgca^{\cT^{op}}.
\]

\subsection{Comparison of both models }
Fix again a Maurer-Cartan element $\alpha\in \GC_n^+\hotimes \bar H(B\SO(n))$. Then we will need to compare the two algebraic objects of the previous subsections.

\begin{prop}\label{prop:GN NG}
    Let $\alpha\in \bar H(B\SO(n))\hotimes \GC_n^+$ be a Maurer-Cartan element.
Then there is a zigzag of quasi-isomorphims in $\dgHAlgc\dgca^{\cT^{op}}$
\[
(\mU \fg^c_n, K\otimes_H N \BGraphs_n^\alpha ) 
\to 
(\mU \hat \fg^c_n, \hat K\otimes_H N \BGraphs_n^\alpha ) 
\xleftarrow{f}
(\mU \hat \fg^c_n, N\Graphs_n ) 
\]
connecting the two pairs constructed in sections \ref{sec:gn action} and \ref{sec:another model} above.
\end{prop}
\begin{proof}
The left-hand arrow is induced by the quasi-isomorphism $K\to \hat K$ and is obviously a quasi-isomorphism.
To define the right-hand arrow $f$ note that the $ \hat K\otimes_H N \BGraphs_n^\alpha$ is object-wise a quasi-cofree $\mU \hat \fg^c$-comodule, cogenerated by $N \BGraphs_n^\alpha$. 
Hence $f$ is defined by its composition with the projection to cogenerators 
\[
    \Graphs_n \to \BGraphs_n
\]
and for this we take the obvious inclusion.
Next note that 
\[
\hat K \otimes_H \BGraphs_n
=
(\underbrace{\mU\hat \fg_n^c\otimes H(B\SO(n))}_{\cong \hat K \simeq \mathbb Q}\otimes \Graphs_n, D),
\]
with the first two factors on the right-hand side forming an essentially acyclic complex (the Koszul complex) as indicated. Taking a spectral sequence for the filtration by the degree in $\stG_n$ we may hence conclude that our morphism $f$ is a quasi-isomorphism.

\end{proof}

\subsection{From Hopf objects to simplicial}
The right-adjoint $\G$ (geometric realization) of the rational homotopy adjunction \eqref{equ:rht adj} is symmetric monoidal.
Hence it extends to a functor 
\[
\G \colon \dgHAlgc\dgHOpc \to \sGrp\sOp  
\]
from the category of pairs consisting of a dg commutative Hopf algebra coacting upon a dg Hopf cooperad to the category of pairs consisting of a simplicial group acting on a a simplicial operad.

In particular, from our pair $( \mU \hat\fg_n^c,\Graphs_n)$ of the previous subsection we obtain a pair of a simplicial group and a simplicial operad with an action of the group
\[
   ( \G \mU \hat\fg_n^c,\G \Graphs_n)= ( \G \mU \hat\fg_n^c,\G \Graphs_n)_\alpha =  
    (\Exp_\bullet(\hat \fg_n),\G \Graphs_n )_\alpha  \in \sGrp\sOp.  
\]
We emphasize that the pair depends on the Maurer-Cartan element $\alpha$, in that $\alpha$ encodes the action of $\Exp_\bullet(\hat \fg_n)$ on $\G \Graphs_n$.
Nevertheless, we will suppress $\alpha$ from the notation if the action is clear from the context.

\section{Derivations of the main results and corollaries from Theorem \ref{thm:main equiv}}
\label{sec:main derivations}

Since the proof of Theorem \ref{thm:main equiv} is fairly long and technical, we shall first show here how to reduce the other main results from the introduction to Theorem \ref{thm:main equiv}.
The proof of Theorem \ref{thm:main equiv} will then follow in the remaining sections \ref{sec:equiv forms} to \ref{sec:auxthmproof}.

\subsection{Proof of Theorem \ref{thm:main pair} from Theorem \ref{thm:main equiv}}
Let $(G, \POp)$ be a simplicial model for $(\SO(n), \lD_n)$, let $\K=\R$, and abbreviate $H=H(\bar W G)=H(B\SO(n))$. 
From Theorem \ref{thm:main equiv} we know that there is a zigzag of weak equivalences in $\dgca_*(\dgca^{/})^{\cT^{op}}$ connecting $(\Omega(\bar WG), \Omega(N\POp\sslash G))$ to $(H, \BGraphs_n^{m_n})$.\footnote{The "$\Omega$" appearing here is the functor of \eqref{equ:rht adj} and compared to Theorem \ref{thm:main equiv} we use that it is weakly equivalent to its PA version, see \cite{HLTV} and also the discussion in Section \ref{sec:equiv forms} below.}
From Lemma \ref{lem:rat simpl} we hence conclude that there is a zigzag of weak equivalences in $\sGrp\sOp$ connecting $(G^\R, \POp^\R)$ and 
$(G', \tau Z)$ for $(G',Z)$ any cofibrant replacement of 
$$
(\ell\G H, \G N \BGraphs_n^{m_n}\times_{\G H} (\G H)_\eta)\in \sGrp\sSet^{\cT}.
$$
In other words 
\begin{equation}\label{equ:51 proof 1}
  (G^\R, \POp^\R) \simeq (\ell\G H, \tau^h \G N \BGraphs_n^{m_n}\times_{\G H} (\G H)_\eta)
\end{equation}
Let us write 
\[
  E:= \Exp_\bullet \fg_n = \G H(G).
\]

\begin{lemma}\label{lem:W K relation}
We have that $(\ell\G H, \G N\BGraphs_n^{m_n}\times_{\G H} (\G H)_\eta)$
is weakly equivalent to $(E, \G (N\BGraphs_n^{m_n}\otimes_H K))$ in $\sGrp\sSet^{\cT}$, with $K$ the Koszul complex of the previous section.
\end{lemma}
\begin{proof}
First we note that $\G K\to \G H$ is a universal $E$-fibration, that is, $\G K$ is contractible Kan complex with a free $E$-action, and hence in particular a cofibrant element of $E\sSet$.
We may hence pick a lift in $E\sSet$
\[
     \begin{tikzcd}
        \emptyset \ar[hookrightarrow]{d} \ar{r} & WE \ar[twoheadrightarrow]{d}{\sim} \\
        \G K \ar{r} \ar[dashed]{ur} 
        & *
     \end{tikzcd}.
\]
The lift $\G K\to WE$ is automatically a weak equivalence since both sides are contractible. Taking $E$-quotients we hence have a weak equivalence 
\[
    \G H \to \bar WE.    
\]
Applying $\ell$ (which preserves all weak equivalences) and using that the adjunction counit is a weak equivalence we obtain the weak equivalences of simplicial groups
\[
    \ell \G H \xrightarrow{\sim}  \ell \bar W E \xrightarrow{\sim} E.
\]
Furthermore, recall that $(\G H)_\eta=\G H \times_{\bar W\ell\G H} W\ell\G H$ and note that $\G K = \G H \times_{\bar W E} WE$.
Hence, via the weak equivalence $\ell \G H\to E$ from above we obtain a weak equivalence $\phi:(\G H)_\eta \to \G K$ of (weakly contractible) simplicial sets over $\G H$ that is compatible with the $\ell \G H$-action on both sides. 
For any object $M\in \sSet_{/\G H}^{\cT}$ the morphism $\phi$ then induces a weak equivalence in $\sGrp\sSet^{\cT}$
\[
(\ell \G H, M \times_{\G H} (\G H)_\eta )
\to 
(E, M \times_{\G H} \G K ).
\]
Here we are interested in $M=\G N\BGraphs_n^{m_n}$.
We then conclude by noting that by Remark \ref{rem:G pullback}
\[
    \G N\BGraphs_n^{m_n} \times_{\G H} \G K
    =
    \G(N\BGraphs_n^{m_n}\otimes_H K).
\]
\end{proof}

Using Lemma \ref{lem:W K relation} we may finish the proof of Theorem \ref{thm:main pair}. Inserting the result of the Lemma into \eqref{equ:51 proof 1} we find
\begin{equation}\label{equ:51 proof 2}
  (G^\R, \POp^\R) \simeq (E, \tau^h\G (N\BGraphs_n^{m_n}\otimes_H K)).
\end{equation}

We next apply $\G$ to the zigzag of Proposition \ref{prop:GN NG}. From the fact that $\G$ preserves weak equivalences between cofibrant dgcas we then conclude that
\[
    (E, \G (N\BGraphs_n^{m_n}\otimes_H K))
    \simeq 
    (\hat E, \G N \Graphs_n)= (\hat E, N \G\Graphs_n),
\] 
where $\hat E =\G\mU\hat\fg_n^c=\Exp_\bullet(\hat\fg_n)$, and the action of $\mU\hat\fg_n^c$ on $\Graphs_n$ is defined via the Maurer-Cartan element $m_n$. Inserting the previous equation into \eqref{equ:51 proof 2} we find that 
\begin{equation*}
  (G^\R, \POp^\R) \simeq (\hat E, \tau^h N \G\Graphs_n ).
\end{equation*}
Next we use that the derived adjunction counit is a weak equivalence by Proposition \ref{prop:adj we}, so that 
\[
    (G^\R, \POp^\R) 
    \simeq (\hat E, \G\Graphs_n).
\]

We thus obtain a zigzag (see \eqref{equ:rht pair zigzag})
\[
(G,\POp) \xrightarrow{r} (G^\R, \lD_n^R) \xleftarrow{\sim} \cdots \xrightarrow{\sim}  (\hat E, \G\Graphs_n)
\]
connecting a simplicial model $(G,\POp)$ of the pair $(\SO(n), \lD_n)$ to the pair $(\hat E, \G\Graphs_n)$.
The arrows other than $r$ are weak equivalences, so to show Theorem \ref{thm:main pair} it suffices to check that $r$ is a real homotopy equivalence. 
We first note that this is true on the group part 
\[
G \xrightarrow{r} G^\R,
\]
see Proposition \ref{prop:G rht equiv}.
Next we have:
\begin{lemma}
The map $r$ induces a real homotopy equivalence on the $2$-ary parts of the operads 
\[
\POp(2) \xrightarrow{r} \lD_n^\R(2) \simeq \G\Graphs_n(2).
\]
\end{lemma}
\begin{proof}
We have that both sides have the same real homotopy type $(S^{n-1})^\R$.
Let us first consider generally endomorphisms of $(S^{n-1})^\R$.
The Sullivan minimal model of $S^{n-1}$ is 
\[
A_{n-1} := \begin{cases}
\R[\omega] & \text{for $n$ even} \\ 
\R[\omega,\eta] & \text{for $n$ odd, with $d\eta=\omega^2$} 
\end{cases},
\]
with $\omega$ of degree $n-1$ and $\eta$ of degree $2n-3$.
Any dgca endomorphism $\phi$ of $A_{n-1}$ has the form $\phi(\omega)=\lambda\omega$ and (for $n$ odd) $\phi(\eta)=\lambda^2\eta$ for some $\lambda\in \R$.
In either case, we have the following dichotomy:
\begin{itemize}
\item Either $\lambda\neq 0$, then the endomorphism is an isomorphism on real cohomology and on real homotopy groups.
\item Or $\lambda=0$, then the morphism is zero on all positive degree real cohomology and real homotopy groups. 
\end{itemize}
Now we return to our morphism $r$. Since $r$ preserves the $\SO(n)$-actions we have the following commutative diagram on real homotopy groups
\[
    \begin{tikzcd}
        \pi_k^\R(\SO(n)) \ar{d} \ar{r}{\simeq} & \pi_k^\R(\SO(n)) \ar{d} \\
        \pi_k^\R(S^{n-1}) \ar{r}{r} & \pi_k^\R(S^{n-1})
    \end{tikzcd}.
\]
The upper horizontal arrow is an isomorphism since we already know that $r$ is a real homotopy equivalence on the group part of our pair.
The vertical arrows are induced by the action of $\SO(n)$ on a point in $S^{n-1}$. The vertical arrows are both non-trivial, in that for $n$ even the Euler class, and for $n$ odd the top Pontryagin class is sent to a non-trivial real homotopy class of $S^{n-1}$.
More precisely, this is well known for the left-hand vertical arrow, coming from the topological action of $\SO(n)$ on $S^{n-1}$. For the right-hand vertical arrow this follows from inspection of our model $(\hat E, \Graphs_n(2))$ for $(\SO(n), S^{n-1})$.
We conclude that in our dichotomy for endomorphisms of the real homotopy type of $S^{n-1}$ the map $r$ cannot be of the second type, and hence must be a real homotopy equivalence.
\end{proof}

Finally, note that the cooperad $H(\lD_n)$ is cogenerated by its binary part $H(\lD_n(2))$.
Hence from the fact that $r$ is an operad morphism and that the binary part of $r$ is a cohomology isomorphism by the previous Lemma, we conclude that $H(r)$ is injective.
But since $H(\Graphs_n)=H(\lD_n)$ and by finite dimensionality in each arity, we conclude that $r$ induces an isomorphism $H(\lD_n)\to H(\lD_n^\R) =H(\Graphs_n)=H(\lD_n)$ on cohomology, and is hence a real homotopy equivalence.

\hfill\qed


\subsection{Models for the framed little \texorpdfstring{$n$}{n}-disks operads, formality and proof of Corollaries \ref{cor:flD model} and \ref{cor:partial framed formality}}\label{sec:models formality}
First note that the framed-operad construction of section \ref{sec:framed operads} defines a functor 
\begin{gather*}
\rtimes \colon \sGrp\sOp \to \sOp \\
(G, \POp) \mapsto \POp \rtimes G.
\end{gather*}
This functor preserves weak equivalences and rational weak equivalences.
Hence from the zigzag of Theorem \ref{thm:main pair} we obtain a zigzag of rational weak equivalences connecting (a simplicial model of) $\flD_n$ to $\G(\Graphs_n)\rtimes \G(\mU \hat \fg_n^c)$, see also section \ref{sec:gn action}. Since $\G$ sends coproducts to products by adjunction we have $\G(\Graphs_n)\rtimes \G(\mU \hat \fg_n^c)=\G(\Graphs_n\rtimes \mU \hat \fg_n^c)$.
This shows Corollary \ref{cor:flD model}.
By Corollary \ref{cor:omega sharp comparison} this also implies that
\[
\Omega_\sharp(\flD_n) \simeq \Graphs_n\rtimes \mU \hat \fg_n^c,  
\]
so that $\Graphs_n\rtimes \mU \hat \fg_n^c$ is a real dg Hopf cooperad model for $\flD_n$ within Fresse's real homotopy theory of operads \cite{Frextended}. 

Next restrict to $n$ even. In this case the explicit form of the MC element \eqref{equ:mn intro} shows that the coaction of $\hat \fg_n^c$ actually factors through $\fg_n^c$. Hence we may simplify our model to $\Graphs_n\rtimes \mU \fg_n^c$. Furthermore, the action of the tadpole graph $\tadpole$ is intertwined under the formality quasi-isomorphism 
\[
\Graphs_n \xrightarrow{\sim} \Poiss_n^c
\]
with an operator $T$ on $\Poiss_n^c$, that will be further discussed in section \ref{sec:FM equiv cohom} below. We hence obtain a quasi-isomorphism 
\[
\Graphs_n\rtimes \mU \hat \fg_n^c
\to 
\Poiss_n^c \rtimes \mU \fg_n^c \cong H(\flD_n).
\]

This shows that for $n\geq 4$ even $\flD_n$ is formal in either of the following equivalent senses, which are the precise meaning of Corollary \ref{cor:partial framed formality}:
\begin{itemize}
\item The simplicial operad $\sS \flD_n$ may be connected by a zigzag of real weak homotopy equivalences of simplicial operads to the simplicial operad $\G(\COp)$ for $\COp$ a cofibrant replacement of the dg Hopf cooperad $H(\flD_n)$. (For example we may take $\COp=  \Graphs_n\rtimes \mU \fg_n^c$.)
\item The dg Hopf cooperad $\Omega_\sharp(\flD_2)$ may be connected by a zigzag of quasi-isomorphisms of dg Hopf cooperads to the dg Hopf cooperad $H(\flD_n)$.
\end{itemize}

\begin{rem}\label{rem:tilde g def}
In the case of odd $n$ the action of $\hat \fg_n$ on $\Graphs_n$ encoded by the Maurer-Cartan element $m_n$ in fact factorizes through a smaller dg Lie algebra $\tilde \fg_n$ that fits into a diagram 
\[
 \hat\fg_n\xrightarrow{\sim}  \tilde \fg_n \xrightarrow{\sim}  \fg_n.
\]
Concretely, recall that $\hat\fg_n$ is generated by $\bar H(B\SO(n))$, i.e., by monomials in the Pontryagin classes. Then $\tilde \fg_n$ is the quotient of $\hat\fg_n$ obtained by setting to zero all non-linear such monomials, except for the powers of the top Pontryagin class $\Pp_{2n-2}^k$.
Then a smaller dg Hopf cooperad model for $\flD_n$ ($n$ odd) is 
\[
\Graphs_n \rtimes \mU \tilde \fg_n^c,
\]
with $\tilde \fg_n^c$ the dg Lie coalgebra (pre-)dual to $\tilde \fg_n$.
\end{rem}

\subsection{Recollection: The Drinfeld-Kohno (Quillen) models for \texorpdfstring{$D_n$}{Dn}}\label{sec:drinfeld-kohno}
The higher Drinfeld-Kohno Lie algebras are Lie algebras $\DK_n(r)$ generated by symbols $t_{ij}=(-1)^n t_{ji}$ of degree $n-2$, with $1\leq i\neq j \leq n$, with relations $[t_{ij}, t_{kl}]=0$ for $\#\{i,j,k,l\}=4$ and $[t_{ij},t_{ik}+t_{jk}]=0$ for $\#\{i,j,k\}=3$.
These Lie algebras assemble (for varying $r$) into an operad in Lie algebras which we call the $n$-Drinfeld Kohno operad.
Concretely, the symmetric group action is by the obvious permutation of indices. The operadic compositions
\[
\circ_r: \DK_n(r) \oplus \DK_n(r') \to \DK_n(r+r'-1)
\]
are defined by the following rules.
\begin{align*}
 \DK_n(r) \ni t_{ij} &\mapsto t_{ij}  \in\DK_n(r+r'-1)& \text{for $i,j<r$} \\
 \DK_n(r) \ni t_{ir} &\mapsto t_{ir}+t_{i(r+1)}+\dots + t_{i(r+r'-1)}  \in\DK_n(r+r'-1) & \text{for $i<r$} \\
 \DK_n(r') \ni t_{ij} &\mapsto t_{(i+r)(j+r)}  \in\DK_n(r+r'-1) & \\ 
\end{align*}
All other operadic compositions are determined by the operad axioms and the symmetric group action.
The following result is well known.
\begin{thm}\label{thm:drinfeld-kohno}
The Drinfeld-Kohno operad $\DK_n$ forms a rational Quillen model for the operad $\lD_n$ for each $n\geq 2$.
\end{thm}
Let us give a sketch of a proof that these objects are indeed real Quillen models. The proof necessarily uses some form of the (real) formality of $\lD_n$
First, we have seen that the Hopf operad $\Graphs_n$ is a (dgca) model for $\lD_n$. Furthermore, $\Graphs_n=C(\ICG_n^c)$ is the Chevalley complex of the cooperad in $L_\infty$-coalgebras $\ICG_n^c$ given by the internally connected graphs. We are done if we can show that the operads in $L_\infty$-algebras $\ICG_n:=(\ICG_n^c)^*$ and $\DK_n$ are quasi-isomorphic. To this end, one can define an auxiliary grading on $\ICG_n$ for which a graph $\Gamma$ has auxiliary degree 
\[
2 \# (\text{internal vertices of }\Gamma) -  \# (\text{edges of }\Gamma) +1.
\]
The definition is chosen such that:
\begin{itemize}
  \item The grading is compatible with the $L_\infty$-algebra structure, in the sense that the $k$-ary $L_\infty$-operation on $\ICG_n$ has auxiliary degree $2-k$. In particular, the differential has auxiliary degree $+1$.
  \item The grading is compatible with the operad structure, in the sense that the $k$-ary parts of the $L_\infty$-morphisms describing the operadic composition have auxiliary degree $1-k$.
  \item The cohomology is concentrated in auxiliary degree $0$. This is because the generators $t_{ij}$ live in degree 0, the corresponding graphs having one edge and no (internal) vertices, and the auxiliary grading is compatible with the Lie bracket.
  \end{itemize}

Given these conditions we may define the truncated sub-operad in $L_\infty$ algebras
\[
\TCG_n \subset \ICG_n
\]
formed by all elements of auxiliary degree $<0$, the closed elements in degree $0$, and no elements in auxiliary degree $>0$.
We then have the zig-zag of quasi-isomorphisms
\[
\ICG_n\leftarrow \TCG_n \to H(\ICG_n) {=}\DK_n
\]
showing that $\ICG_n$ is formal, and hence that $\lD_n$ is coformal and that $\DK_n$ is a Quillen model for $\lD_n$.

\subsection{Framed Drinfeld-Kohno models}\label{sec:quillen}
In this section we discuss framed analogs $\DKF_n$ of the higher Drinfeld-Kohno operads of Lie algebras. 
Concretely, let as before $\alg g_n=\pi_{>0}(\SO(n))\otimes \R$ as abelian Lie algebra. Concretely, $\alg g_n$ is spanned as a vector space by the Pontryagin classes $p_{4s-1}$ in cohomological degree $1-4s$, and, if $n$ is even, the Euler class $e$ in cohomological degree $1-n$. We now define, as a Lie algebra
\[
\DKF_n(r) := \DK_n(r) \oplus \underbrace{\alg g_n[-1]\oplus \cdots \oplus \alg g_n[-1]}_{r\times}.
\]
In words, $\DKF_n(r)$ is generated by the $t_{ij}$ with relations as in the previous section, and by additional commuting generators $p_{4s-1}^{j}$ and $e^j$ for $1\leq j\leq r$. (The superscript is an index, not an exponent.) There is an obvious action of the symmetric group by permuting indices.
We define the operad structure by extending that on $\DK_n$, such that the composition
\[
\circ_r: \DKF_n(r) \oplus \DKF_n(r') \to \DKF_n(r+r'-1)
\]
acts on the additional generators as follows.
\begin{align*}
 \DKF_n(r) \ni e^{j} &\mapsto e^j  & \text{for $n$ even, $j<r$} \\
 \DKF_n(r) \ni e^{r} &\mapsto \sum_{j=r}^{r'} e^{j} + \sum_{r\leq i<j \leq r'} t_{ij}   & \text{for $n$ even} \\
 \DKF_n(r') \ni e^{i} &\mapsto e^{i+r-1}  & \text{for $n$ even} \\
 \DKF_n(r) \ni p_{2n-3}^{r}&\mapsto \sum_{j=r}^{r'} p_{2n-3}^{j} + \sum_{r\leq i<j \leq r'}[t_{ij},t_{ij}] +  2\sum_{r\leq i<j<k \leq r'}[t_{ij},t_{jk}] & \text{for $n$ odd} \\
 \DKF_n(r) \ni p_{4s-1}^{j}&\mapsto p_{4s-1}^{j}  & \text{for $4s-1<2n-3$} \\
 \DKF_n(r') \ni p_{4s-1}^{j}&\mapsto p_{4s-1}^{j+r-1}  & \text{for $4s-1\leq 2n-3$}
\end{align*}
Again, the other operadic compositions are determined by the symmetric group action. 
In words, the operad structure is (in a suitable sense) trivial for all basis elements of $\alg g_n$ except for the Euler class (for $n$ even), and the top Pontryagin class (for $n$ odd), for which we modified the composition rules.
We will show below that $\DKF_n$ is a Quillen model for $\flD_n$. For now, we show the following:

\begin{thm}\label{thm:homFICG}
The framed Drinfeld-Kohno operad $\DKF_n$ is the homology of a real Quillen model for $\flD_n$ for each $n\geq 2$.
\end{thm}
Note that the theorem shows that $\DKF_n$ is a Quillen model for $\flD_n$ provided we can show coformality.
\begin{proof}
By section \ref{sec:models formality} we already have a dg Hopf cooperad model 
\[
\stG_n \rtimes \mU\hat{\alg g}_n^c
\]
for $\flD_n$, depending on the graphical Maurer-Cartan element $m_n$ of \eqref{equ:mn intro}. 
Nevertheless, the above model is quasi-free as a (collection of) dgcas. The dual space of the generators is
\[
\FICG_n(r):= \ICG_n(r) \oplus\underbrace{ \hat{\alg g}_n[-1]\oplus\cdots\oplus \hat{\alg g}_n[-1]}_{r\times}
\]
where $\hat{\alg g}_n$ is (up to a degree shift) isomorphic to a (quasi-)free Lie algebra with generators indexed by elements of $\bar H_\bullet(B\SO(n))$.
The differential on $\FICG_n$ is not affected by the twist with our MC element $m_n$, and hence it is immediate that, as a collection of graded vector spaces
\begin{equation}\label{equ:ICGFDKF}
H(\ICGF_n) \cong \DKF_n.
\end{equation}
However, note that the operadic compositions ($L_\infty$-morphisms) in $\FICG_n$ are relatively complicated and do depend on $m_n$, and it is a priori not obvious that the isomorphism \eqref{equ:ICGFDKF} is compatible with the operadic compositions.
To show this statement let us note three facts: (i) The only contributions to the homology of $\ICG_n$ can come from graphs that are trivalent trees after deleting the external vertices. This is true because it is true for the generators $t_{ij}$, and the property is preserved by the Lie bracket. (ii) Looking at the way the action of $\GC_n^+$ is defined, the only graphs in $\GC_n^+$ that can possibly create such objects are graphs which after deletion of one vertex become trivalent trees.
(iii) For $n$ even the only graph appearing in $m_n$ is the tadpole $\tadpole$, which satisfies this condition. For $n$ odd, only the theta graph $\thetagr$ satisfies the condition and is present in $m_n$. 

Now, we know the coefficients of the tadpole and theta graphs in $m_n$, and hence we can understand the operad structure on $H(\ICGF_n)$. The operad structure on $\DKF_n$ is defined precisely such that \eqref{equ:ICGFDKF} is compatible with the compositions.

\end{proof}

\subsection{ Coformality of \texorpdfstring{$\flD_n$}{fDn} and proof of Corollary \ref{cor:FE3coformal}}
Now we are ready to prove the coformality of $\flD_n$, and hence check that $\DKF_n$ is indeed a Quillen model.

After the discussion in the previous subsection, it suffices to construct a quasi-isomorphism between the operad in $L_\infty$-algebras $\ICGF_n$ and its homology $H(\ICGF_n)=\DKF_n$. 
As a first step we may pass to the quasi-isomorphic quotient $\wICGF_n$ defined such that 
\begin{equation}\label{equ:wICGF}
  \wICGF_n(r) = \ICG_n(r)\oplus \tilde{\alg g}_n[-1]\oplus\cdots \oplus\tilde{\alg g}_n[-1]
\end{equation}
with $\tilde\fg_n=\fg_n$ for even $n$ and $\tilde\fg_n$ as defined in Remark \ref{rem:tilde g def} for odd $n$.
Since the action of $\hat{\alg g}_n$ on $\Graphs_n$ factors through the quasi-isomorphic quotient $\tilde{\alg g}_n$ we clearly have a quasi-isomorphism 
\[
  \wICGF_n \to \ICGF_n. 
\]

Now we desire to copy the truncation trick in the proof of the non-framed analogous result Theorem \ref{thm:drinfeld-kohno} above. Looking at that proof, it is clear that we are done if we can define an auxiliary grading on $\wICGF_n$ with the same formal properties, i.e.: 

\begin{itemize}
\item The grading is compatible with the $L_\infty$-algebra structure, in the sense that the $k$-ary $L_\infty$-operation on $\wICGF_n$ has auxiliary degree $2-k$. In particular, the differential has auxiliary degree $+1$.
\item The grading is compatible with the operad structure, in the sense that the $k$-ary parts of the $L_\infty$-morphisms describing the operadic composition have auxiliary degree $1-k$.
\item The cohomology is concentrated in auxiliary degree $0$.
\end{itemize}
Let us for the moment assume that such an auxiliary grading exists. Then we can just proceed as in the non-framed case, cf. section \ref{sec:drinfeld-kohno}. We construct the desired quasi-isomorphism connecting $\wICGF_n$ to its cohomology $\DKF_n$ as
\[
\wICGF_n\leftarrow \TCGF_n \to H(\wICGF_n) \stackrel{\text{Thm. \ref{thm:homFICG}}}{=}\DKF_n,
\]
where $\TCGF_n\subset \wICGF_n$ is the sub-operad in $L_\infty$-algebras obtained by auxiliary-degree truncation, containing all elements of negative auxiliary degree and the closed elements of auxiliary degree zero.
Hence Corollary \ref{cor:FE3coformal} would be shown.

So we are left with verifying that an auxiliary grading satisfying the above three properties exists.
We define our auxiliary grading on the summand $\ICG_n$ of \eqref{equ:wICGF} as in the proof of Theorem \ref{thm:drinfeld-kohno} above as (for $\Gamma\in \ICG_n$ some graph)
\[
2 \# (\text{internal vertices of }\Gamma) -  \# (\text{edges of }\Gamma) +1.
\]
It follows as above that all the generators $t_{ij}$ of the part $\DK_n$ of the cohomology, and hence the whole $\DK_n$, are concentrated in auxiliary degree 0.
Furthermore, the auxiliary grading on the pieces $\tilde{\alg g}_n[-1]$ is defined as follows. 
In the case of even $n$ we declare all of $\tilde{\alg g}_n=\alg g_n$ to be in auxiliary degree $0$. 
For odd $n$, $\tilde{\alg g}_n$ is generated by symbols $p_{4j-1}$ of cohomological degree $1-4j$ for $j=1,\dots,(n-3)/2$ and symbols $a_k$ ($k=1,2,\dots$) of degree $1-k(2n-2)$ that are dual to the powers of the top Pontryagin class $\Pp_{2n-2}^{k}$.
The generators $p_{4j-1}$ are central, while the $a_k$ generate a free Lie subalgebra of $\tilde{\alg g}_n$. We have $dp_{4j-1}=0$ and 
\begin{equation}\label{equ:dak}
d a_k=\sum_{i,j\geq 1\atop i+j=k} [a_i,a_{j}].
\end{equation}
We now define the auxiliary grading by declaring the abelian piece (i.e., all the $p_{4j-1}$) to live in auxiliary degree zero, by declaring $a_k$ to live in auxiliary degree $2-2k$, and by declaring the Lie bracket in $\tilde \fg_n$ to be of auxiliary degree $-1$. Note that the latter assertion is not in contradiction with the previous assertion that the Lie bracket on $\wICGF_n(r)$ is to have degree zero. This is because the Lie bracket on $\wICGF_n(r)$ does not involve the Lie bracket in $\tilde \fg_n$, but rather the Lie bracket on the part $\tilde{\fg}_n[-1]$ is zero. Furthermore, all $L_\infty$-algebra operations of $\wICGF_n(r)$ that involve at least one argument in $\tilde{\fg}_n[-1]$ are zero, and hence the first of the three assertions above holds. The third also holds since we defined the auiliary grading on $\tilde{\fg}_n[-1]$ such that the cohomology $\fg_n[-1]$ is concentrated in auxiliary degree 0.
Next, we turn to the second assertion above, i.e., the compatibility of the auxiliary grading with the operadic structure, which is the most difficult verification.
Consider a generic operadic composition (an $L_\infty$-morphism) 
\[
\theta: \wICGF_n(r) \oplus \wICGF_n(s_1) \oplus \cdots \oplus \wICGF_n(s_r)
\to \wICGF_n(s_1+\cdots +s_r).
\]
A generic non-vanishing $N$-ary morphism of $\theta$ has the form 
\[
\theta_N(\gamma_1^0,\dots,\gamma_k^0, x_1^0,\dots,x_l^0,
\gamma_1^1,\dots,\gamma_{k_1}^1, \dots , \gamma_1^r,\dots,\gamma_{k_r}^r),
\]
where the superscript indicates in which of the $k+1$ summands the argument lives, e.g., $\gamma_1^0\in \wICGF_n(r)$, $\gamma_1^1\in \wICGF_n(s_1)$ etc. Symbols ``$\gamma$'' live in the graphical part $\ICG_n(-)$ of the respective summand, while the ``$x$'' live in the part $\tilde \fg_n$. The value of $\theta_N$ is computed (roughly) by the following combinatorial procedure:
\begin{enumerate}
    \item Make each group $\gamma_1^j,\dots,\gamma_{k_j}^j$ into one graph $\Gamma_j\in \Graphs_n^*(s_j)$ by gluing at the external vertices.
    \item Replacing the $x_j^0$ by their corresponding graphs in $\GC_n^+$ according to the explicit form of $m_n$. Then the $x_j^0$ that are in the $i$-th summand $\tilde \fg_n[-1]$ of \eqref{equ:wICGF} act on the graph $\Gamma_i$ according to the action of $\GC_n^+$ on $\Graphs_n^*$. We call the resulting elements $\Gamma_1',\dots, \Gamma_r'$.
    \item Next we compute the usual operadic composition in the operad $\Graphs_n^*$ of the element $\Gamma_0\in \Graphs_n(r)$ with the elements $\Gamma_j'\in \Graphs_n(s_j)$ to obtain an element $\Gamma_0(\Gamma_1',\dots, \Gamma_r')\in \Graphs^*(s_1+\dots+s_r)$.
    \item Finally we project $\Gamma_0(\Gamma_1',\dots, \Gamma_r')\in \Graphs^*(s_1+\dots+s_r)$ to the internally connected part to obtain the desired element $\theta_N(\gamma_1^0,\dots,\gamma_{k_0}^0, x_1^0,\dots,x_l^0,
\gamma_1^1,\dots,\gamma_{k_1}^1, \dots , \gamma_1^r,\dots,\gamma_{k_r}^r)\in \wICGF(s_1+\cdots+s_r)$. 
\end{enumerate}
Let us trace the auxiliary degrees through this procedure. Temporarily denote the auxiliary degree of some object $X$ by $|X|$.
Also extend the definition of the auxiliary degree to $\Graphs_n^*$ by declaring a graph $\Gamma\in \Graphs_n^*$ to be of auxiliary degree
\[
|\Gamma|=2 \# (\text{vertices of }\Gamma) -  \# (\text{edges of }\Gamma).
\]
Then we have that
\[
|\Gamma_j| = |\gamma_1^j| +\cdots +|\gamma_{k_j}^j| - k_j,
\]
since the gluing at external vertices preserves all internal vertices and edges.
Also temporarily extend the definition of the auxiliary degree to $\GC_n^+$, declaring a graph $\Gamma\in \GC_n^+$ to be of auxiliary degree
\[
2 \# (\text{vertices of }\Gamma) -  \# (\text{edges of }\Gamma) -2.
\]
Then easy inspection of the MC element $m_n$ shows that an $x_j^0$ of auxiliary degree $|x_j^0|$ is sent to an element of auxiliary degree $|x_j^0|-1$ of $\GC_n^+$. To finish the analysis of step (2) above note that the auxilary degrees on $\Graphs_n^*$ and $\GC_n^+$ are defined so that the action of the latter on the former preserves the auxiliary degree. We hence find:
\[
|\Gamma_1'|+\cdots +|\Gamma_r'| =
\sum_{j=1}^l |x_j|^0 + \sum_{j=1}^r\sum_{i=1}^{k_j} |\Gamma_i^j| 
-
l - \sum_{j=1}^r k_j.
\]
The operadic composition in step (3) is additively compatible with the auxiliary gradings, so that 
\[
|\Gamma_0(\Gamma_1',\dots, \Gamma_r')|
=
\sum_{j=1}^l |x_j^0| + \sum_{j=0}^r\sum_{i=1}^{k_j} |\Gamma_i^j| 
-
\underbrace{(l + \sum_{j=0}^r k_j)}_{=N}.
\]
Finally, the projection to $\wICGF$ raises the auxiliary degree by one again, due to the extra ``$+1$'' in the definition of the auxiliary degree on $\wICGF$. We conclude that the $N$-ary part of the composition morphisms has auxiliary degree $1-N$ as desired.

  
  

\hfill \qed

\subsection{Non-formality for odd \texorpdfstring{$n$}{n}: proof of Corollary \texorpdfstring{\ref{cor:odd nonformality}}{} }

Let $n\geq 3$ be an odd integer, and consider the $\SO(n)$-framed little $n$-disks operad.
We obtain a model for its chains by dualizing our model for differential forms of section \ref{sec:models formality}.
Concretely, the operad of chains is quasi-isomorphic (as a homotopy operad) to 
\[
 \op P := \Graphs_n^* \rtimes \mU \tilde \fg_n,
\]
where $\mU \tilde \fg_n$ is the universal enveloping algebra of $\tilde \fg_n$ of Remark \ref{rem:tilde g def}.
More explicitly, $\mU \tilde \fg_n=H_\bullet(\SO(n-2))\otimes F$ 
with $F=\R\langle a_1,a_2,\dots \rangle$ is a free algebra in symbols $a_k$ of degree $1-k(2n-2)$, with differential \eqref{equ:dak}, see also the previous section.
In particular, $a_1$ represents the (dual) top Pontryagin class.
The action of the Hopf algebra $\mU \tilde \fg_n$ on the operad $\Graphs_n$ is such that $H_\bullet(\SO(n-2))$ acts trivially, and the action of $a_j$ is (up to an unimportant prefactor) the action of the graph in $\GC_n$ with two vertices and $2j+1$ edges, see \eqref{equ:mn intro}.

On the other hand, the homology operad is \cite{SW}
\[
 \op H := \Poiss_n \circ H_\bullet(\SO(n)),
\]
with $H_\bullet(\SO(n))$ acting trivially on $\Poiss_n$.

Our goal is to show that $\op P$ is not quasi-isomorphic to $\op H$ as a dg operad.
If it was we could find a quasi-isomorphism
\beq{equ:nfo1}
 \op H_\infty \to \op P
\eeq
for $\op H_\infty$ a(ny) cofibrant replacement of $\op H$.
We want to show by obstruction theory that such a morphism cannot exist.
First, the operad $\op H$ is Koszul, and we may pick 
\[
 \op H_\infty := \Omega(\op H^{\vee})
\]
to be the cobar construction of the Koszul dual cooperad.
More explicitly, the Koszul dual operad is identified with 
\[
 \op H^! = H(B\SO(n)) \otimes \Poiss_n\{n\}.
\]

We will try to construct \eqref{equ:nfo1} inductively on the arity $r$ and hit an obstruction in arity 3.
We will impose a filtration on $\op P$ as follows.
We say that the weight of a graph in $\Graphs_n$ is the number of edges.
We say that the lower Pontryagin classes (generators of $H_\bullet(\SO(n-2))$) have some weight $>3$, say $4$.
We say that $a_{j}$ has weight $2j+1$.
This imposes a filtration by weight on $\op P$. 

To simplify the obstruction argument, we will (try to) construct \eqref{equ:nfo1} only up to weight 3, i.e., we ignore terms of weight $\geq 4$ in $\op P$.
Mind that only a small (finite dimensional in each arity) subspace of $\op P$ lives in weight $\leq 3$:
\begin{itemize}
 \item We can have an empty graph decorated by one copy of $P_1$.
\item We can have a graph with at most three edges decorated by the trivial element of $A$.
\end{itemize}

In particular, in arity one the subspace of elements of weight $\leq 3$ is 3-dimensional.
To simplify further, we note that $H(B\SO(n))\cong \R[\Pp_4,\Pp_8,\dots,\Pp_{2n-2}]$. Hence we may equip $\op H^!$ with a (``co-weight'') grading such that each $\Pp_j$ ($j\leq 2n-6$) has co-weight 2 and $\Pp_{2n-2}$ has co-weight 1.
We only consider terms of co-weight $\leq 1$.

Now, in arity $r=1$, we have to provide a map from the generators $\op H^{\vee}(1)$ to $\op P(1)$, or dually an element of $\op H^!(1)\otimes \op P(1)$.
In co-weight $\leq 1$ we have only the $\Pp_{2n-2}$ (of degree $2n-2$) as generator.
It has to be sent to a closed element of $\op P(1)$. In weight $\leq 3$, the only closed element is $a_1$, in degree $3-2n$.
Hence, using that the map must be a quasi-isomorphism, modulo terms of co-weight $\geq 2$ or weight $\geq 4$ the arity 1 map is described by 
\[
 \Pp_{2n-2} \otimes a_1 + (\cdots) \in \op H^!(1)\otimes \op P(1).
\]
Next, consider the arity $r=2$ part.
In coweight $0$ we have the bracket ($=:b$) and product ($=:p$) generators in $\op H^!(2)$.
In coweight 1 we have the $\Pp_{2n-2} b$ and $\Pp_{2n-2}p$.
In $\op P(2)$ one can list all elements of weight $\leq 3$ (we draw only one graph for each $S_2$ orbit):
\begin{align*}
 &\begin{tikzpicture}
  \node[ext] (v1) at (0,0) {};
  \node[ext] (v2) at (.5,0) {};
 \end{tikzpicture}
& &
\begin{tikzpicture}
  \node[ext] (v1) at (0,0) {};
  \node[ext] (v2) at (.5,0) {};
\draw (v1) edge (v2);
 \end{tikzpicture}
& &
\begin{tikzpicture}
  \node[ext] (v1) at (0,0) {};
  \node[ext] (v2) at (.5,0) {};
\draw (v1) edge[bend left] (v2) (v1) edge[bend right] (v2);
 \end{tikzpicture}
& &
\begin{tikzpicture}
  \node[ext] (v1) at (0,0) {};
  \node[ext] (v2) at (.5,0) {};
\draw (v1) edge (v2) edge[bend left] (v2) edge[bend right] (v2);
 \end{tikzpicture}
\\
&
\begin{tikzpicture}
  \node[ext, label={$\scriptstyle a_1$}] (v1) at (0,0) {};
  \node[ext] (v2) at (.5,0) {};
 \end{tikzpicture}
& &
\begin{tikzpicture}
  \node[ext] (v1) at (0,0) {};
  \node[ext] (v2) at (.5,0) {};
  \node[int] (v3) at (0,0.5) {};
\draw (v1) edge (v3) edge[bend left] (v3) edge[bend right] (v3);
 \end{tikzpicture}
& &
\begin{tikzpicture}
  \node[ext] (v1) at (0,0) {};
  \node[ext] (v2) at (.5,0) {};
  \node[int] (v3) at (0.25,0.5) {};
\draw (v1) edge[bend left] (v3) edge[bend right] (v3) (v2) edge (v3);
 \end{tikzpicture}\, .
\end{align*}
Writing down the requirement that \eqref{equ:nfo1} should commute with the differentials and induce an isomorphism (say the identity) on cohomology, one quickly checks that in arity $2$ the map \eqref{equ:nfo1} must be described by
\[
 b\otimes 
 \begin{tikzpicture}
  \node[ext] (v1) at (0,0) {};
  \node[ext] (v2) at (.5,0) {};
 \end{tikzpicture}
  + p \otimes 
  \begin{tikzpicture}
    \node[ext] (v1) at (0,0) {};
    \node[ext] (v2) at (.5,0) {};
  \draw (v1) edge (v2);
   \end{tikzpicture} + 
\Pp_{2n-2}p \otimes 
\begin{tikzpicture}
  \node[ext] (v1) at (0,0) {};
  \node[ext] (v2) at (.5,0) {};
\draw (v1) edge (v2) edge[bend left] (v2) edge[bend right] (v2);
 \end{tikzpicture}
+
\Pp_{2n-2}b \otimes
\begin{tikzpicture}
  \node[ext] (v1) at (0,0) {};
  \node[ext] (v2) at (.5,0) {};
\draw (v1) edge[bend left] (v2) edge[bend right] (v2);
 \end{tikzpicture}
+ (\cdots).
\]

Next consider arity $r=3$.
The co-weight $\leq 1$ elements are built using 0,1 or 2 brackets, possibly times $\Pp_{2n-2}$.
Again one computes that the double bracket $b_{1,23}b_{23}$ must be paired with an element $x\in \op P(3)$ whose differential is the graph 
\[
\begin{tikzpicture}
 \node[ext] (v1) at (0,0) {$\scriptstyle 1$};
\node[ext] (v2) at (0.5,0) {$\scriptstyle 2$};
\node[ext] (v3) at (1,0) {$\scriptstyle 3$};
\draw (v1) edge [bend right] (v2) edge[bend left] (v3);
\end{tikzpicture}
\]
Since no such element exists we have found our obstruction.

\hfill\qed

%
%


\section{PA models for equivariant cohomology}
\label{sec:equiv forms}
The goal of the remainder of the paper is to show Theorem \ref{thm:main equiv}.
In this section we discuss several models of equivariant forms, and in particular introduce the version $\Omega^{PA}_{\SO(n)}(-)$ occurring in Theorem \ref{thm:main equiv}.

\subsection{Semi algebraic sets and PA forms}\label{sec:PAforms}
Following \cite{K2} and \cite{LV} we will study the real homotopy type of the (framed or unframed) little cubes operads by considering the dgca of PA forms on (a version of) this operad.
The construction of the dgca of PA forms $\Omega_{PA}(X)$ on a semi-algebraic set $X$ was sketched in the appendix of \cite{KS}, and worked out in detail in \cite{HLTV}.
For the purposes of this paper, we will use the following properties of PA forms shown in \cite{HLTV}.
\begin{itemize}
\item The functor $\Omega_{PA}$ is a contravariant functor from the category of semi-algebraic sets to the category of dgcas. It is lax monoidal in that there is a natural quasi-isomorphism $\Omega_{PA}(X)\otimes \Omega_{PA}(X)\to \Omega_{PA}(X\times Y)$.
\item It is weakly equivalent to Sullivan's functor applied to the simplicial complex $\Omega(\sS -)$.
\item There is a dg subalgebra $\Omega_{min}(X)\subset \Omega_{PA}(X)$ containing the semi-algebraic functions, and for $\pi: X\to Y$ an SA bundle (see \cite{HLTV}) there there is a push-forward (``fiber integral'') operation
\[
 \pi_* : \Omega_{min}(X) \to \Omega_{PA}(Y)
\]
that we shall also denote with a ``fiber integral'' sign $\pi_*=\int_f$ if no confusion arises.

We note in particular that the forgetful maps $\pi: \FM_m(r+s)\to \FM_n(r)$ of the Fulton-MacPherson compactification of the configuration spaces of points satisfy the hypothesis, and hence give rise to push-forward operations.
\item The push-forward satisfies the Stokes formula \cite[Proposition 8.12]{HLTV}
\begin{equation}\label{equ:PA Stokes}
d\int_f \alpha = \int_f d\alpha+(-1)^{|\alpha|-k}\int_{\partial f} \alpha,  
\end{equation}
where $k$ is the fiber dimension and $\int_{\partial f}$ denotes the push-forward associated to the fiberwise boundary.
\end{itemize}

We shall treat the functor $\Omega_{PA}$ mostly as a ``blackbox'', using only the above formal properties, and refer the reader to loc. cit. for more information on the construction of $\Omega_{PA}$.

\subsection{Models for the homotopy quotient and equivariant forms}\label{sec:simpl models BG}
Let $G$ be a simplicial group and let $X\in G\sSet$ be a simplicial set with a $G$-action.
We considered above the homotopy quotient (with $WG$ a simplicial model for $EG$ as in Proposition \ref{prop:DDK}) 
\[
X\sslash G := (X\times WG)/G \in \sSet,
\]
and our model $\Omega(X\sslash G)\in \dgca$ for the equivariant forms on $X$.

For computational purposes it will be convenient to use a different model for the equivariant forms, based on the classical construction of the homotopy quotient as the fat geometric realization of the bar construction $X$, see \cite{Segal}.
More precisely, we define a dgca $\Omega^{fat}_{G}(X)$ as the end
\[
\Omega^{fat}_{G}(X)
:=
\int^{[j]\in {\Delta_+}}
\Omega(X\times G^j) \otimes \Omega(\Delta^j),
\]
where $\Delta_+$ is the semi-simplicial category.
The following result shows that the construction indeed gives a model for the equivariant differential forms on $X$.
\begin{prop}\label{prop:fat equiv model}
There is a zigzag of natural weak equivalences of dg commutative algebras 
\[
    \Omega^{fat}_{G}(X) \to \bullet \leftarrow \Omega(X\sslash G).
\]
\end{prop}
\begin{proof}
We begin by considering the bar construction 
\[
B(X, G) = X \times G^{\times \bullet}    
\]
which is a bisimplicial set. More concretely, we have 
\[
    B_{jk}(X, G) = X_j \times G_j^{\times k}.        
\]
We have that the construction $X\sslash G$ agrees with the "bar construction" (or total simplicial set) $\bar W(B(X, G))$, see \cite{CegarraRemedios}, of the bisimplicial set $B_\bullet(X, G)$.
Then \cite[Theorem 1.1]{CegarraRemedios} asserts that $X\sslash G$ is naturally weakly equivalent to the diagonal $\diag B(X, G)\in \sSet$ of the bisimplicial set $B(X, G)$.
(We note that a similar result is also shown in \cite{BergerHuebschmann}.)
Next, for any bisimplicial set the diagonal is equal to the realization along one simplicial dimension, see \cite[Proposition II.3.3.20]{F}
\[
(\diag B(X, G) )_k
=
|B(X, G)|_k
:=
\int_{[j]\in {\Delta}}
B_{kj}(X,G) \times \Delta^j
.
\]
Now for the bar construction we furthermore have that the realization is weakly equivalent to the fat realization 
\[
|B(X, G)|_k
\simeq
\|B(X, G)\|_k
:=
\int_{[j]\in {\Delta_+}}
B_{kj}(X,G) \times \Delta_k^j
.
\]
More precisely the weak equivalence of normal and fat realization is a classical result for good simplicial spaces shown in \cite[Appendix A]{SegalCategories}. The proof therein can be transcribed for bisimplicial sets instead of simplicial spaces using the Quillen equivalence between simplicial sets and topological spaces.
Finally, the bisimplicial set $B(X, G)$ is then good in the sense of Segal, since all "vertical" degeneracy morphisms of $B(X, G)$ are injective and hence cofibrations.
We finally obtain a chain of natural weak equivalences of simplicial sets 
\[
X\sslash G \to \bullet \leftarrow \|B(X, G)\|
\]
connecting our initial model for the homotopy quotient to the fat realization of the bar construction. We apply $\Omega(-)$ and use that $\Omega(-)$ preserves all weak equivalences and sends coends to ends by adjunction, to find a zigzag of natural weak equivalences 
\[
   \Omega( X\sslash G) \to \bullet \leftarrow \Omega(\|B(X, G)\|)
   =
   \int^{[j]\in {\Delta_+}}
\Omega(X\times G^j\times\Delta^j).
\]
Next, we have a natural weak equivalence 
\begin{equation*}
    \Omega(X\times G^j) \otimes \Omega(\Delta^j)\to \Omega(X\times G^j\times\Delta^j)
\end{equation*}
and we claim that the induced morphism 
\[
    \Omega^{fat}_{G}(X) \to \Omega(\|B(X, G)\|)
\]
is a weak equivalence.
To this end we equip both sides with the complete descending filtration by the "level" $j$ appearing in the ends on either side. It suffices to show that the associated graded morphism is a quasi-isomorphism. But the $j$-th graded piece of this morphism in turn is identified with the inclusion
\[
    \Omega(X\times G^j) \otimes \Omega_\partial(\Delta^j)\to \Omega_\partial(X\times G^j\times\Delta^j),
\]
where the subscript $(-)_\partial$ means that the forms vanish upon appliying any cosimplicial coboundary map.
The latter morphism is evidently a quasi-isomorphism, hence the proposition is shown. 
\end{proof}

The model for equivariant forms of the proposition has the advantage that it can be easily adapted to the setting of PA forms. 
If $G$ is an algebraic (topological) group that acts semi-algebraically on a semi-algebraic set $X$ we define

\begin{equation}\label{equ:om G pa def}
 \Omega^{PA}_{G}(X)
:=
\int^{[j]\in {\Delta_+}}
\Omega_{PA}(X\times G^j) \otimes \Omega(\Delta^j).
\end{equation}
Using that there is a zigzag of natural weak equivalences connecting $\Omega_{PA}(X)$ to $\Omega(\sS X)$ we then obtain, by using the argument at the end of the proof of Proposition \ref{prop:fat equiv model}:
\begin{lemma}\label{lem:Om PA zigzag}
There is a zigzag of natural weak equivalences
\[
    \Omega^{PA}_{G}(X) \to \bullet \leftarrow 
    \Omega(\sS X \sslash \sS G).
\]
\end{lemma}
\hfill \qed 

By functoriality we also obtain a zigzag of weak equivalences connecting the morphism $\Omega(* \sslash \sS G)\to \Omega(\sS X \sslash \sS G)$ to $\Omega^{PA}_{G}(*)\to \Omega^{PA}_{G}(X)$.

We also record here the smooth analog of the above constructions:
If we $G$ is a Lie group acting an a smooth manifold, then we define 
\begin{equation}\label{equ:Omega sm def}
 \Omega^{sm}_{G}(X)
:=
\int^{[j]\in {\Delta_+}}
\Omega_{dR}(X\times G^j) \otimes \Omega(\Delta^j),
\end{equation}
where $\Omega_{dR}(-)$ refers to the de Rham forms. By analogous arguments as before, we see that there is again a zigzag of natural weak equivalences 
\[
    \Omega^{sm}_{G}(X) \to \bullet \leftarrow 
    \Omega(\sS X \sslash \sS G).
\]

Finally, we will use the notation
\begin{equation}\label{equ:Omega min def}
    \Omega^{minimal}_{G}(X)
   :=
   \int^{[j]\in {\Delta_+}}
   \Omega_{min}(X\times G^j) \otimes \Omega(\Delta^j) \subset \Omega^{PA}_{G}(X),
\end{equation}
to denote the dg commutative subalgebra of (level-wise) minimal forms.

\subsection{A comparison morphism from the Cartan model}
We note that we have a direct Weyl group equivariant comparison morphism
\begin{equation}\label{equ:Phi announce}
\Phi\colon \Omega_T^{Cartan}(X)
\to 
\Omega^{sm}_T(X)
\end{equation}
from the Cartan model of section \ref{sec:compactGrecollection} to the simplicial de Rham model of \eqref{equ:Omega sm def}.
To define $\Phi$, we need some preparation.
First, that elements of $\beta \in \Omega^{sm}_T(X)$ are collections of elements $\beta_n \in \Omega_{dR}(T^n\times X) \otimes \Omega(\Delta^n)$ satisfying the compatibility relations 
\begin{equation}\label{equ:end rel}
d_\mu\beta_n  = b_\mu \beta_{n+1},    
\end{equation}
where $d_\mu$ are the cosimplicial coboundary maps using the product on and action of $T$, and $b_\mu$ are the simplicial boundary maps that act by restricting the differential form on the simplex to the $\mu$-th face of the simplex.
Let us introduce the coordinates $\phi_{ik}$ on the $i$-circle of the $k$-th copy of $T$ in $T^n\times X$. Also, let $0\leq t_1\leq \cdots\leq t_n\leq 1$ be the standard coordinates on the simplex $\Delta^j$. 
We may also use the isomorphism
\[
    \Omega_{dR}(T^n\times X)
    \cong 
    \Omega_{dR}(T^{n+1}\times X)_{T\mathrm{-basic}}
\]
where the right-hand dgca is the sub-dgca
\[
    \Omega_{dR}(T^{n+1}\times X)_{T\mathrm{-basic}}
    \subset \Omega_{dR}(T^{n+1}\times X)
\]
consisting of the differential forms that are basic for the action of $T$, diagonally on the last $T$-factor and $X$. 
Here basic means that these forms are $T$-invariant and the contraction with any of the vector fields generating the $T$-action is zero. 
Equivalently, the basic forms are the ones pulled back from the $T$-quotient.

We define the differential forms (connection forms)
\[
\eta_{in} =\sum_{j=1}^{n+1} t_j d\phi_{ij} \in 
 \Omega_{dR}(T^{n+1}\times X)\otimes \Omega(\Delta^n)
\]
with $t_{n+1}:=1$ and $\phi_{i(n+1)}$ the coordinates on the last $T$-factor. Furthermore, let
\[
\tilde u_{in} := d\eta_{in}
=
 \sum_{j=1}^n dt_j d\phi_{ij}.
\]
The forms $\tilde u_{in}$ are $T$-basic and satisfy the relations \eqref{equ:end rel} and hence assemble into an element 
\[
\tilde u_i \in \Omega_T^{sm}(X).
\]

Now let $\alpha\in \Omega_{dR}(X)^T\subset \Omega_T^{Cartan}(X)$ be a $T$-invariant de Rham form.
Then we define $\tilde\alpha\in \Omega^{sm}_T(X)$ such that 
\[
\tilde\alpha_n =
\left(\prod_{i=1}^r(1- \eta_{in} \iota_{e_{i}} ) \right) \alpha
=
\exp\left(-\sum_{i=1}^r \eta_{in} \iota_{e_{i}}\right)\alpha.
\]
One easily checks that indeed the right-hand side is basic under the additional $T$-action and that the compatibility relations \eqref{equ:end rel} are satisfied.
Finally, we define our desired morphism $\Phi$ of \eqref{equ:Phi announce} such that  

\begin{align}\label{equ:Phipdef}
    \Phi(\alpha) &:= \tilde \alpha
    &
    \Phi(u_i) &:= \tilde u_i.
\end{align}

It is clear that $\Phi$ respects products since the operators $\eta_{in} \iota_{e_i}$ are derivations.
Furthermore, $\Phi$ respects the differentials since
\[
  d\Phi(\alpha)
  =
  d \exp\left(-\sum_{i=1}^r \eta_{in} \iota_{e_{i}}\right)\alpha
  =
  \exp\left(-\sum_{i=1}^r \eta_{in} \iota_{e_{i}}\right)
  \left( d\alpha
  -
  \sum_{i=1}^r \tilde u_i \iota_{e_{i}}
  \right)
  =
  \Phi(d_u\alpha).
\]


\subsection{Cartan model in the PA setting}\label{sec:PA Cartan model}
In the relevant situation for this paper $G=\SO(n)$
we would much prefer to work with the small Cartan models of section \ref{sec:compactGrecollection}, rather than the unwieldy simplicial models of subsection \ref{sec:simpl models BG}.
However, for technical reasons apparent later we are forced to work in the semi-algebraic setting, with PA forms instead of smooth \cite{HLTV}.
Unfortunately, for such forms the definition of the Cartan model does not readily carry over since the contraction operators $\iota_{e_j}$ of \eqref{equ:dudef} are a priori not defined on the PA forms.
We will hence resort to a workaround, that will allow us to work with small ``models'' in practice nevertheless, but is arguably somewhat unsatisfying conceptually.

To this end, let $\Omega_{min}^{sm}(X) \subset \Omega_{min}(X)$ be the dg commutative subalgebra of smooth minimal forms.
Then we define the dgca
\[
\Omega_{G,min}^{Cartan}(X) := (\R[u_1,\dots,u_r]\otimes \Omega_{min}^{sm}(X))^N
\]
with the differential \eqref{equ:dudef}.
Generally, $\Omega_{G,min}^{Cartan}(X)$ is not expected to be quasi-isomorphic to $\Omega(X\sslash G)$. However, the comparison morphism $\Phi$ is defined using only algebraic operations and hence the same formula \eqref{equ:Phipdef} gives rise to a dgca morphism
\begin{equation}\label{equ:Phi announce min}
    \Phi\colon \Omega_{G,min}^{Cartan}(X)
    \to 
    \Omega^{minimal}_T(X)^W=:\Omega^{min}_G(X).
\end{equation}

A particular special case is $X=*$, for which we obtain the morphisms 
\begin{equation}\label{equ:Phi point}
  \Phi\colon \Omega_{G,min}^{Cartan}(*)
  = H(BG)\cong\R[u_1,\dots, u_r]^W
  \to 
  \Omega^{min}_G(*) \hookrightarrow  \Omega^{PA}_G(*),
\end{equation}
realizing a map from the cohomology of $BG$ into the PA model of equivariant forms of a point.

\newcommand{\HK}{\tilde K}

\section{Equivariant Kontsevich morphism}
\label{sec:graphs}
The goal of this section is to construct an equivariant model for the little disks operads, using diagrams.
The construction is essentially merely the equivariant version of a construction employed by Kontsevich \cite{K2} in order to show the real formality of these operads. Hence we shall begin by recalling Kontsevich's formality proof.

\subsection{Kontsevich's proof of real formality of \texorpdfstring{$\lD_n$}{Dn}}
M. Kontsevich showed in \cite{K2} that the operads of real chains on the little disks operads are formal.
Some of the steps and underlying technicalities where however only sketched in his paper and later developed more carefully by Hardt, Lambrechts, Voli\'c, and Turchin \cite{LV,HLTV}. 
The main step of the proof is to construct a quasi-isomorphism 
\begin{equation}\label{equ:Kmap}
 \Graphs_n \to \Omega_{PA}(\FM_n)
\end{equation}
between the graphical cooperad $\stG_n$ introduced above and the PA forms \cite{HLTV} on the Fulton-MacPherson-Axelrod-Singer compactification of the moduli space of points on $\R^n$ introduced in \cite{GJ}.
Before recalling the definition of the map above, let us recall some details on the topological operad $\FM_n$.
Let $\Conf_{N}(R^n)$ be the space of configurations of $N$ distinguishable points on $\R^n$. It is acted upon freely by the group $\R_{>0}\ltimes \R^n$ by scaling and translation.
The spaces $\FM_n(N)$ are compactifications (iterated real bordifications) of the quotient space under this action.
\[
 \FM_n = \overline{ \Conf_N(\R^n)/(\R_{>0}\ltimes \R^n)}
\]
Concretely, the compactification is defined such that the $\FM_n$ as an operad in sets rather than spaces is the free operad generated by $\Conf_N(\R^n)/(\R_{>0}\ltimes \R^n)$. From this description the definition of the operadic composition in $\FM_n$ is also obvious. The topological operad $\FM_n$ is homotopic to the little $n$-disks operad $\lD_n$.
For more details on the definition we refer the reader to \cite{GJ} or \cite{Si}.

Now let us turn to the definition of Kontsevich's map \eqref{equ:Kmap}. 
For a graph $\Gamma\in \Graphs_n(N)$ with $k$ internal vertices the map is defined by the formula 
\begin{equation}\label{equ:Kintegral}
 \Gamma \mapsto \omega_\Gamma:= \int_f \bigwedge_{(i,j)\text{ edge}} \pi_{ij}^* \Omega_{S^{n-1}} \in \Omega_{PA}(\FM_n(N))
\end{equation}
where 
\[
\pi_{ij}: \FM_n(N+k) \to \FM_{n}(2)=S^{n-1}
\]
is the forgetful map forgetting all vertices in a configuration except for the $i$-th and $j$-th, 
$\Omega_{S^{n-1}}\in \Omega_{min}(S^{n-1})$ is the round volume form and the integral is over the fiber of the SA bundle 
\[
 \FM_n(N+k) \to \FM_n(N),
\]
see also section \ref{sec:PAforms}.
The fiber integral does in general not produce a smooth differential form, and that is the reason why one has to work with PA forms instead of smooth forms.
It can be checked by using Stokes' Theorem that the map \eqref{equ:Kmap} respects the differentials and is compatible with the cooperad structure on $\Graphs_n$ and the operadic composition on $\FM_n$ in a natural way. It is furthermore a quasi-isomorphism.

The purpose of the rest of this section is to construct an equivariant version of the Kontsevich map \eqref{equ:Kmap}. Naively speaking this may be done by simply replacing PA forms by equivariant PA forms, while essentially retaining the formula \eqref{equ:Kintegral}, which, in its equivariant form, will re-appear as \eqref{equ:Kintegralequiv} below. However, in practice various steps of the proof that the map \eqref{equ:Kmap} is compatible the differential and cooperad structure will (at least naively) fail in the equivariant setting, the ``defects'' accounting exactly for the rational nontriviality of the action of $\SO(n)$ on $\FM_n$.

\begin{rem}
 Note that a priori the formula \eqref{equ:Kintegral} is defined without restrictions on the arity of vertices in the graph $\Gamma$. In particular, it in fact defines a map (and a quasi-isomorphism)
 \[
  \stG_n^2\to \Omega_{PA}(\FM_n).
 \]
As part of Kontsevich's construction of \eqref{equ:Kmap} one then has to check that this map indeed factors through the quotient cooperad $\stG_n\leftarrow \stG_n^2$. In other words, one has to check that the integrals corresponding to graphs with bivalent internal vertices vanish.
In fact, it turns out to be sufficient to check that for the graph
\[
 \Gamma = 
 \begin{tikzpicture}[baseline=-.65ex]
  \node[ext] (v) at (0,0) {};
  \node[int] (x) at (0.5,0) {};
  \node[ext] (w) at (1,0) {};
  \draw (x) edge (v) edge (w);
 \end{tikzpicture}
\]
we obtain $\omega_\Gamma=0$, which was shown by Kontsevich, cf. also Lemma \ref{lem:bivalentvanish} in the Appendix.
\end{rem}

\subsection{A propagator}\label{sec:propagator}
Let $T$ be the usual maximal torus of $\SO(n)$ and let $N=N(T)$ be its normalizer.
We choose a smooth algebraic equivariant differential form $\Omega_{sm}\in \Omega_{\SO(n),min}^{Cartan}(S^{n-1})$ on the $(n-1)$-sphere with the following properties:
\begin{enumerate}
\item $\Omega_{sm}$ is of degree $n-1$.
\item The image of $\Omega_{sm}$ in $\Omega^\bullet(S^{n-1})$ obtained by setting all $u_j=0$ is a volume form of area 1.
\item For $n$ even $d_u\Omega_{sm}=-E$ is minus the Euler class in $H(B\SO(n))$, and $d_u\Omega_{sm}=0$ for $n$ odd.
\item If $f:S^{n-1}\to S^{n-1}$ is the inversion (i.e., $f(x)=-x$) then $f^*\Omega_{sm}=(-1)^{n}\Omega_{sm}$.
\item Note that by being in $\Omega_{\SO(n),min}^{Cartan}(S^{n-1})$ the form $\Omega_{sm}$ is required to be invariant under the action of $N$. Furthermore, let us require that it is anti-invariant under the action of $\pi_0(O(n))\cong\Z_2$. In other word $\tau^*\Omega_{sm}=-\Omega_{sm}$ for $\tau\in O(n)$ a reflection that preserves $T$.
\end{enumerate}
We will call this form the (equivariant) propagator.
An explicit formula for $\Omega_{sm}$ is given in the following subsection \ref{sec:explicitpropagator}.

We will also define the element 
\[
 \Omega = \Phi(\Omega_{sm}) \in \Omega_{\SO(n)}^{min}(\FM_n(2))
\]
where $\Phi$ is the map \eqref{equ:Phi announce min}.

\subsubsection{An explicit formula for the propagator}
\label{sec:explicitpropagator}
We parameterize the sphere $S^{n-1}$ by a torus and a simplex as follows.
For $n$ odd we parameterize each hemisphere separately, and get
\begin{gather*}
 \{\pm 1\}\times (S^1)^k \times \Delta_{k} \to S^{n-1}\subset \R^n \\
 (\epsilon, \phi_1,\dots,\phi_k,\sigma_0,\dots,\sigma_{k}) 
\mapsto (\epsilon \sqrt{\sigma_0}, \sqrt{\sigma_1}\cos \phi_1,\sqrt{\sigma_1}\sin\phi_1,\dots , \sqrt{\sigma_k}\cos \phi_k,\sqrt{\sigma_k}\sin\phi_k) 
\end{gather*}
where $k=(n-1)/2$, and where we use the standard coordinates on the simplex $\sigma_0,\dots,\sigma_{k}\geq 0$ such that $\sum_{j=0}^{k}\sigma_j=1$. In the following, we will forget $\epsilon$ and restrict to the upper hemisphere (i.e., $\epsilon=+1$), the formula for the lower hemisphere can then be recovered by reflection anti-invariance.

For $n$ even we use the similar parameterization
\begin{gather*}
(S^1)^{k+1} \times \Delta_{k} \to S^{n-1}\subset \R^n \\
 (\phi_0,\dots,\phi_{k},\sigma_0,\dots,\sigma_{k}) 
\mapsto(\sqrt{\sigma_0}\cos \phi_0,\sqrt{\sigma_0}\sin\phi_0,\dots , \sqrt{\sigma_{k}}\cos \phi_{k},\sqrt{\sigma_{k}}\sin\phi_{k}) 
\end{gather*}
where now $k=n/2-1$.
In the above parameterization the $T=(S^1)^k$-action is obvious.

Again in these coordinates the round volume form on the sphere $S^{n-1}$ has the form
\newcommand{\cE}{{\mathcal{E}}}
\begin{align}
&\left(\frac 1 2\right)^{k-1} \iota_\cE \left( d\sqrt{\sigma_0}\prod_{j=1}^{k} (d\sigma_jd\phi_j ) \right) = \left(\frac 1 2\right)^{k}\frac 1 {\sqrt{\sigma_0}} \prod_{j=1}^k ( d\sigma_jd\phi_j) && \text{$n$ odd} \\
\label{equ:roundvol2}
&\left(\frac 1 2\right)^{k} \iota_\cE \prod_{j=0}^{k} (d\sigma_j d\phi_j )  = \pm \left(\frac 1 2\right)^{k} d\phi_0\cdots d\phi_{k} d\sigma_1\cdots d\sigma_{k}   && \text{$n$ even},
\end{align}
where we choose the orientation on the parameter space such that the above forms are positive, and where $\iota_E$ is the operator of contraction with the Euler vector field
\begin{align*}
  \cE&= \sum_{j=0}^{k} \sigma_j \frac \partial {\partial \sigma_j} .
\end{align*}
Note that the Euler vector field is defined not on the simplex but on the larger space $\hat \Delta_{k} = \{\sigma_0,\dots,\sigma_{k}\geq 0\}\supset \Delta_{k}$, and the notation in \eqref{equ:roundvol2} above shall silently mean the restriction of the stated form on $\hat \Delta_{k}$ to $\Delta_{k}$, {\em after} the contraction of the vector field.

To state the formula for the propagator, let us introduce the following notation. For $K$ a subset of indices we shall set
\begin{align*}
u^K &:= \prod_{j\in K} u_j & 
(d\sigma d\phi )^K &:= \prod_{j\in K}( d\sigma_j d\phi_j).
\end{align*}
Furthermore, we denote the complement of the subset $K$ by $\bar K$, and the number of elements by $|K|$.

The explicit formula for the propagator is then
\begin{align*}
 \Omega_{sm}
&= 
C_n \iota_\cE
\left( 
d\sqrt{\sigma_0}
\sum_{K\subset \{1,\dots,k\}}
 {(|\bar K|-\frac 1 2)!} u^K ( d\sigma d\phi )^{\bar K}
\right)
 && \text{$n$ odd} \\
 \Omega_{sm}
&= 
C_n \iota_\cE
\left(
\sum_{K\subsetneq \{0,\dots,k\}}
 {(|\bar K|-1)!} u^K (d\sigma d\phi )^{\bar K}
\right)
&& \text{$n$ even},
\end{align*}
where $x!:=\Gamma(x+1)$ with $\Gamma$ the Euler $\Gamma$-function, and $C_n$ is an unimportant normalization constant chosen such that the integral over the sphere of the above form is $1$. Concretely,
\begin{align*}
C_n = \frac{1}{2^{n/2-1} \Gamma(n/2) \mathit{vol}(S^{n-1})} = 
\begin{cases}
\frac{1}{\sqrt{\pi}(2\pi)^k } & \text{$n$ odd} \\ 
\frac{1}{(2\pi)^{k+1} } & \text{$n$ even} 
\end{cases}.
\end{align*}

\begin{lemma}\label{lem:prop properties}
 The above propagator is well defined and non-singular on the sphere, and satisfies the following conditions:
 \begin{enumerate} 
  \item For $f:S^{n-1}\to S^{n-1}$ the inversion $f(x)=-x$ we have $f^*\Omega_{sm}=(-1)^{n}\Omega_{sm}$.
  \item For $\tau\in O(n)$ a(ny) reflection that preserves the maximal torus we have that $\tau^*\Omega_{sm}=-\Omega_{sm}$.
  \item We have the following formulas for the equivariant differential:
\begin{align}
\label{equ:equivdprop1}
  (d-\sum_{i=1}^k u_i\iota_i) \Omega_{sm} &= 0 & \text{$n$ odd} \\
  \label{equ:equivdprop2}
(d-\sum_{i=0}^{k} u_i\iota_i) \Omega_{sm} &= -C_n u_0\cdots u_{k} =: -E 
& \text{$n$ even.}
\end{align}
\end{enumerate}
\end{lemma}
\begin{proof}
The above form is obviously smooth away from the singular loci of our parameterization, which are the union of the sets $\{\sigma_j=0\}$. The functions $\sigma_j$ are smooth functions on the sphere, and hence are the forms $d\sigma_j$.
The forms $d\phi_j$ has a singularity at $\{\sigma_j=0\}$, however one easily checks that the combinations $\sigma_jd\phi_j$ and $d\sigma_jd\phi_j$ occurring in our formula are smooth forms on the sphere.
Hence the only possible source of a singularity stems from the power of $\sigma_0$ in the formula for $n$ odd.
However, $\sqrt{\sigma_0}$ is one of the Euclidean coordinate functions on the sphere, and hence smooth, and hence so are all of its non-negative powers and the differential $d\sqrt{\sigma_0}$. 

To check property (2) take for $\tau$ the reflection along any coordinate axis, say that one sending $\phi_1\to -\phi_1$ in our parameterization. We then have that $\tau^* u_1=-u_1$ and $\tau^*u_j=u_j$ for $j\geq 2$ so that clearly $\tau^*\Omega_{sm}=-\Omega_{sm}$. Property (1) follows from (2) since the map $f$ is the composition of the $n$ reflections along all coordinate axis.

Finally, let us consider the stated formulas for the equivariant differentials.
Consider first the case of even $n$, for which we compute:
\begin{align*}
d\Omega_{sm} &=C_n  L_\cE
\sum_{K\subsetneq \{0,\dots,k\}}
 {(|\bar K|-1)!} u^K (d\sigma d\phi )^{\bar K}
= C_n \sum_{K\subsetneq \{0,\dots,k\}}
{|\bar K|!} u^K (d\sigma d\phi )^{\bar K}
\\
\sum_{j=0}^{k} u_j \iota_j \Omega_{sm}
&=
C_n
\iota_\cE
\sum_{K\subsetneq \{0,\dots,k\}}
 {(|\bar K|-1)!} u^K \sum_{i\in \bar K} u_i d\sigma_i ( d\sigma d\phi)^{\bar K\setminus \{i\}}
\end{align*}
In the first line we denoted the Lie derivative with respect to the Euler vector field by $L_\cE$.
Now collect powers of $u$ in the final expression in the second line (i.e., change summation variables $K\mapsto K\cup\{i\}$), yielding
\[
C_n
\iota_\cE
\sum_{\emptyset \neq K\subset \{0,\dots,k\}}
 {|\bar K|!} u^K (d\sigma d\phi )^{\bar K} \sum_{i\in K} d\sigma_i 
=
C_n
\iota_\cE
\sum_{\emptyset \neq K\subset \{0,\dots,k\}}
|\bar K|! u^K (d\sigma d\phi )^{\bar K} \sum_{i=0}^{k} d\sigma_i  \, .
\] 
To simplify further we need to carry out the contraction and obtain
\[
\iota_\cE \left((d\sigma d\phi )^{\bar K} \sum_{i=0}^{k} d\sigma_i\right)
=
(\iota_\cE (d\sigma d\phi )^{\bar K})  \underbrace{\sum_{i=0}^{k} d\sigma_i}_{=0\text{ on }\Delta_{k}}
+
(d\sigma d\phi )^{\bar K} \underbrace{\iota_\cE \sum_{i=0}^{k} d\sigma_i}_{=1\text{ on }\Delta_{k}}
=
(d\sigma d\phi )^{\bar K}.
\]
Collecting the previous computations we find that for even $n$
\begin{align*}
(d-\sum_{i=0}^{k} u_i\iota_i) \Omega_{sm} &=
C_n \sum_{K\subsetneq \{0,\dots,k\}}
{|\bar K|!} u^K (d\phi d\sigma)^{\bar K}
-
C_n \sum_{\emptyset \neq K\subset \{0,\dots,k\}}
{|\bar K|!} u^K (d\phi d\sigma)^{\bar K}
\\&=
-
C_n \, u_0\cdots u_{k},
\end{align*}
and thus \eqref{equ:equivdprop2} is shown. (Here we note that the term $K=\emptyset$ in the first sum does not contribute, since the restriction of that summand to the simplex vanishes.)

Next, let us turn to $n$ odd, and compute similarly:

\begin{align*}
d\Omega_{sm} &=C_n  L_\cE
\left( 
d\sqrt{\sigma_0}
\sum_{K\subset \{1,\dots,k\}}
 {(|\bar K|-\frac 1 2)!} u^K (d\sigma d\phi )^{\bar K}
\right)
= C_n \left( 
d\sqrt{\sigma_0}
\sum_{K\subset \{1,\dots,k\}}
 {(|\bar K|+\frac 1 2)!} u^K (d\sigma d\phi )^{\bar K}
\right)
\\
\sum_{j=1}^{k} u_j \iota_j \Omega_{sm}
&=-
C_n
\iota_\cE
\left( 
d\sqrt{\sigma_0}
\sum_{K\subset \{1,\dots,k\}}
 {(|\bar K|-\frac 1 2)!} u^K \sum_{i\in \bar K} u_i d\sigma_i (d\sigma d\phi )^{\bar K\setminus \{i\}} 
\right)
\end{align*}

Collect again powers of $u$ in the last expression (i.e., change summation variables $K\mapsto K\cup\{i\}$), yielding
\[
-C_n
\iota_\cE
\left( 
d\sqrt{\sigma_0}
\sum_{\emptyset \neq K\subset \{1,\dots,k\}}
 {(|\bar K|+\frac 1 2)!} u^K ( d\sigma d\phi)^{\bar K} \sum_{i\in K}d \sigma_i 
\right)
=
-C_n
\iota_\cE
\left( 
d\sqrt{\sigma_0}
\sum_{\emptyset \neq K\subset \{1,\dots,k\}}
 {(|\bar K|+\frac 1 2)!} u^K ( d\sigma d\phi)^{\bar K} \sum_{i=0}^k d \sigma_i 
\right) \, .
\] 
Now carry out the contraction as for even $n$ and obtain
\[
C_n
\left( 
d\sqrt{\sigma_0}
\sum_{K\subset \{1,\dots,k\}}
 {(|\bar K|+\frac 1 2)!} u^K (d\sigma d\phi )^{\bar K}
\right) \, .
\]
Comparing terms, \eqref{equ:equivdprop1} follows.

\end{proof}

%
%

\begin{rem}\label{rem:prop restrict}
We note that the forms $\Omega_{sm}$ are stable under restriction to the $n-2$-dimensional subspace defined by $\sigma_{k}=0$, in the sense that 
\[
\Omega_{sm}^{n\text{-dim}} \mid_{S^{n-3}} = \frac{ u_{k}}{2\pi} \Omega_{sm}^{n-2\text{-dim}}.
\]
This also implies that for $n-m=2r$ even
\[
  \Omega_{sm}^{n\text{-dim}} \mid_{S^{n-m-1}} = E_{n-m} \Omega_{sm}^{n-m\text{-dim}},
\]
where $E_{n-m}=\frac{u_{k-r+1}\cdots u_{k}}{(2\pi)^{r}}$ is the orthogonal Euler class.
\end{rem}

In order to facilitate explicit computations, let us also note the following.
\begin{lemma}\label{lem:prop at north pole}
 If $n$ is odd the value of $\Omega_{sm}$ at the north pole of the sphere (i.e., at $\sigma_0=1$, $\sigma_1=\cdots=\sigma_k=0$) is 
\[
C_n \frac 1 2 \Gamma(\frac 1 2) u_1\cdots u_k= \frac 1 {2(2\pi)^k} u_1\cdots u_k.
\]\hfill\qed
\end{lemma}

\newcommand{\Grav}{\mathsf{Grav}}
\subsection{Equivariant cohomology of \texorpdfstring{$\FM_n$}{FMn}}
\label{sec:FM equiv cohom}
Let us pause here and evaluate the $\SO(n)$-equivariant cohomology of $\FM_n(r)$.
For the moment, we disregard the operad structure, we care only about the cohomology of the dg vector space of equivariant forms.
This cohomology is easily computed using the smooth Cartan model.
There is an evident spectral sequence whose $E_1$ page reads
\beq{equ:E1equiv}
 E_1 = H(B\SO(n))\otimes H(\FM_n(r)).
\eeq
Recall that by results of F. Cohen \cite{Cohen} the cohomology of $\FM_n(r)$ is described as a commutative algebra by generators and relations as follows:
The generators are (classes represented by) forms 
\[
 \alpha_{ij} = \pi_{ij}^* \Omega_{S^{n-1}},
\]
where $1\leq i\neq j\leq r$ and $\Omega_{S^{n-1}}$ is a form on $S^{n-1}$ generating $H^{n-1}(S^{n-1})$.
The relations are the following
\begin{align*}
 \alpha_{ij}&=(-1)^n\alpha_{ji}
\\
\alpha_{ij}^2&=0
\\
\alpha_{ij}\alpha_{jk}+\alpha_{jk}\alpha_{ki}+\alpha_{ki}\alpha_{ij}&=0.
\end{align*}

Now, if $n$ is odd, all the $\alpha_{ij}$ may in fact be extended to equivariantly closed forms, for example we can take for the extension
\[
 \pi_{ij}^*\Omega_{sm},
\]
where $\Omega_{sm}$ is our propagator from the preceding subsection.
Hence we conclude that for odd $n$ the spectral sequence abuts at this stage.

For even $n$ we may proceed similarly using our propagator to extend the forms, but since $d_u\Omega_{sm}=-E$ the spectral sequence does not abut here.
Rather, defining the operator $T: H(\FM_n(r))\to H(\FM_n(r))$ as
\[
 T=\sum_{i\neq j} \frac{\partial}{\partial \alpha_{ij}},
\]
the next (distinct) page in the spectral sequence reads
\[
 \R[\Pp_4,\cdots,\Pp_{2n-4}]\otimes H( H(\FM_n(r))[E], -ET).
\]
It is known that\footnote{This is for example contained in the statement of \cite[Theorem 2.18]{DCV}.}
\[
\Grav(r) :=\ker(T) \to (H(\FM_n(r))[E], -ET)  
\]
is a quasi-isomorphism for $r\geq 2$. (In fact, $\Grav$ is the gravity operad \cite{Getzler0}.)
Since we now have closed representatives for all remaining classes on the present page our spectral sequence, the spectral sequence abuts here.

Let us summarize our finding.
\begin{prop}\label{prop:FMequiv}
 The $\SO(n)$-equivariant cohomology of $\FM_n(r)$ is 
\[
 \begin{cases}
  H(B\SO(n)) &\text{if $r=1$} \\
  H(B\SO(n))\otimes H(\FM_n(r)) & \text{if $r\geq 2$ and $n$ is odd} \\
  H(B\SO(n-1))\otimes \Grav(r) &\text{if $r\geq 2$ and $n$ is even}
 \end{cases}
\]
\end{prop}

Note also that the explicit representatives we constructed are algebraic, by our choice of $\Omega_{sm}$ as an algebraic form.
The corresponding representatives in $\Omega_K^{s,PA}(\FM_n(r))$ may be obtained by just replacing $\Omega_{sm}$ by $\Omega=\Phi(\Omega_{sm})$.




\subsection{A Maurer-Cartan element}
\label{sec:MC int formula}
The integration for PA forms over the whole of $\FM_n(k)$ yields a map 
\begin{gather*}
  \int_{\FM_n(k)} \colon \Omega_{\SO(n),min}^{Cartan}(\FM_n(k))
  =
  (\R[u_1,\dots,u_r]\otimes \Omega_{min}^{sm}(\FM_n(k)) )^N
  \to H(B\SO(n))= \R[u_1,\dots,u_r]^W \\
  p\otimes \alpha \mapsto p \otimes \int_{\FM_n(k)}\alpha.
\end{gather*}
This in turn can be used to define a map of graded vector spaces
\[
\psi: \sG_n \to H(B\SO(n))
\] 
from the dg Lie coalgebra $\sG_n$ of section \ref{sec:k graphs} by sending a graph $\gamma\in \sG_n$ with $k$ vertices to
\[
  \psi(\gamma) := \int_{\FM_n(k)} \bigwedge_{(i,j)\in E\gamma} \pi_{ij}^* \Omega_{sm}.
\]
Here the propagator $\Omega_{sm}$ is the differential form of the last section, the product is over the edges of $\gamma$, and $\pi_{ij}: \FM_n(|V\gamma|)\to \FM_n(2)=S^{n-1}$ is the projection forgetting the locations of all points in a configuration except for the $i$-th and $j$-th. Note that we require here that $i\neq j$ -- if $\gamma$ has a tadpole (self-edge) then we set $\psi(\gamma)=0$.

Let us denote by $E=-d_u\Omega_{sm}\in H(B\SO(n))$ the Euler class if $n$ is even or 0 if $n$ is odd.
Then the morphism $\gamma$ satisfies the following compatibility relation with the Lie cobracket. Let us denote by $T=\tadpole\cdot$ the action of the tadpole graph on $\Graphs_n$. (Combinatorially, the operator $T$ acts on a graph by summing over all ways of removing one edge.)

\begin{lemma}\label{lem:int compat G}
For any $\gamma\in \sG_n$ we have 
\[
  -E\psi(T \gamma) + \psi(d\gamma) +  \frac 12 \sum \psi(\gamma') \psi(\gamma'') =0,
\]
where we use the Sweedler notation $\gamma\mapsto \sum\gamma'\otimes \gamma''$ for the Lie cobracket in $\sG_n$.
\end{lemma}
\begin{proof}
This follows from the Stokes' formula for the integral \eqref{equ:PA Stokes}.
Concretely, we have 
\[
  -E\psi(T \gamma)
  =
  \int_{\FM_n(k)} d_u \bigwedge_{(i,j)\in E\gamma} \pi_{ij}^* \Omega_{sm}
\]
by definition of $T$ and $d_u\Omega_{sm}=-E$. Next 
\[
\int_{\FM_n(k)} d_u \bigwedge_{(i,j)\in E\gamma} \pi_{ij}^* \Omega_{sm}
=
\int_{\FM_n(k)} d \bigwedge_{(i,j)\in E\gamma} \pi_{ij}^* \Omega_{sm}
\]
because for all differential forms $\alpha$ the contraction $\iota_\xi\alpha$ with a vector field $\xi$ yields a form of less than top degree, which hence vanishes upon integration.
Then, by Stokes' Theorem \eqref{equ:PA Stokes}
\[
  \int_{\FM_n(k)} d \bigwedge_{(i,j)\in E\gamma} \pi_{ij}^* \Omega_{sm}
  =
  \int_{\partial \FM_n(k)} \bigwedge_{(i,j)\in E\gamma} \pi_{ij}^* \Omega_{sm}  
\]
where on the right-hand side we integrate over the codimension one boundary strata of $\FM_n(k)$. Those are in 1-1-correspondence with subsets $A\subset \{1,\dots,k\}$ of the $k$ points, with $2\leq |A|\leq k-1$, that ``collide''. Denoting that boundary stratum by $\partial_A \FM_n(k)$ we have that 
\[
  \partial_A \FM_n(k) \cong \FM_n(k-|A|+1)\times \FM_n(|A|)
\]
and the integral splits into a product of two similar integrals by Fubinis' Theorem.
\[
  \int_{\partial_A \FM_n(k)}  \bigwedge_{(i,j)\in E\gamma} \pi_{ij}^* \Omega_{sm}  
  =
  \pm\left(\int_{\FM_n(k-|A|+1)} 
   \bigwedge_{(i,j)\in E\gamma\atop i\notin A \text{ or } j\notin A} \pi_{ij}^* \Omega_{sm} \right)
  \left(\int_{\FM_n(|A|)} 
   \bigwedge_{(i,j)\in E\gamma \atop i,j\in A} \pi_{ij}^* \Omega_{sm} \right)
   =
   \pm\psi(\gamma')\psi(\gamma'').
\]
In the last line we denote by $\gamma''\subset \gamma$ the full subgraph with vertex set $A$ and $\gamma':=\gamma/\gamma''$ is obtained by contracting $\gamma''$.
Now if $|A|=2$ and $\gamma''$ contains exactly one edge, then corresponding terms produce the differential (edge contraction).
If otherwise $\gamma''$ is not connected or has vertices of valence $\leq 2$, then $\psi(\gamma'')=0$ as shown in Appendix \ref{app:vanishing}.
The remaining terms are those occurring in the Lie cobracket.
Determining the correct signs is unfortunately somewhat tedious, and they depend on conventions such as the orientation of the underlying spaces, that we suppress here.
We just mention that we follow the conventions of \cite{LV}, where a detailed discussions of signs can be found.
\end{proof}

The correct way to interpret Lemma \ref{lem:int compat G} is that the element dual to $\psi$ defines a Maurer-Cartan element.
More precisely, we define a Maurer-Cartan element $\tm^n\in \GC_n^+\hat \otimes H(B\SO(n))$ by the formula
\begin{equation}\label{equ:MCelement}
  \tm^n = -E \tadpole + \sum_{\gamma} \gamma^* \psi(\gamma),
\end{equation}
where the sum is over graphs $\gamma$ forming a basis of $\stGC_n$, while $\gamma^*$ are the dual basis elements in $\GC_n$. 
The immediate corollary of Lemma \ref{lem:int compat G} is then:
\begin{cor}\label{cor:mMC}
The element $\tm^n$ is a Maurer-Cartan element.
\end{cor}

We claim that the gauge equivalence class of $\tm^n$ completely characterizes the (real) homotopy type of the action of $\SO(n)$ on $\FM_n$. To see this, we will use $\tm^n$ to build a model for the $\SO(n)$-equivariant differential forms in the next section.

Before we do this, let us also define a minor variation $\tilde \tm^n\in \GC_n^+\hat \otimes \Omega_{\SO(n)}^{PA}(*)$ of the Maurer-Cartan element $\tm^n$. The difference is that we need to use as the coefficient ring our (large) simplicial model $\Omega_{\SO(n)}^{PA}(*)$ (see \eqref{equ:om G pa def}) for the differential forms on $B\SO(n)$, instead of its cohomology ring $H(B\SO(n))$. 
\begin{equation}\label{equ:MCelement2}
 \tilde \tm^n = -E \tadpole + \sum_{\gamma} \gamma^* \int_{\FM_n(|V\gamma|)} \bigwedge_{(i,j)\in E\gamma} \pi_{ij}^* \Omega.
\end{equation}
This element is defined in the same way as $m$ before, except that one uses the propagator $\Omega$ of section \ref{sec:propagator} instead of the "Cartan" version $\Omega_{sm}$.
However, one checks that the two elements $\tm^n$, $\tilde \tm^n$ are in fact identified via the comparison morphism $\Phi$.
\begin{lemma}\label{lem:m to OBG}
 The element $\tilde \tm^n$ is a Maurer-Cartan element.
 It is the image of $\tm^n$ under the map of dg Lie algebras $\GC_n^+\hat \otimes H(B\SO(n))\to \GC_n^+\hat \otimes \Omega_{\SO(n)}^{PA}(*)$ induced by the map \eqref{equ:Phi point}.
\end{lemma}
\begin{proof}
 The first statement clearly follows from the second and Corollary \ref{cor:mMC}.
 To see the second statement, denote by $I_{sm}$ and $I$ the two integrands appearing in \eqref{equ:MCelement} and \eqref{equ:MCelement2}. Then by \eqref{equ:Phipdef} the integrands differ only by contractions with vector fields on $\FM_n$, i.e., $I=\Phi(I_{sm})$ is the same as $\left(\prod_{j=1}^r(1+ \eta_j \otimes \iota_{\xi_{j}} ) \right) I_{sm}$ up to an identification of basic forms with forms on the quotient. In particular, the contractions necessarily produce forms that are not of top degree along $\FM_n$, and hence do not contribute to the integral. Hence the only surviving terms in the integrals in \eqref{equ:MCelement2} are those already present in \eqref{equ:MCelement}.
\end{proof}

\subsection{A model for the equivariant forms on \texorpdfstring{$\FM_n$}{FMn}}
Recall from section \ref{sec:k graphs} that the dg Lie algebra $\GC_n^+$ acts on the Hopf cooperad $\stG_n$, and hence the dg Lie algebra 
\[
\BGC_n := \GC_n^+\hat \otimes H(B\SO(n))
\]
acts on the Hopf cooperad $\BGraphs_n = \stG_n \otimes H(B\SO(n))$.
Given the Maurer-Cartan element $\tm^n\in \BGC_n$ as above we consider the twist of $\BstG_n$ to a Hopf cooperad $\BstG_n^{Z^n}$ under $H(B\SO(n))$.

We claim that this is a a dg Hopf cooperad model for $\FM_n\sslash\SO(n)$.
To show this, we first apply the fiber integral of minimal forms 
\[
  \int_f\colon 
\Omega_{min}(T^k\times \FM_{n}(k+r))
\to   
\Omega_{PA}(T^k\times \FM_{n}(r))
\] 
level-wise to define a morphism of graded vector spaces
\begin{gather*}
  \int_f\colon 
\Omega^{min}_{\SO(n)}(\FM_n(k+r))
\to 
\Omega^{PA}_{\SO(n)}(\FM_n(r))\\
(\alpha_n\otimes \beta_n)_n
\mapsto 
( (\int_f\alpha_n)\otimes \beta_n)_n
.
\end{gather*}

Then we define our morphism
\begin{align*}
  \omega\colon \BstG_n^{\tm^n} &\to \Omega_{\SO(n)}^{PA}(\FM_n) \\
  \Gamma &\mapsto \omega_\Gamma
\end{align*}
by assigning to a graph $\Gamma\in \stG_n(N)$ with $k$ internal vertices the element 
\begin{equation}\label{equ:Kintegralequiv}
 \omega(\Gamma)=\omega_\Gamma = \int_f \bigwedge_{(i,j)\in E\Gamma} \pi_{ij}^* \Omega
\end{equation}
where $\Omega$ is the propagator from above and the product is again over edges.

We note that this map is well defined, in the sense that it vanishes on graphs with bivalent internal vertices by Lemma \ref{lem:bivalentvanish}.
Furthermore, we let $\Z_2=\pi_0(O(n))$ act on $\Graphs_n$ by multiplying a graph $\Gamma$ with $k$ internal vertices and $e$ edges by $(-1)^{k-e}$.
This action readily extends to $\BstG_n$, and it is elementary to check that the map $\omega$ is $\Z_2$-equivariant, given that our propagator is reflection anti-invariant.

\begin{thm}\label{thm:equivariant model}
The map $\omega$ above realizes $\BstG_n^{\tm^n}$ as an equivariant model for $\FM_n$ as an operad in $\SO(n)$-spaces.
 Concretely, $\omega$ respects the differentials, the (co)operad structure, the map from $H(B\SO(n))$ and is a quasi-isomorphism. 
 It also respects the $\Z_2$ action.
 \end{thm}
\begin{proof}
The proof is closely analogous to proof of the corresponding non-equivariant statement by Kontsevich \cite{K2} and Lambrechts-Voli\'c \cite{LV} -- essentially one just has to replace the non-equivariant propagator of op. cit. by our equivariant propagator $\Omega$.
This change does not affect signs and orientations, so we shall not discuss these in detail and just refer to \cite{LV}, whose conventions we follow.

To check that the map $\omega$ commutes with the differentials, one applies Stokes' Theorem for PA forms.
\begin{equation}\label{equ:tempdF}
 d\omega(\Gamma) = \int_{\partial f} \bigwedge_{(i,j)} \pi_{ij}^* \Omega
+
\int_{f} d\bigwedge_{(i,j)} \pi_{ij}^* \Omega
\end{equation}
where the first integral is over the fiberwise boundary. Again as in Kontsevich's proof the fiberwise boundary consists of several strata corresponding to bunches of points colliding.
Now, however, the integrals associated to these strata do not vanish. Rather, they produce precisely the terms of the Maurer-Cartan element $\tilde \tm^n$, except for the term $\tm_0:=-E\tadpole$. Using 
Lemma \ref{lem:m to OBG} these terms are accounted for by taking the twist with $\tm^n-\tm_0$ in 
\beq{equ:tempBstGm}
\BstG_n^{\tm^n}=(\BstG_n^{\tm_0})^{\tm^n-\tm_0}.
\eeq

Next, the second term of \eqref{equ:tempdF} can be simplified as follows:
\begin{align*}
 d\prod_{(i,j)} \pi_{ij}^* \Omega
&= 
 d\prod_{(i,j)} \pi_{ij}^* \Phi(\Omega_{sm})
 = d\Phi( \prod_{(i,j)} \pi_{ij}^* \Omega_{sm})
 = \Phi( d_u \prod_{(i,j)} \pi_{ij}^* \Omega_{sm}).
\end{align*}
For $n$ odd we have that $d_u\Omega_{sm}=0$ and the above expression vanishes. For $n$ even we have $d_u\Omega_{sm}=-E$ and the expression becomes 
\begin{align*}
-\sum_{e=(p,q)} (-1)^e  E \wedge \prod_{(i,j)\neq e} \pi_{ij}^* \Phi(\Omega_{sm}).
\end{align*}
Here we sum over edges $e=(p,q)$ in our graph $\Gamma$, and we set $(-1)^e$ to be 1 for the first edge in the ordering $-1$ for the second etc.
Inserting back into \eqref{equ:tempdF}, the second term of that equation may be identified with 
\[
-\sum_{e=(p,q)} (-1)^e E \omega(\Gamma-e) = \omega(\tm_0\cdot \Gamma).
\]
Hence this term reproduces precisely the twist by $\tm_0$ in \eqref{equ:tempBstGm}.

Finally we claim that the map $\omega$ is a quasi-isomorphism. Indeed, recall the computation of the equivariant cohomology of $\FM_n$ from Proposition \ref{prop:FMequiv}.
On the other hand, we may compute the cohomology of $\BstG_n^{\tm^n}$ by using the spectral sequence on the ``number of $u$'s''.
The first convergent is 
\[
 H(B\SO(n))\otimes H(\Graphs_n).
\]
Using that $H(\Graphs_n)\cong H(\FM_n)$ this agrees with \eqref{equ:E1equiv}.
Furthermore, one immediately checks that the further pages of the spectral sequence agree, so that indeed $H(\BstG_n^{\tm^n})\cong H_{\SO(n)}(\FM_n)$.
Finally, it is clear from looking at the representatives of the cohomology of both sides that $\omega$ induces an isomorphism on cohomology.
\end{proof}

\begin{rem}
Note that by its construction the dg Hopf cooperad $\BstG_n^{\tm^n}$ (in $\dgca^{H(B\SO(n))/}$) is acted upon by the twisted dg Lie algebra
\[
(\GC_n\hat\otimes H(B\SO(n)))^{\tm^n}.
\]
If one wants to preserve also the $\Z_2$-module structure on $\BstG_n^{\tm^n}$, one has to restrict to the $\Z_2$-invariant dg Lie subalgebra 
\[
\left( (\GC_n\hat\otimes H(B\SO(n)))^{\tm^n} \right)^{\Z_2}.
\]
\end{rem}

\section{The Maurer-Cartan element \texorpdfstring{$\tm^n$}{Z}}\label{sec:MC}
Above we have seen that the study of the real homotopy type of the $\SO(n)$-action boils down to understanding the Maurer-Cartan element $\tm^n\in \BGC_n$. This section is hence devoted to studying the gauge equivalence class of $\tm^n$. 
In fact, we will see that $\tm^n$ is gauge equivalent to a quite trivial graphical Maurer-Cartan element.

\begin{thm}\label{conjthm:main}
\begin{itemize}
 \item For $n\geq 2$ even, the Maurer-Cartan element $\tm^n$ is gauge equivalent to $-E\tadpole$, where $E\in H(B\SO(n))$ is the Euler class.
 \item For $n\geq 3$ odd, the Maurer-Cartan element $\tm^n$ is gauge equivalent to 
 \begin{equation}\label{equ:conjectured m odd}
 \sum_{j\geq 1}
 \frac{\Pp_{2n-2}^{j}}{4^j}
\frac{1}{2(2j+1)!} 
\begin{tikzpicture}[baseline=-.65ex]
 \node[int] (v) at (0,.5) {};
 \node[int] (w) at (0,-0.5) {};
 \draw (v) edge[bend left=50] (w) edge[bend right=50] (w) edge[bend left=30] (w) edge[bend right=30] (w);
 \node at (2,0) {($2j+1$ edges)};
 \node at (0,0) {$\scriptstyle\cdots$};
\end{tikzpicture}
 \end{equation}
 with $\Pp_{2n-2}\in H(B\SO(n))$ the top Pontryagin class.
 \end{itemize}
\end{thm}
In the above constructions we may replace the group $\SO(n)$ by subgroups $G=\SO(m)$ or $G=\SO(n-m)\times \SO(m)$.
The same conclusions will apply, except that the ring $H(B\SO(n))$ is replaced everywhere by its quotient $H(B\SO(m))$ or $H(B\SO(n-m))\otimes H(B\SO(m))$.
Let us call $\tm^n_G\in \GC_n^+\hotimes H(BG)$ the Maurer-Cartan element obtained by projecting the coefficient ring via the map $H(B\SO(n))\to H(BG)$. This Maurer-Cartan element governs the $G$-equivariant homotopy type of $\FM_n$. 
 In particular, let us assemble some special consequences of the previous Theorem in the following result.
 \begin{cor}\label{conjthm:main2}
\begin{itemize}
 \item For $n$ even and $G=\SO(n-1)$, $\tm^n_G$ is gauge trivial.
 \item For $n$ odd and $G= \SO(n-2)$, $\tm^n_G$ is gauge trivial.
 \item For $n$ odd and $G= \SO(n-1)$, $\tm^n_G$ is gauge equivalent to 
  \[
 \sum_{j\geq 1}
 \frac{E^{2j}}{4^j}
\frac{1}{2(2j+1)!} 
\begin{tikzpicture}[baseline=-.65ex]
 \node[int] (v) at (0,.5) {};
 \node[int] (w) at (0,-0.5) {};
 \draw (v) edge[bend left=50] (w) edge[bend right=50] (w) edge[bend left=30] (w) edge[bend right=30] (w);
 \node at (2,0) {($2j+1$ edges)};
 \node at (0,0) {$\scriptstyle\cdots$};
\end{tikzpicture}
 \]
 where $E\in H(BG)$ is the Euler class.
\end{itemize}
\end{cor}

We will prove Theorem \ref{conjthm:main} in several steps.
Throughout this section let us use the following notation:

\begin{itemize}
\item We abbreviate $Z_m^n:=Z_{\SO(m)}^n$ and $Z_{k,l}^n:=Z_{\SO(k)\times \SO(l)}^n$ for $k+l\leq n$.
\item We denote the Maurer-Cartan elements appearing in \ref{conjthm:main} by
\[
Z^n_{conj} = -E_n\tadpole.
\]
for $n$ even and 
 \begin{equation}\label{equ:Zn conj odd}
 Z^n_{conj} =
\sum_{j\geq 1}
 \frac{P_{2n-2}^{j}}{4^j}
\frac{1}{2(2j+1)!} 
\begin{tikzpicture}[baseline=-.65ex]
 \node[int] (v) at (0,.5) {};
 \node[int] (w) at (0,-0.5) {};
 \draw (v) edge[bend left=50] (w) edge[bend right=50] (w) edge[bend left=30] (w) edge[bend right=30] (w);
 \node at (2,0) {($2j+1$ edges)};
 \node at (0,0) {$\scriptstyle\cdots$};
\end{tikzpicture}
 \end{equation}
for $n$ odd.
\end{itemize}

\subsection{Explicit computations of leading order terms}
We can use the explicit integral formulas \eqref{equ:MCelement} to understand the leading order terms of the MC element $\tm^n$.
Concretely, for $n=2k+1$ odd one may compute the coefficients of the graphs of the form 
\[
\begin{tikzpicture}[baseline=-.65ex]
 \node[int] (v) at (0,.5) {};
 \node[int] (w) at (0,-0.5) {};
 \draw (v) edge[bend left=50] (w) edge[bend right=50] (w) edge[bend left=30] (w) edge[bend right=30] (w);
 \node at (2,0) {($2r+1$ edges)};
 \node at (0,0) {$\scriptstyle\cdots$};
\end{tikzpicture}.
\]
\begin{lemma}\label{lem:leading int eval}
 The integral weight of the above graph is
\[
 \left(\frac{u_1\cdots u_k}{2(2\pi)^k}\right)^{2r} =: \frac{\Pp_{2n-2}^r}{4^r}.
\]
Hence the coefficient of the same graph in $\tm^n$ (as in \eqref{equ:MCelement}) is
\[
 \frac 1 {2(2r+1)!} \frac{\Pp_{2n-2}^r}{4^r}
\]

\end{lemma}
\begin{proof}
 The integral weight is the integral appearing in \eqref{equ:MCelement}.
In our case this integral takes the form
\[
 \int_{S^{2k}} \Omega_{sm}^{2r+1}.
\]
It is an integral of an equivariantly closed form over a manifold without boundary.
Hence we may use the Berline-Vergne equivariant localization formula (see \cite[Theorem 46]{Libine}) to evaluate the integral.
The fixed point set of the torus action consists of two points, the north and south pole of the sphere. 
By symmetry, both points contribute the same value in the localization formula. Denoting the north pole by $N$ temporarily, the integral hence evaluates to
\[
 2 \frac{(2\pi)^k}{u_1\cdots u_k} (\Omega_{sm}(N))^{2r+1},
\]
where the $2$ accounts for the contribution of the south pole and the remaining prefactor comes from the localization formula.
Using now Lemma \ref{lem:prop at north pole}, we evaluate the expression to
\[
 2 \frac{(2\pi)^k}{u_1\cdots u_k} \left(\frac {u_1\cdots u_k} {2(2\pi)^k}\right)^{2r+1}
=
\left(\frac {u_1\cdots u_k} {2(2\pi)^k}\right)^{2r}
=
\frac{\Pp_{2n-2}^r}{4^r},
\]
where the top Pontryagin class is defined as 
\[
 \Pp_{2n-2} := \frac {(u_1\cdots u_k)^2} {(2\pi)^k}.
\]
This immediately yields the coeffient of that graph in the formula for $\tm$, which differs only by a conventional combinatorial prefactor, which is the size of the symmetry group of the graph.
\end{proof}

\begin{rem}
 Let us quickly comment on the somewhat ``strange'' combinatorial prefactor occurring in the Lemma.
Note that in sum-of-graphs formulas such as \eqref{equ:MCelement} there appears over basis elements $\gamma$ of a space of graphs, and the corresponding dual elements $\gamma^*$ in the dual graph space.
Now, spaces of linear combinations of graphs come with a natural basis, given by (individual) graphs, and hence so do their dual spaces.
However, conventionally, in the identification of a graph as an element of the primal space, or as an element of the dual space, one often introduces a conventional combinatorial prefactor of size the order of the symmetry group of the graph.
This makes formulas for the differential and bracket in the dual complex more pretty.
We note however that this prefactor is purely conventional and could be absorbed in different conventions.
\end{rem}

\subsection{Proof of Theorem \ref{conjthm:main} for \texorpdfstring{$n=2$}{n=2} and \texorpdfstring{$n=3$}{n=3}}

\begin{prop}\label{prop:conjmain23}
Theorem \ref{conjthm:main} holds for $n=2$  and $n=3$.
\end{prop}
\begin{proof}[Proof of Proposition \ref{prop:conjmain23} for $n=2$] 
The result is well known for $n=2$, see also \cite{pavolfr, GS}. In this case all integrals vanish by the Kontsevich vanishing Lemma \cite[Lemma 6.4]{K1}.
\end{proof}

For $n=3$ we already know from the discussions for general $n$ above that the Maurer-Cartan element $Z^3\in \GC_3^{\Z_2}[[u]]$ lives in the the $\Z_2$-invariant subspace of $\GC_3[[u]]$, with $u=P_4$ temporarily abbreviating the Pontryagin class of degree 4. Note that the action of $\Z_2=\pi_0(O(3))$ on the Pontryagin class is trivial, and $\GC_3^{\Z_2}\subset \GC_3$ is the dg Lie subalgebra formed by graphs of even loop order. We furthermore know by Lemma \ref{lem:leading int eval} that the Maurer-Cartan element $\tm^3$ is a deformation of the conjectured one $\tm^3_{conj}$ above,
\begin{equation}\label{equ:Z3 form}
\tm^3=\tm^3_{conj} +z_4+z_6+\cdots,
\end{equation}
where $z_k\in \GC_3^{\Z_2}[[u]]$ is some series of graphs with exactly $k$ vertices. Note here that the number of vertices of a graph of even loop order and even cohomological degree must necessarily be even, hence there is no term $z_3,z_5$ etc. We then desire to show Proposition \ref{prop:conjmain23} by an obstruction theoretic argument, for which we need the following auxiliary result:

\begin{lemma}\label{lem:cohom van}
 We have the vanishing of the first cohomology
 \[
 H^1\left( (\GC_3^{\Z_2}[[u]])^{Z^n_{conj}} \right)_{\geq 3 \text{ vertices}} =0, 
 \]
 where the subscript means that we restrict to the graded pieces of the cohomology corresponding to graphs with $\geq 3$ vertices.
\end{lemma}

\begin{proof}[Proof of Proposition \ref{prop:conjmain23} for $n=3$] 
Assuming Lemma \ref{lem:cohom van} and starting from \eqref{equ:Z3 form} one can finish the proof of the proposition by standard arguments that we now spell out explicitly: From the Maurer-Cartan equation for $\tm^3$, $\delta Z^3+\frac 12 [Z^3,Z^3]=0$, we learn by projecting to the 5-vertex part that $Dz_4=0$, where $D=\delta + [Z_{conj}^3,-]$ denotes the twisted differential. Hence by Lemma \ref{lem:cohom van} there is some $y_3\in \GC_3^{\Z_2}[[u]]$ (a series of graphs with 3 vertices), such that $Dy_3=z_4$.
But then by applying a gauge transformation with $-y_3$ to $Z^3$, we may remove the term $z_4$. Proceeding inductively and taking the limit, we have shown that $Z^3$ is gauge equivalent to $Z^3_{conj}$ as desired.
\end{proof}
It remains to show the lemma.
\begin{proof}[Proof of Lemma \ref{lem:cohom van}]
Denote by $\GC_3^{\Z_2}((u))$ the variant of $\GC_3^{\Z_2}[[u]]$ obtained by formally inverting $u$. The elements of $\GC_3^{\Z_2}((u))$ are series of graphs of even loop order with coefficients in Laurent series in $u$. Note that there is an isomorphism of dg Lie algebras
\[
\rho: \GC_3^{\Z_2}((u)) \to \GC_1^{\Z_2}((u))
\]
sending a graph $\gamma$ of loop order $\ell$ to $u^{-\ell}\gamma$.
This sends our conjectural Maurer-Cartan $Z^3_{conj}\in \GC_3^{\Z_2}((u))$ to another one $Z^1_{conj}:=\rho(Z^3_{conj})\in \GC_1^{\Z_2}\subset \GC_1^{\Z_2}((u))$, which is obtained by merely setting $P_{2n-2}=1$ in \eqref{equ:Zn conj odd}.

Now from \cite[Theorem 2]{KWZ} we know that the twisted dg Lie algebra $(\GC_1^{\Z_2})^{Z^1_{conj}}$ has one-dimensional cohomology, i.e., $H(\GC_1^{\Z_2},\delta+[Z^1_{conj},-])=\R[-1]$, and the series representing the non-trivial class consists of graphs with two vertices only. It follows that $\left(\GC_1^{\Z_2}((u))\right)^{Z^1_{conj}}$ has cohomology $\R((u))[-1]$ and via the isomorphism $\rho$ also that $\left(\GC_3^{\Z_2}((u))\right)^{Z^3_{conj}}$ has the same cohomology, concentrated again in the 2-vertex part.

Hence, suppressing the twist by $Z^3_{conj}$ temporarily for notational reasons, our complex $\GC_3^{\Z_2}[[u]]\subset \GC_3^{\Z_2}((u))$ is quasi-isomorphic to $u^{-1}\GC_3^{\Z_2}[u^{-1}]=\GC_3^{\Z_2}((u))/\GC_3^{\Z_2}[[u]]$ up to a degree shift by one and a correction in the 2-vertex part. 
Now let us compute the cohomology of $u^{-1}\GC_3^{\Z_2}[u^{-1}]$ in a different way, by using the loop order spectral sequence, whose first term is $u^{-1}H(\GC_3^{\Z_2})[u^{-1}]$. Note that $H(\GC_3)$ is concentrated in degrees $\leq -3$ and $u^{-1}$ has degree -4. 
Hence, also accounting for the degree shift by one, there cannot be classes in cohomological degree 1 with $\geq 3$ vertices in $\left(\GC_3^{\Z_2}((u))\right)^{Z^3_{conj}}$ as claimed. 
\end{proof}

\subsection{An auxiliary Theorem}

We denote gauge equvalence by the symbol $\sim$, so that the statement of Theorem \ref{conjthm:main} can be rephrased as 
\[
 Z_n^n \sim Z^n_{conj}.
\]

In particular, let us rephrase Corollary \ref{conjthm:main2} in this language.
\begin{cor}[Triviality of actions, special case of Corollary \ref{conjthm:main2}]
\label{cor:trivial}
We have $Z_m^n\sim 0$ in either of the two cases (i) $m\leq n-1$ and $n$ even or (ii) $m\leq n-2$ and $n$ odd. 
\end{cor}

Note that on $\GC_n$ we have a grading by loop order. Then, for $k$ even, we have a map of dg Lie algebras
\begin{align*}
\Phi_{k,l}^n:  \GC_{n-k}\hotimes H(B \SO(l)) &\to  \GC_n\hotimes H(B(\SO(k)\times \SO(l)))
\\
\Gamma &\mapsto E_k^L \Gamma
\end{align*}
where $E_k$ is the Euler class in $H(B\SO(k))$ and $\Gamma$ is a graph of loop order $L$.

Furthermore, restricting the group $\SO(n)$ to the subgroup $\SO(k)\times \SO(l)$ with $k+l=n$ we obtain dgca maps $H(B\SO(n))\to H(B(\SO(k)\times \SO(l)))$ and hence (restriction) maps of dg Lie algebras
\beq{equ:Rkldef}
 R_{k,l} : \GC_n^+ \hotimes H(B\SO(n))\to \GC_n^+\hotimes H(B(\SO(k)\times \SO(l))).
\eeq

In particular we have
\[
 R_{k,l} (Z_n^n) = Z_{k,l}^n.
\]
Furthermore, we shall use below that
\beq{equ:PhiZ}
\Phi_{2,n-2}^n(Z_{conj}^{n-2}) = R(Z_{conj}^n).
\eeq

The main claim is that using a version of equivariant localization one can show the following Theorem.
\begin{thm}\label{thm:locmain}\label{thm:mainloc}
  We have that for $2\leq k\leq n-2$ even and $l\geq 0$ such that $k+l\leq n$
  \[
  Z_{k,l}^n \sim_{E_k} \Phi_{k,l}^n (Z_{l}^{n-k}).
  \]
  Here $\sim_{E_k}$ means "gauge equivalent after formally inverting $E_k$". In other words this is gauge equivalence in the graph complex  
  $\GC_n^+\hotimes H(B(\SO(k)\times \SO(l)))_{E_k}$
  with coefficients in the localized ring. 
  \end{thm}

Theorem \ref{thm:mainloc} will be proven in section \ref{sec:auxthmproof} below. For now, let us believe the statement and use it to derive our main Theorem \ref{conjthm:main}.

\subsection{Derivation of Theorem \ref{conjthm:main} from the auxiliary Theorem \ref{thm:locmain} }

\subsubsection{Goldman-Millson type Lemma}

Our proof of Theorem \ref{conjthm:main} will be based on the following technical Lemma, which can be seen as a version of the Goldman-Millson Theorem \cite{DolRog}.
\begin{lemma}\label{lem:MC inj}
Let $f:\fg\to \fh$ be morphism of differential graded Lie algebras such that the following holds:
\begin{enumerate}
    \item The dg Lie algebras $\fg$ and $\fh$ are each equipped with a complete grading, which we call the auxiliary grading,
    \begin{align*}
        \fg &= \prod_{k\geq 1} \fg_k
        &
        \fh &= \prod_{k\geq 1} \fh_k.
    \end{align*}
    \item The auxiliary gradings are compatible with the dg Lie structure in the sense that 
    \begin{equation}\label{equ:MC inj compat}
    \begin{aligned}
        d\fg_k &\subset \fg_{k+1}
        &
        [\fg_k,\fg_\ell] &\subset \fg_{k+\ell} 
        \\
        d\fh_k &\subset \fh_{k+1}
        &
        [\fh_k,\fg_\ell] &\subset \fh_{k+\ell}.
    \end{aligned}
\end{equation}
Note that the fact that $d$ raises the auxiliary degree homogeneously by $+1$ implies that the cohomology of $\fg$ and $\fh$ also inherits a grading by auxiliary degree.
\item The morphism $f$ satisfies $f(\fg_k)\subset\prod_{\ell\geq k}\fh_\ell$.
We shall write $f=f_0+f_1+\cdots$ with $f_j(\fg_k)\subset \fh_{k+j}$.
\item The leading part $f_0$ of $f$ induces an injective map on the parts of the cohomology of auxiliary degree $k$,
\begin{equation}\label{equ:lem MC inj}
    [f_0]:H(\fg)_k \to H(\fh)_k  
\end{equation}
for all $k$.
\end{enumerate}
Suppose that $\alpha\in \MC(\fg)$ is a Maurer-Cartan element such that $\alpha\in \prod_{k\geq 2}\fg_\alpha$ and such that the image Maurer-Cartan element $f(\alpha)\in \MC(\fh)$ is gauge trivial, $f(\alpha)\sim 0$. Then $\alpha$ is also gauge trivial,
\[
\alpha \sim 0.    
\] 
\end{lemma}
\begin{proof}
We decompose $\alpha$ in summands with respect to the auxiliary grading,
\[
\alpha = a_p + a_{p+1}+\cdots    
\]
with $a_j\in \fg_j$ and $p\geq 2$.
Projecting the Maurer-Cartan equation 
\[
    d\alpha+\frac12 [\alpha,\alpha]=0
\]
to $\fg_{p+1}$ we obtain that 
\[
d\alpha_{p} =0.    
\]
Here we are using that $p\geq 2$ and that by the compatibility of the auxiliary grading with the dg Lie structure \eqref{equ:MC inj compat} no term in the bracket can contribute to the part of auxiliary degree $p$.
By assumption the MC element $f(\alpha)\in \fh$ is gauge trivial. This means that there exists an element $y\in \fh$ of cohomological degree $0$ such that 
\begin{equation}\label{equ:MC inj p1}
f(\alpha) = 
-\frac{\exp(\ad_y)-1}{\ad_y} dy
=
-\sum_{i\geq 0} \frac{\ad_y^i}{(i+1)!}dy,
\end{equation}
with $\ad_y=[y,-]$.
We decompose 
\[
y=y_1 +y_2+\cdots 
\]
with $y_k\in \fh_k$.
From the parts of \eqref{equ:MC inj p1} of auxiliary degrees up to $p$ it then follows that
\[
d y_k = 0
\]
for $k=1,\dots, p-2$ and 
\[
dy_{p-1} = f_0(\alpha_p),    
\]
using the compatibility of $f$ with the grading.
This means that the image of the cocycle $\alpha_p$ under $f_0$ is exact. By the injectivity assumption of the Lemma this implies that $\alpha_p$ is exact, $\alpha_p=dx$. We may assume without loss of generality that $x =:x_{p-1}\in \fg_{p-1}$. (By \eqref{equ:MC inj compat} the degree $p-1$ part $x_{p-1}$ of $x$ also satisfies $\alpha_{p-1}=dx_p$.)

Now we gauge transform $\alpha$ by $x_{p-1}$ to 
\[
\alpha' =
\sum_{i\geq 0} \frac{\ad_{x_{p-1}}^i}{i!}\alpha
-\sum_{i\geq 0} \frac{\ad_{x_{p-1}}^i}{(i+1)!}dy
(\alpha_{p} - dx_{x_{p-1}}) + (\cdots) 
=:
\alpha'_{p+1} + (\cdots).
\]
The MC element $\alpha'$ is concentrated in auxiliary degrees $\geq p+1$. Furthermore, $f(\alpha')$ is still gauge trivial in $\fh$, with the gauge transformation obtained by chaining the original gauge transform by $y$ with that by $f(x_{p-1})$.
Hence we may repeat the above argument and construct an element $x_{p}\in \fg_p$ that gauge transforms $\alpha'$ to an MC element $\alpha''$ concentrated in auxiliary degrees $\geq p+2$ etc.
Finally composing the gauge transformations by all the $x_j$ yields a gauge transfomation from $\alpha$ to zero, thus showing the Lemma.
(We refer to the Appendix of \cite{DolRog} for the composition of an infinite sequence of gauge transformations.)
\end{proof}

\subsubsection{Proof of Theorem \ref{conjthm:main} provided Theorem \ref{thm:locmain}}\label{sec:mainfromlocmain}

We proceed by induction on $n$. For $n=2$ and $n=3$ Theorem \ref{conjthm:main} is known, see Proposition \ref{prop:conjmain23}.
Next we suppose that we already know that $Z_{n-2}^{n-2}\sim Z^{n-2}_{conj}$ and we desire to show that 
\beq{equ:tmpdes}
Z_{n}^{n}\sim Z^{n}_{conj}.
\eeq
We invoke Theorem \ref{thm:locmain} for $k=2$, $l=n-2$, assuming $n\geq 4$.\footnote{In fact, the case $n=3$ may also be tackled in this way, giving a second proof of the conjecture for $n=3$. However, in the interest of uniformity of notation, let us assume $n\geq 4$.}
Theorem \ref{thm:locmain} then states that
\[
 Z_{2,n-2}^n \sim_{u} \Phi_{2,n-2}^n (Z_{n-2}^{n-2}),
\] 
where we abbreviate the orthogonal Euler class by $u$ (it has degree $+2$).
Now by our induction hypothesis $Z_{n-2}^{n-2}\sim Z_{conj}^{n-2}$, and hence, using \eqref{equ:PhiZ} we find that
\beq{equ:tmp11}
Z_{2,n-2}^n \sim_{u}  R_{2,n-2}(Z_{conj}^n).
\eeq
Where $R_{2,n-2}$ is as in \eqref{equ:Rkldef}.
Note also that clearly $Z_{2,n-2}^n=R_{2,n-2}(Z_n^n)$, so that we have 
\[
 R_{n-2,2}(Z^{n}_{conj})\sim_{u} R_{n-2,2}(Z_{n}^{n}).
\]
Let us be explicit how the map of the coefficient rings underlying $R_{n-2,2}$ looks like.
For $n$ even we have
\begin{align*}
H(B\SO(n)) = \R[P_4,\dots,P_{2n-4}, E_n] &\to H(B(\SO(2)\times \SO(n-2))) \cong \R[u, P_4,\dots,P_{2n-8},E_{n-2}]
\\
E_n &\mapsto u E_{n-2} \\
P_{2n-4}&\mapsto u^2 P_{2n-8}+ E_{n-2}^2 \\
P_j&\mapsto u^2 P_{j-4} + P_j \quad \text{(for $j\neq 2n-4$)}\, .
\end{align*}
In the above and in the formulas for $n$ odd below, we assume $P_0=1$.
For $n$ odd we have
\begin{align*}
H(B\SO(n)) = \R[P_4,\dots,P_{2n-2}] &\to H(B(\SO(2)\times \SO(n-2))) \cong \R[u, P_4,\dots,P_{2n-6}]
\\
P_{2n-2}&\mapsto u^2 P_{2n-6} \\
P_j&\mapsto u^2 P_{j-4} + P_j \quad \text{(for $j\neq 2n-2$)} \, .
\end{align*}

Now localize the rings on the right hand side over $u$. We can then exchange the generator $E_{n-2}$ by $E_n:=uE_{n-2}$, respectively $P_{2n-6}$ by $P_{2n-2} := u^2P_{2n-6}$. This will make uniform the formula
for the differential for both left- and right-hand sides.
The maps above then change in that for $n$ even
\begin{equation}\label{equ:coeff map ev}
\begin{aligned}
E_n&\mapsto E_n \\ 
P_{2n-4}&\mapsto u^2 P_{2n-8}+ u^{-2} E_{n}^2 \\
P_j&\mapsto u^2 P_{j-4} + P_j 
\quad\quad\quad \text{for $j=4,8,\dots,2n-8$}
\end{aligned}
\end{equation}
and for $n$ odd
\begin{equation}\label{equ:coeff map odd}
    \begin{aligned}
P_{2n-2} &\mapsto P_{2n-2} \\  
P_{2n-6}&\mapsto u^2 P_{2n-10}+u^{-2} P_{2n-2} \\
P_j&\mapsto u^2 P_{j-4} + P_j 
\quad\quad\quad \text{for $j=4,8,\dots,2n-10$}. 
\end{aligned}
\end{equation}

Now we want to use \eqref{equ:tmp11}, or equivalently $R_{2,n-2}(Z_{n}^n) \sim_{u}  R_{2,n-2}(Z_{conj}^n)$, to show that $Z_{n}^n \sim  Z_{conj}^n$.
We have to make a case distinction according to whether $n$ is even or odd.

\subsubsection*{The case of even $n$:}
Suppose first that $n$ is even.
To begin with, we define on 
$$
\fg := \GC_n\hotimes H(B\SO(n))=\GC_n\hotimes \R[P_4,\dots,P_{2n-4}, E_n]
$$ 
and on 
$$
\tilde \fh :=\GC_n\hotimes H(B\SO(2)\times B\SO(n-2))_{u}
=
\GC_n\hotimes \R[u, u^{-1},P_4,\dots,P_{2n-8}, E_n]
$$
an auxiliary grading given by the number of vertices plus the power of $E_n$ in the coefficient, minus one. We call this number the \emph{EV degree}.
In other words, elements of the homogeneous subspace $\fg_k$ (resp. $\fh_k$) of EV degree $k$ are linear combinations of expressions $\Gamma \otimes E_n^j p(\dots)$ with $\Gamma$ a graph with $k+1-j$ vertices and $p(\dots)$ some polynomial in the other variables.
Note that this grading on $\fg$ is complete, i.e.,
\[
  \fg = \prod_{k\geq 1} \fg_k,
\]
because there are only finitely many graphs with at most $k+1$ vertices, and the coefficient polynomials are non-negatively graded.
The grading on $\tilde \fh$ is a priori not complete, because the latter assertion fails since $u^{-1}$ has negative degree. But we may define the completion by EV degree 
\[
\fh := (\tilde \fh )^\wedge := \prod_{k\geq 1} \tilde \fh_k.
\]
To illustrate the difference, the infinite linear combination 
\[
\sum_{j\geq 0} \Gamma \otimes u^{-nj}E_n^{2j}   
\]
is an element of $\fh$ but not $\tilde \fh$.

Now we are ready to apply Lemma \ref{lem:MC inj} with the following data:
 \begin{itemize}
\item The dg Lie algebras $\fg$ and $\fh$ of Lemma \ref{lem:MC inj} are taken to be the dg Lie algebras $\fg$ and $\fh$ above, with the EV grading as the auxiliary grading.
The differential is the twisted differential $D=\delta+[Z^n_{conj},-]$ on $\fg$, resp. $D=\delta+[R_{2,n-2}(Z^n_{conj}),-]$ on $\fh$. As before, the Lie brackets are obtained from the Lie bracket on $\GC_n$ by linear extension over the coefficient ring. 
\item The map $f:\fg\to \fh$ is $f=R_{2,n-2}$. In other words, it is the identity on the factor $\GC_n$ and the natural morphism $H(B\SO(n))\to H(B\SO(2)\times B\SO(n-2))_{u}$ described above on the second factor.
\item The Maurer-Cartan element $\alpha\in \fg$ is $\alpha = Z^n -Z^n_{conj}$. In particular, we are done if we can apply Lemma \ref{lem:MC inj} to conclude that 
$\alpha\sim 0$ in $\fg$.
 \end{itemize}
We check the conditions of Lemma \ref{lem:MC inj}.
First, it is easy to check that the EV degree is respected by the Lie bracket, differential and the map $f$.
Next, the MC element $\alpha$ has only terms of EV degree $\geq 2$. This is because the leading term of $Z^n_n$ (the tadpole graph, of EV degree 1) was subtracted to obtain $\alpha$.
By assumption, we furthermore have that $f(\alpha)\sim 0$ is gauge trivial in $\fh$. To be more precise, we know that $f(\alpha)$ is gauge trivial as an element of the dg Lie subalgebra $\tilde \fh\subset\fh$. But this implies a fortiori that it is also gauge trivial in $\fh$.

This leaves only condition (4) of Lemma \ref{lem:MC inj} to be checked, namely the injectivity of $f_0$ on cohomology in each fixed EV degree. 
To this end we note that, as graded vector spaces,
\begin{align*}
\fg_k &=
(\GC_n\hotimes \R[E_n])_k\hotimes \R[P_4,\dots,P_{2n-4}]
=
(\GC_n\otimes \R[E_n])_k \otimes \R[P_4,\dots,P_{2n-4}]
\\
\fh_k &=
(\GC_n\hotimes \R[E_n])_k\hotimes \R[u,u^{-1},P_4,\dots,P_{2n-8}]
=
(\GC_n\otimes \R[E_n])_k\otimes \R[u,u^{-1},P_4,\dots,P_{2n-8}]
,
\end{align*}
where $(-)_k$ refers to taking a part of fixed EV degree,
and where we use the filtration on $\GC_n$ by the number of edges to complete the tensor products. The completion can then be omitted because the vector space $(\GC_n\hotimes \R[E_n])_k$ is in each case finite dimensional.
Note furthermore that the differential $D$ acts only on the first factors, as a map $(\GC_n\otimes \R[E_n])_k\to (\GC_n\otimes \R[E_n])_{k+1}$, and trivially on the second factors in the tensor product above.
This means that by the Künneth formula the part of the cohomology of EV degree $k$ is 
\begin{align*}
  H(\fg)_k &=
  H(\GC_n\otimes \R[E_n])_k \otimes \R[P_4,\dots,P_{2n-4}]
  \\
  H(\fh)_k &=
  H(\GC_n\otimes \R[E_n])_k\otimes \R[u,u^{-1},P_4,\dots,P_{2n-8}].
\end{align*}
The map $[f_0]: H(\fg)_k \to H(\fh)_k$ induced by $f$ between these spaces (that we need to check is injective, see \eqref{equ:lem MC inj}) is given as the identity on the first factors $H(\GC_n\otimes \R[E_n])_k$, and on the second factors by the morphism of graded rings (cf. \eqref{equ:coeff map ev})
\begin{equation}\label{equ:coeff map ev mod}
  \begin{aligned}
    \phi: \R[P_4,\dots,P_{2n-4}] &\to \R[u,u^{-1},P_4,\dots,P_{2n-8}] \\
  P_{2n-4}&\mapsto u^2 P_{2n-8} \\
  P_j&\mapsto u^2 P_{j-4} + P_j 
  \quad\quad\quad \text{for $j=4,8,\dots,2n-8$}.
  \end{aligned}
\end{equation}
To check that $[f_0]$ is injective it suffices to check that the map $\phi$ is injective.
But this is clearly true: Put a lexicographic order on monomials corresponding to the variable ordering $P_i>P_j$ if $i>j$ and $P_i>u>u^{-1}$.
Then leading oder terms of \eqref{equ:coeff map ev mod} read
\begin{equation*}\label{equ:coeff map ev mod 2}
    \begin{aligned}
    P_{2n-4}&\mapsto u^2 P_{2n-8} \\
    P_j&\mapsto P_j 
    \quad\quad\quad \text{for $j=4,8,\dots,2n-8$}
    \end{aligned},
\end{equation*}
and this assignment clearly yields an injection on monomials, and hence $\phi$ is injective.

This means that we have verified condition (4) of Lemma \ref{lem:MC inj}. We can hence apply the Lemma to conclude that $\alpha\sim 0$ is gauge trivial in $\fg$.
But this is equivalent to $Z^n_n$ being gauge equivalent to $Z^n_{conj}$ in $\GC_n\hotimes H(B\SO(n))$. This then shows Theorem \ref{conjthm:main} for the case of even $n$.

\subsubsection*{The case of odd $n$:}
Next consider the case of odd $n$. Here we proceed analogously to the case of even $n$ above, and apply Lemma \ref{lem:MC inj} to the following data:

\begin{itemize}
    \item The dg Lie algebras $\fg$ and $\fh$ are
    \begin{align*}
    \fg&=(\GC_n\hotimes H(B\SO(n)))^{Z^n_{conj}}
    =(\GC_n\hotimes\R[P_4,\dots,P_{2n-2}], D) \\
    \fh&=( \GC_n\hotimes H(B\SO(2)\times B\SO(n-2))_{u} )^{R_{2,n-2}(Z^n_{conj})}
    =
    (\GC_n\hotimes\R[u,u^{-1}, P_4,\dots,P_{2n-6},E_n], D).
    \end{align*}
    \item The auxiliary grading is given by the number of vertices minus one.
    In other words, elements of $\fg_k\subset \fg$ (resp. $\fh_k\subset \fh$) are linear combinations of expressions $\Gamma \otimes p(\dots)$ with $\Gamma$ a graph with $k+1-j$ vertices.
    We call the auxiliary degree the vertex-degree.
    \item The map $f:\fg\to \fh$ is again $f=R_{2,n-2}$. 
    \item The Maurer-Cartan element $\alpha\in \fg$ is $\alpha = Z^n -Z^n_{conj}$.
     \end{itemize}
    As before, one readily checks the conditions of Lemma \ref{lem:MC inj} except for the injectivity condition (4) of that Lemma. We have
    \begin{equation}\label{equ:ghk map odd} 
    \begin{aligned}
    \fg_k &=
    (\GC_n)_k\hotimes \R[P_{2n-2}] \hotimes \R[P_4,\dots,P_{2n-6}],
    &
    \fh_k &=
    (\GC_n)_k\hotimes \R[P_{2n-2}] \hotimes \R[u,u^{-1},P_4,\dots,P_{2n-10}],
    \end{aligned}
    \end{equation}
    with $\hotimes$ the completed tensor product with respect to the filtration by the number of edges in graphs.
    The differential $D$ only acts non-trivially on the left-hand factor $\GC_n\hotimes \R[P_{2n-2}]$, and is trivial on the right-hand factors $\R[P_4,\dots,P_{2n-6}, P_{2n-2}]$ and $\R[u,u^{-1},P_4,\dots,P_{2n-10}]$ in the tensor products.
    The map $f_0:\fg_k\to \fh_k$ in this case is equal to $f_0=f=\mathit{id}_{\GC_n}\otimes \psi$ and is given by the identity on the factor $(\GC_n)_k$, and maps the right-hand factors according to:
    \begin{equation}\label{equ:coeff map odd mod}
      \begin{aligned}
        \psi\colon \R[P_4,\dots,P_{2n-6}, P_{2n-2}]
        &\to 
        \R[u,u^{-1},P_4,\dots,P_{2n-10}, P_{2n-2}] \\
        P_{2n-2} & \mapsto P_{2n-2} \\
      P_{2n-6}&\mapsto u^2 P_{2n-10} + u^{-2} P_{2n-2} \\
      P_j&\mapsto u^2 P_{j-4} + P_j 
      \quad\quad\quad \text{for $j=4,8,\dots,2n-8$}
      \end{aligned}
  \end{equation}

    Also note that in contrast to the case of even $n$, graphs now can have multiple edge between any pair of vertices. Hence there are (potentially) infinitely many graphs on $k$ vertices, and we cannot replace the completed tensor products in \eqref{equ:ghk map odd} above by ordinary tensor products. 
    To compensate, we use the descending filtration in $(\GC_n\hotimes \R[P_{2n-2}],D)$ by loop order.

Our goal now is to show that the map $[f_0]=[f]$ of \eqref{equ:lem MC inj} is injective.
To this end, suppose that we have a cocycle $x$ in $\fg_k$ such that the image $y:= f(x)\in \fh_k$ is exact, that is, $y=dz$ for some $z\in \fh_{k-1}$.
Concretely, we write out the three elements $x$, $y$, $z$ as 
\begin{align*}
x &= \sum_{\underline i} x_{\underline i} P^{\underline i}
&
y&=
\sum_{\underline j} y_{\underline j} \bar P^{\underline j}
&
z&=
\sum_{\underline j} z_{\underline j} \bar P^{\underline j},
\end{align*}
where $\underline i=(i_4,i_8\dots,i_{2n-6})$ and $\underline j=(j_u, j_4,j_8,\dots, j_{2n-10})$ are multi-indices, $x_{\underline i},y_{\underline j},z_{\underline j}\in \GC_n\hotimes \R[P_{2n-2}]$ and with the monomials
\begin{align*}
P^{\underline i} &:= P_4^{i_4}\cdots P_{2n-6}^{i_{2n-6}}
&
\bar P^{\underline j} &:= u^{j_u} P_4^{j_4}\cdots P_{2n-10}^{j_{2n-10}}.
\end{align*}
We want to construct an element $w=\sum_{\underline i} x_{\underline i} P^{\underline i}$ such that $dw=x$.
Mind that all the series $x,y,z,w$ may be infinite, but each graph in $\GC_n$ may appear only in finitely many $x_{\underline i}, y_{\underline j} ,z_{\underline j} w_{\underline i}$.

We construct $w$ by the following inductive algorithm.

\medskip 

{\bf Start:} Set $x^{(0)} = x$, $y^{(0)} = y$, $z^{(0)} = z$ and $w^{(0)} = 0$.

\medskip 

{\bf Step $n$:} Let 
\begin{align*}
    x^{(n)} &= \sum_{\underline i} x^{(n)}_{\underline i} P^{\underline i}
    \in \fg_k
    &
    y^{(n)}&=
    \sum_{\underline j} y^{(n)}_{\underline j} \bar P^{\underline j}
    \in \fh_k
    &
    z^{(n)}&=
    \sum_{\underline j} z^{(n)}_{\underline j} \bar P^{\underline j}\in \fh_{k-1}
    &
    w^{(n)}\in \fg_{k-1}
    \end{align*}
 be given, such that 
\begin{align}\label{equ:alg invariants}
    y^{(n)} &= f(x^{(n)}) = d z^{(n)} &
    x-x^{(n)} &= dw^{(n)}.
\end{align}
Let $\ell$ be the smallest loop order of graphs occurring in $x^{(n)}$.
Since there are only finitely many graphs of loop orders $\ell$ with $k$ vertices, there are only finitely many $\underline i$ such that $x_{\underline i}^{(n)}$ contains graphs of loop order $\ell$. 
Among those we pick the multi-index $\underline i^{max}$ that is maximal, in the sense that 
the monomial $P^{\underline i^{max}}$ is of maximal degree, and among those of maximal degree it is lexicographically largest, using the ordering on generators $P_{2n-6}>P_{2n-8}>\dots$.
Then looking at \eqref{equ:coeff map odd mod} we see that the loop order $\ell$ parts of $x_{\underline i^{max}}^{(n)}$ is the same as the loop order $\ell$ part of $y_{\underline j^{max}}^{(n)}$ with $\underline j^{max}$ such that 
\[
    \bar P^{\underline j^{max}}
    =
    (u^{2}P_{2n-10})^{i^{max}_{2n-6}}
    P_{2n-10}^{i^{max}_{2n-10}}
    \cdots 
    P_{4}^{i^{max}_{4}},
\]
because the terms of $x$ corresponding to other $\underline i$ can only contribute to lower monomials or higher loop orders by the choice of $\underline i^{max}$.
But we know that 
\[
D z^{(n)}_{\underline j^{max}} = y^{(n)}_{\underline j^{max}}.
\]
Then we define
\begin{align*}
x^{(n+1)} & := x^{(n)} - D z_{\underline j^{max}}^{(n)} P^{\underline i^{\max}}
&
y^{(n+1)} & := f(x^{(n+1)})
\\
z^{(n+1)} & := z^{(n)} - f(z^{(n)}_{\underline j^{max}} P^{\underline i^{\max}} )
&
w^{(n+1)}:= w^{(n)} + z^{(n)}_{\underline j^{max}} P^{\underline i^{\max}}
\end{align*}
and start the next algorithm step $n+1$.
Note that the conditions \eqref{equ:alg invariants} are still true at step $n+1$ by construction.

\medskip 

Finally, we define $w:= \lim_{n\to \infty} w^{(n)}\in \fg_{k-1}$. We claim that (i) this limit exists and that (ii) $\lim_{n\to \infty} x^{(n)}=0$. If this is shown then it follows by \eqref{equ:alg invariants} that $dw=x$ as desired.

To check (ii) note that $x^{(n+1)}$ still contains no graphs of order $\leq \ell -1$, and that the graphs of $\ell$ occur with coefficients being monomials that are smaller (in above ordering) than those in $x^{(n)}$. Since only finitely many monomials can be present as coefficients of loop order $\ell$ graphs, the algorithm will eventually eliminate all loop order $\ell$ graphs from $x^{(n)}$. This means that for large enough $n$ graphs occurring in $x^{(n)}$ will have arbitrarily high loop orders only, so that indeed
\[
    \lim_{n\to \infty} x^{(n)}=0.  
\]
We next check that the sequence $w^{(n)}$ converges. This means for any fixed graph $\Gamma$ its coefficient is eventually constant.
To see this note that the update rule for $z^{(n)}$ removes the coefficient for a monomial, and adds it again, but with strictly lower monomials. 
But since $\Gamma$ can occur only at finitely many of the $z_{\underline j}$, it means that can be added to some $w^{(n)}$ only finite many times.
Hence the sequence $w^{(n)}$ indeed converges.

We hence have shown that \eqref{equ:lem MC inj} is indeed injective.
But this means that we can apply Lemma \ref{lem:MC inj} and Theorem \ref{conjthm:main} follows for odd $n$ as well.
\hfill\qed

\newcommand{\hFM}{\widehat{\FM}}
\section{Equivariant localization and proof of Theorem \ref{thm:locmain}} \label{sec:auxthmproof}


\subsection{A relative version of configuration space}
Consider $\R^m$ as a fixed subset of $\R^n$ by embedding it along the first $m$ coordinates.
We define the space
\[
\FM_{m,n}(r,s) \subset \FM_n(r+s)
\]
as the subspace for which the last $s$ points lie on $\R^m\subset \R^n$.
More precisely, the space $\FM_{m,n}(r,s)$ fits into a pullback diagram
\[
 \begin{tikzcd}
\FM_{m,n}(r,s) \ar{r}\ar{d} & \FM_n(r+s) \ar{d} \\
\FM_m(s) \ar{r} & \FM_n(s)
 \end{tikzcd}\, .
\]

Following Kontsevich's notation we call the $r$ first points \emph{type I points} and the others $s$ points (which lie in the $m$-dimensional subspace) \emph{type II points}.
The totality of spaces $\FM_{m,n}(-,-)$ together with $\FM_n$ forms a colored operad $\hFM_{m,n}$ such that
\begin{itemize}
\item The operations with output in color 1 are
\[
\hFM_{m,n}^1(r,s)=
\begin{cases}
\FM_n(r) & \text{for $s=0$} \\
\emptyset & \text{otherwise}
\end{cases}.
\]
\item The operations with output in color 2 are
\[
\hFM_{m,n}^2(r,s)=
\FM_{m,n}(r,s) \quad \text{for $r\geq 0$ and $s\geq 1$}.
\]
\end{itemize}
The operadic compositions are inherited from those on $\FM_n$, i.e., defined by gluing one configuration into another.

\begin{rem}
There is also a variant of the above colored operad in which one allows for operations with output in color 2 but no input in color 2.
The definition of the appropriate compactification in that case is slightly more intricate, however. In this paper we only need to work with the version above.
\end{rem}

Obviously, the colored operad $\hFM_{m,n}$ is equipped with a natural action of $O(m)\times O(n-m)$, by restriction of the $O(n)$ action on $\FM_n$.

%


\begin{cons}\label{cons:twocolopfromop}
Let $\op P$ be an operad. Let us define a two-colored operad $\op P^{2-col}$, such that 
\begin{align*}
 \op P^{2-col,1}(r,s) &=
\begin{cases}
 \op P(r) & \text{if $s=0$} \\
0& \text{otherwise}
\end{cases}
&
 \op P^{2-col,2}(r,s) &=
\begin{cases}
 \op P(r+s) & \text{if $s\geq 1$} \\
0& \text{otherwise},
\end{cases}
\end{align*}
with the operadic compositions inherited from $\op P$. Dually, given a cooperad $\op C$, we define a two colored cooperad $\op C^{2-col}$ by the analogous construction.
\end{cons}

\newcommand{\stGra}{\gra}

\newcommand{\sthGra}{\widehat{\stGra}}

\newcommand{\hZ}{\hat Z}

\subsection{A complex of graphs}
\label{sec:two v graphs}
Recall the cooperad $\stGra_n$ from section \ref{sec:GraDefinitions}. Set $G=\SO(m)\times \SO(n-m)$, fix a maximal torus $T$ and compact subgroup $K\cong W\ltimes T$, with $W$ the Weyl group as before.
First let us define a two colored cooperad $\stGra_{m,n}=\stGra_n^{2-col}$ from $\stGra_n$ using Construction \ref{cons:twocolopfromop}.
More concretely, we define a family of graded vector spaces $\stGra_{m,n}(r,s)=\stGra_n(r+s)$ consisting of graphs in $r$ "type I" and $s$ "type II" vertices, with the same sign and degree conventions as for $\stGra_n$. 
In pictures, we shall distinguish the type II vertices by drawing them on a "baseline", which shall be thought of representing $\R^m$, as follows
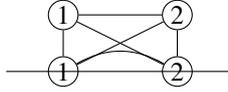
\[
\begin{tikzpicture}[scale=1.5]
\draw (-1,0) -- (1,0);
\node[ext] (u) at (-.5,0) {$1$};
\node[ext] (v) at (.5,0) {$2$};
\node[ext] (w) at (-.5,.5) {$1$};
\node[ext] (x) at (.5,.5) {$2$};
\draw (u) edge[bend left] (v) edge (w) (w) edge (v) edge(x) (x) edge (v) edge (u);
\end{tikzpicture}
\]

The pair $\stGra_n(-)$ and $\stGra_{m,n}(-,-)$ is naturally a two colored Hopf cooperad, which we call $\sthGra_{m,n}$.
Let $G=\SO(m)\times \SO(n-m)\subset \SO(n)$. Then there is a map of symmetric sequences of dgcas
\beq{equ:sthGramap}
\sthGra_{m,n}\otimes H(BG) \to \Omega_{G}^{min} (\hFM_{m,n})
\eeq
sending a graph $\Gamma$ to the product over edges
\[
  \bigwedge_{(i,j)\in E\Gamma} \pi_{ij}^* \Omega, 
\]
where $\pi_{ij}$ is the forgetful map forgetting all but points $i$ and $j$ from a configuration, and $\Omega$ is the ``propagator'', the $G$-equivariant form on $S^{n-1}$ as in section \ref{sec:propagator}. (Precisely, we restrict the $\SO(n)$-equivariant form of section \ref{sec:propagator} to a $G\subset\SO(n)$-equivariant form.)

We also define the dual two colored operad $\hGra_{m,n}$.
We consider the graded Lie algebras of invariants of those colored operads.
To describe them correctly including signs and degrees, consider the two-colored operad $\Lie_{m,n}$ whose algebras are pairs consisting of a $\Lie_m$ algebra and a $\Lie_n$ algebra, together with an action of the $\Lie_n$ algebra on the $\Lie_m$ algebra by derivations. 
Concretely, the colored operad $\Lie_{m,n}$ has three generators: A degree $1-n$ Lie bracket $\mu_2\in \Lie_{m,n}^1(2,0)$ with inputs and output in color 1, a degree $1-m$ Lie bracket $\mu_{0,2}\in \Lie_{m,n}^2(0,2)$ with input and output in color 2, and the action $\mu_{1,1}\in \Lie_{m,n}^2(1,1)$ of degree $1-n$ with output in color 2 and one input of each color. The two Lie brackets satisfy the (graded) Jacobi identity, and in addition we have the ternary relations 
\begin{align*}
    \mu_{1,1}(x, \mu_{0,2}(a,b)) &=
    \mu_{0,2}(\mu_{1,1}(x, a),b)
    +
    (-1)^{(1-n)(1-m)}\mu_{0,2}( a,\mu_{1,1}(x,b))
    \\
    \mu_{1,1}(x, \mu_{1,1}(y, a))
    &= \mu_{1,1}(\mu_2(x,y), a)+(-1)^{1-n} \mu_{1,1}(y, \mu_{1,1}(x, a)).
\end{align*}

Note in particular that there is no operation with output in color 2, but no input in color 2.
We denote the minimal resolution by $\hoLie_{m,n}$.
Concretely, $\hoLie_{m,n}$ is generated by the following operations: 
\begin{itemize}
 \item Operations $\mu_k$ with $k\geq 2$ inputs in color 1 and the output in color one, spanning a one-dimensional representation of $S_k$ in degree $1-(k-1)n$. The operations generate $\hoLie_n$.
\item Operations $\mu_{k,l}$ with $k$ inputs in color one, $l$ inputs in color 2 and output in color 2, where $k\geq 0$, $l\geq 1$, $k+l\geq 2$. The operation $\mu_{k,l}$ has degree $1-kn-(l-1)m$, and spans a one-dimensional subspace under the action of the group $S_k\times S_l$.
The $\mu_{0,l}$ generate a copy of $\hoLie_m$ inside $\hoLie_{m,n}$.
\end{itemize}

Then the (co)invariant Lie algebra can be defined as the deformation complex
\begin{align*}
\fGC_{m,n} := \Def(\hoLie_{m,n}\stackrel{0}\to \hGra_{m,n})
\cong 
\underbrace{
\prod_{k} \left( \gra_n^*(k)\otimes \K[-n]^{\otimes k}[n]\right)_{S_k}}_{=\fGC_n}
\oplus
\underbrace{
\prod_{k,l}\left(\gra_{m,n}^*(k,l)\otimes \K[-n]^{\otimes k}\otimes \K[-m]^{\otimes l}[m]\right)_{S_k\times S_l}
}_{=\fGC{m,n}'}
\end{align*}
of the trivial maps sending all generators to zero. (This is just the (co)invariants of the total space, up to some degree shifts.)
Via the map \eqref{equ:sthGramap} we obtain a Maurer-Cartan element
\begin{align}\label{equ:hZmn def}
\hZ_{m,n} &= 
\underbrace{
-E_mE_{n-m}\left(\tadpole+
\begin{tikzpicture}[baseline=-.65ex]
  \node[int] (v) at (0,0) {};
  \draw (v) edge[loop] (v);
  \draw (-.3,0)--(.3,0);
  \end{tikzpicture}
  \right)
}_{=:\hZ_{m,n}^1}
  +
\sum_\Gamma \Gamma \int \underbrace{\bigwedge_{(i,j)\in E\Gamma} \pi_{ij}^* \Omega}_{=: \tilde \omega_{\Gamma^*} } \in \fGC_{m,n}\hotimes H(BG).
\end{align}
Note that the first term on the right-hand side reflects the fact that to an edge one assigns a form that is not equivariantly closed. The proof that \eqref{equ:hZmn def} satisfies the Maurer-Cartan equation is analogous to the proof of Corollary \ref{cor:mMC} (resp. Lemma \ref{lem:int compat G}).

Starting from this point on, let us only focus on the case of even $n-m$, which is what we need below. 
To be explicit, the leading order terms of $\hZ_{m,n}$ are
\beq{equ:hZmn}
\hZ_{m,n} = 
\underbrace{
\begin{tikzpicture}
\node[int] (v) at (0,0) {};
\node[int] (w) at (0.6,0) {};
\draw (v) edge (w);
\end{tikzpicture}
+
\begin{tikzpicture}
\draw (-.5,0) -- (.5,0);
\node[int] (v) at (0,0) {};
\node[int] (w) at (0,.5) {};
\draw (v) edge (w);
\end{tikzpicture}
+
E_{n-m}
\begin{tikzpicture}
\draw (-.5,0) -- (.5,0);
\node[int] (v) at (-.3,0) {};
\node[int] (w) at (0.3,0) {};
\draw (v) edge[bend left] (w);
\end{tikzpicture}
}_{=:\hZ_{m,n}^0}
+ (\cdots)
\eeq

We denote the leading terms by $\hZ_{m,n}^0$ as indicated in the formula, and regard the remainder $\hZ_{m,n}-\hZ_{m,n}^0$ as a perturbation of $\hZ_{m,n}^0$. One easily verifies that $\hZ_{m,n}^0$ is itself a Maurer-Cartan element.

\begin{rem}
 The leading term $\hZ_{m,n}^0$ is the Maurer-Cartan element corresponding to the colored operad map
\[
 \hoLie_{m,n} \to \Lie_{m,n} \xrightarrow{f} \hGra_{m,n}\hotimes H(BG)
\]
where $f$ maps the generators as follows:
\begin{align*}
 f(\mu_2) &=
\begin{tikzpicture}
\node[ext] (v) at (0,-0.2) {};
\node[ext] (w) at (0.6,0.2) {};
\draw (v) edge (w);
\end{tikzpicture}
&
  f(\mu_{1,1}) &=
\begin{tikzpicture}
\draw (-.5,0) -- (.5,0);
\node[ext] (v) at (0,0) {};
\node[ext] (w) at (0,.5) {};
\draw (v) edge (w);
\end{tikzpicture}
&
f(\mu_{0,2}) &=
E_{n-m}
\begin{tikzpicture}
\draw (-.5,-0) -- (.5,0);
\node[ext] (v) at (-.3,0) {};
\node[ext] (w) at (0.3,0) {};
\draw (v) edge[bend left] (w);
\end{tikzpicture}\, .
\end{align*}
\end{rem}

\begin{rem}
 Let us collect several maps between the operads constructed so far.
First, there are obvious inclusions 
\begin{align}\label{equ:Liemninclusions}
 \Lie_m&\to \Lie_{m,n} & \Lie_n&\to \Lie_{m,n}\, ,
\end{align}
interpreting the left-hand side in each case as a colored operad concentrated in color 2 (respectively, color 1).

Next, suppose that $\mathcal R$ is any graded commutative ring containing an element $\lambda$ of degree $n-m$. Recall that we require $n-m$ to be even. Then there is a colored operad map (cf. also Construction \ref{cons:twocolopfromop})
\begin{equation}
 \label{equ:Liemn2Lie}
\Lie_{m,n} \to \Lie_n^{2-col} \otimes \mathcal R \,.
\end{equation}
This map is defined on generators as follows:
\begin{align*}
 \mu_2 &\mapsto \mu_2 
&
\mu_{1,1} &\mapsto \mu_2
&
\mu_{0,2} &\mapsto \lambda \mu_2.
\end{align*}

\end{rem}

\subsection{Combinatorial description of \texorpdfstring{$\fGC_{m,n}$}{fGCmn}}
The graded Lie algebra $\fGC_{m,n}$ has a semi-direct product structure owed to its definition as a deformation complex of a colored operad.
Concretely, as graded Lie algebra
\beq{equ:fsplitting}
\fGC_{m,n} = \fGC_n \ltimes \fGC_{m,n}'
\eeq
where $\fGC_n$ is spanned by graphs "without baseline", i.e., with only type I vertices, while $\fGC_{m,n}'$ is spanned by graphs with baseline, with at least one type II vertex on the baseline.
The Lie bracket on $\fGC_{m,n}'$ is by inserting into type II (i.e., baseline-)vertices. The Lie action of $\fGC_n$ is by insertion into type I vertices.

Changing the ground ring to $H(BG)$ and twisting by the Maurer-Cartan element $\hZ_{m,n}$ produces several terms in the differential.
Among them are the differential on $\fGC_n$, and terms sending $\fGC_n\otimes H(BG)\to \fGC_{m,n}'\otimes H(BG)$.

\subsection{Connectedness and the dg Lie subalgebra \texorpdfstring{$\GC_{m,n}$}{GCmn} }
We define the connected dg Lie subalgebra $\GC_{m,n}\subset \fGC_{m,n}$ to be composed of the connected graphs.
Here a graph counts as connected if any two vertices can be connected by a path of edges, irrespective of the vertex types (I or II).

Similar to \eqref{equ:fsplitting} we have a splitting of $\GC_{m,n}$ into a semi-direct product (as graded Lie algebra) 
\beq{equ:splitting}
\GC_{m,n} = \fGCc_n \ltimes \GC_{m,n}'
\eeq
where $\fGCc_n\subset \fGC_n$ is the standard (connected) graph complex, but without any valence restriction on vertices.

\begin{lemma}
The Maurer-Cartan element $\hZ_{m,n}$ lives inside the connected part $\GC_{m,n}\subset \fGC_{m,n}$.
\end{lemma}
\begin{proof}
Suppose $\Gamma=\Gamma_1\sqcup \Gamma_2$ is a non-connected graph, with the pieces $\Gamma_1,\Gamma_2$ non-empty and not connected to each other. Then the corresponding integrand $\tilde \omega_\Gamma=\tilde \omega_{\Gamma_1}\wedge \tilde \omega_{\Gamma_2}$ is 
basic under rescaling and translation of the points contributing to $\Gamma_1$ and $\Gamma_2$ \emph{separately}. 
Hence the form can not have a top form component on configuration space, which is obtained by quotienting out (only) the diagonal scaling and translation action. 
\end{proof}

\begin{lemma}\label{lem:vanishing}
Every graph $\Gamma$ appearing non-trivially in the non-leading part $\hZ_{m,n}-\hZ_{m,n}^0$ of the Maurer-Cartan element $\hZ_{m,n}$ satisfies:
\begin{itemize}
\item Every type I vertex has valence $\geq 3$.
\item Every type II vertex that is not connected to a type I vertex has valence $\geq 3$.
\end{itemize}
\end{lemma}
\begin{proof}
This follows from the Lemmas of Appendix \ref{app:vanishing}.
\end{proof}

\begin{cor}
Every graph $\Gamma$ appearing non-trivially in the non-leading part $\hZ_{m,n}-\hZ_{m,n}^0$ of the Maurer-Cartan element $\hZ_{m,n}$ is of loop order $\geq 1$.
\end{cor}
\begin{proof}
Suppose $\Gamma$ is a graph appearing non-trivially in the non-leading part $\hZ_{m,n}-\hZ_{m,n}^0$ of loop order zero, i.e., $\Gamma$ is a tree.
We may assume that $\Gamma$ has at least one vertex of valence $\geq 2$, since the three possible graphs with only univalent vertices are included in the leading order part $\hZ_{m,n}^0$.

Among the vertices that are at least bivalent, we always find one, say $v$, that is connected to at most one other vertex of valence $\geq 2$, and some univalent vertices $w_1,\dots,w_k$, by the combinatorics of trees.
By Lemma \ref{lem:vanishing} all univalent vertices in $\Gamma$ must be of type II and connected to type I vertices, so $v$ must be of type I and $w_1,\dots,w_k$ of type II. But by symmetry (interchanging the $w_j$) the graph vanishes unless $k=1$. (Here we used our standing assumption that $n-m$ is even.) But then $v$ has valence 2, arriving at a contradiction to Lemma \ref{lem:vanishing}. 
\end{proof}

Next, since the action of $\SO(n-m)\times \SO(m)$ naturally extends to an action of $O(n-m)\times O(m)$, we have an action of $\Z_2\times \Z_2=\pi_0(O(n-m)\times O(m))$ on $\GC_{m,n}\hotimes H(BG)$, given as follows:
\begin{itemize}
\item The generator $\tau_\perp$ of $\pi_0(O(n-m))$ acts on $E_{n-m}$ by sign, on all other generators of $H(BG)$ trivially. On a graph $\Gamma\in \fGCc_n$ of loop order $\ell$ it acts as $\tau_\perp \Gamma = (-1)^\ell\Gamma$. On a graph $\Gamma\in\GC_{m,n}'$ with $e$ edges amd $k_I$ type I vertices it acts $\tau_\perp \Gamma = (-1)^{e+v}\Gamma$.
\item The generator $\tau_\parallel$ of $\pi_0(O(n-m))$ acts on $E_{m}$ by sign, and on all other generators of $H(BG)$ trivially. On a graph $\Gamma$ of loop order $\ell$ either in $\fGCc_n$ or $\GC_{m,n}'$ it acts as $\tau_\parallel \Gamma = (-1)^\ell\Gamma$.
\end{itemize}

\begin{lemma}
The Maurer-Cartan element $\hZ_{m,n}$ is invariant under the action of $\Z_2\times \Z_2$, that is,
\[
\hZ_{m,n}\in (\GC_{m,n}\hotimes H(BG) )^{\Z_2\times \Z_2}.
\] 
\end{lemma}
\begin{proof}
First we see immediately from the definition that the first terms (proportional to $E_{n-m}E_m$) in \eqref{equ:hZmn def} are invariant. 

  Next, realize the generator $\tau_{\parallel}$ (resp. $\tau_\perp$) as a reflection of $\R^m$ along a direction in $\R^m$ (resp. in $\R^{m-1}$).
  We apply these reflections as a coordinate change in the integral \eqref{equ:hZmn def} defining $\hZ_{m,n}$.

  Note that the equivariant propagator $\Omega$ is anti-invariant under the action of both $\tau_{\parallel}$ and $\tau_\perp$, producing a sign factor $(-1)^e$, for $e$ the number of edges in the graph $\Gamma$.

  Furthermore, the reflections change the orientations of the configuration spaces integrated over as follows:
  \begin{itemize}
  \item On $\FM_n(k)$ both reflection change the orientation by $(-1)^{k-1}$.
  \item On $\FM_{m,n}(r,s)$ the reflection $\tau_\parallel$ (resp. $\tau_\perp$) changes the orientation by a factor $(-1)^{r+s-1}$ (resp. $(-1)^{s}$).
  \end{itemize}
  Taken together, the sign hence is equal to that of the action of the respective generator on $\Gamma$, thus showing the result.
\end{proof}

In particular, for $m$ odd the action of $\tau_\parallel$ on the coefficient ring $H(BG)$ is trivial. Hence we obtain:

\begin{cor}\label{cor:hZ only 2l}
  Let $m$ be odd.
  Then every graph $\Gamma$ appearing non-trivially in $\hZ_{m,n}$ is of even loop order.
  
  In particular, the non-leading part $\hZ_{m,n}-\hZ_{m,n}^0$ of the Maurer-Cartan element $\hZ_{m,n}$ is of loop order $\geq 2$.
\end{cor}


\subsection{A path object}
Let us now work over the localized coefficient ring $H(BG)_{E_{n-m}}$, formally inverting the orthogonal Euler class.
\begin{prop}\label{prop:pathobject}
The dg Lie algebra $\alg g:=(\GC_{m,n} \otimes H(BG)_{E_{n-m}})^{\hZ_{m,n}^0}$ is a path object for $\alg h=\fGCc_n \otimes H(BG)_{E_{n-m}}$.
This means that there are morphisms of dg Lie algebras factoring the diagonal
\[
\begin{tikzcd}
\alg h \ar{r}{\iota}[swap]{\sim} & \alg g 
\ar[two heads]{r}{(p_0,p_1)}
& \alg h \times \alg h
\end{tikzcd}
\]
with the right-hand map surjective and the left-hand map an (in our case injective) quasi-isomorphism.
\end{prop}
The left-hand map $\iota$ sends a graph to the sum of all graphs obtained by declaring an arbitrary subset of vertices to be of type II, multiplying by $E_{n-m}^{k-1}$, where $k$ is the number of type II vertices. For example, suppressing combinatorial prefactors, the formula schematically looks like this:
\[
\begin{tikzpicture}
\node[int] (v1) at (45:.5){};
\node[int] (v2) at (135:.5){};
\node[int] (v3) at (225:.5){};
\node[int] (v4) at (-45:.5){};
\draw (v1) edge (v2) edge (v3) edge (v4) (v2) edge (v3) edge (v4) (v3) edge (v4);
\end{tikzpicture}
\stackrel{\iota}{\mapsto}
\begin{tikzpicture}
\node[int] (v1) at (45:.5){};
\node[int] (v2) at (135:.5){};
\node[int] (v3) at (225:.5){};
\node[int] (v4) at (-45:.5){};
\draw (v1) edge (v2) edge (v3) edge (v4) (v2) edge (v3) edge (v4) (v3) edge (v4);
\end{tikzpicture}
+
\begin{tikzpicture}
\draw (-1,-.7) -- (1,-.7);
\node[int] (v1) at (45:.5){};
\node[int] (v2) at (135:.5){};
\node[int] (v3) at (225:.5){};
\node[int] (v4) at (.5,-.7){};
\draw (v1) edge (v2) edge (v3) edge (v4) (v2) edge (v3) edge (v4) (v3) edge (v4);
\end{tikzpicture}
+
E_{n-m}
\begin{tikzpicture}
\draw (-1,-.7) -- (1,-.7);
\node[int] (v1) at (45:.5){};
\node[int] (v2) at (135:.5){};
\node[int] (v3) at (-.5,-.7){};
\node[int] (v4) at (.5,-.7){};
\draw (v1) edge (v2) edge (v3) edge (v4) (v2) edge (v3) edge (v4) (v3) edge[bend left] (v4);
\end{tikzpicture}
+
E_{n-m}^2
\begin{tikzpicture}
\draw (-1,-.7) -- (1,-.7);
\node[int] (v1) at (45:.5){};
\node[int] (v2) at (0,-.7){};
\node[int] (v3) at (-.5,-.7){};
\node[int] (v4) at (.5,-.7){};
\draw (v1) edge (v2) edge (v3) edge (v4) (v2) edge[bend right] (v3) edge[bend left] (v4) (v3) edge[bend left] (v4);
\end{tikzpicture}
+
E_{n-m}^3
\begin{tikzpicture}
\draw (-1,-.7) -- (1.5,-.7);
\node[int] (v1) at (1,-.7){};
\node[int] (v2) at (0,-.7){};
\node[int] (v3) at (-.5,-.7){};
\node[int] (v4) at (.5,-.7){};
\draw (v1) edge[bend right] (v2) edge[bend right] (v3) edge[bend right] (v4) (v2) edge[bend right] (v3) edge[bend left] (v4) (v3) edge[bend left] (v4);
\end{tikzpicture}
\]
The first right-hand map $p_0$ is the projection to $\alg h$ that projects to the first factor of \eqref{equ:splitting}.
The map $p_1$ is the projection to the piece where all vertices are type II, multiplying by $E_{n-m}^{1-k}$, where $k$ is the number of type II vertices, sending all graphs with type I vertices to zero. For example:
\begin{align*}
\begin{tikzpicture}
\draw (-1,-.7) -- (1,-.7);
\node[int] (v1) at (45:.5){};
\node[int] (v2) at (135:.5){};
\node[int] (v3) at (-.5,-.7){};
\node[int] (v4) at (.5,-.7){};
\draw (v1) edge (v2) edge (v3) edge (v4) (v2) edge (v3) edge (v4) (v3) edge[bend left] (v4);
\end{tikzpicture}
&\stackrel{p_1}{\mapsto}
0
&
\begin{tikzpicture}[yshift=.7cm]
\draw (-1,-.7) -- (1.5,-.7);
\node[int] (v1) at (1,-.7){};
\node[int] (v2) at (0,-.7){};
\node[int] (v3) at (-.5,-.7){};
\node[int] (v4) at (.5,-.7){};
\draw (v1) edge[bend right] (v2) edge[bend right] (v3) edge[bend right] (v4) (v2) edge[bend right] (v3) edge[bend left] (v4) (v3) edge[bend left] (v4);
\end{tikzpicture}
\stackrel{p_1}{\mapsto}
E_{n-m}^{-3}
\begin{tikzpicture}
\node[int] (v1) at (45:.5){};
\node[int] (v2) at (135:.5){};
\node[int] (v3) at (225:.5){};
\node[int] (v4) at (-45:.5){};
\draw (v1) edge (v2) edge (v3) edge (v4) (v2) edge (v3) edge (v4) (v3) edge (v4);
\end{tikzpicture}
\end{align*}

\begin{proof}
It is an exercise to check that the maps respect the dg Lie structure:
The easiest way is to conduct a small graphical computation. 

Next let us show that the maps $\iota, p_0, p_1$ are quasi-isomorphisms, which is a less straightforward statement.
First note that it suffices to show that $p_1$ is a quasi-isomorphism, because then (by 2-out-of-3) so is $\iota$, and then (by 2-out-of-3 again) so is $p_0$.

So let us check that $p_1$ is a quasi-isomorphism.
Consider a univalent type II vertex attached to a type I vertex as a "marking" of that type I vertex.
Now take a spectral sequence on the total number of edges plus the type II vertices, plus twice the type I vertices, disregarding the markings. (I.e., the univalent type II vertices and their attaching edges to a type I vertex don't contribute to the count).
Then the differential $\delta_0$ on the associated graded of $\alg g$ becomes
\[
\delta_0 : \Gamma \mapsto E_{n-m}^\epsilon \sum_v \Gamma \sqcup(\text{add marking at vertex $v$})
\]
where the sum is over all type I vertices, and $\epsilon=0$ if $\Gamma$ is in the summand $\fGCc_n$, and $\epsilon=1$ if $\Gamma$ is in the second summand $\GC_{m,n}'$ of \eqref{equ:splitting}. In words, we add a marking to one type I vertex, summing over all choices of such vertex.
Note also that in particular, if the graph $\Gamma$ above is in $\fGCc_n$, this operation sends it to a linear combination of graphs in $\GC_{m,n}'$.
Now consider the operation $h_0'$ by summing over all vertices and removing one marking (if one is present).
By a simple computation:
\[
(\delta_0 h_0' + h_0' \delta_0)(\Gamma) = (\text{\# of type I vertices}) E_{n-m}^\epsilon \Gamma.
\]
Hence the operation 
\[
h_0: \Gamma \mapsto
\begin{cases}
0 & \text{if $\Gamma$ has no type I vertices}\\
\frac 1 {(\text{\# of type I vertices})} E_{n-m}^{-\epsilon} h_0'(\Gamma) &\text{otherwise}
\end{cases}
\]
is a homotopy for $\delta_0$, in the sense that $\delta_0 h_0 + h_0\delta_0=\mathit{id}-\pi$, where $\pi$ is the projection onto the subspace spanned by graphs without type I vertices. That means that on the level of associated graded spaces the map $p_1$ induces an isomorphism on cohomology
and hence the spectral sequence collapses here. 
\end{proof}

The following result is evident from the definitions.
\begin{lemma}\label{lem:MCelementspath}
The images of the Maurer-Cartan element $\hZ_{m,n}-\hZ_{m,n}^0$ under the maps $p_0$, $p_1$ of the preceding Proposition \ref{prop:pathobject} are as follows, where $L$ is the generator of the loop order grading.
\begin{align*}
p_0(\hZ_{m,n}-\hZ_{m,n}^0) &= Z_{m,n}^n \\
p_1(\hZ_{m,n}-\hZ_{m,n}^0) &= L^{E_{n-m}} Z_{m}^m 
\end{align*}
\hfill\qed
\end{lemma}

\begin{cor}\label{cor:gauge equiv GC2}
$Z_{m,n}^n$ and $L^{E_{n-m}} Z_{m}^m$ are gauge equivalent Maurer-Cartan elements in $\GC_m^2\hotimes H(BG)_{E_{n-m}}$.
\end{cor}
For the proof we denote by $\mF^\bullet$ the descending complete filtration by loop order on our graph complexes, that is, $\mF^p\GC_{m,n}\subset \GC_{m,n}$ is generated by graphs of loop orders $\geq p$.
We note that the morphisms $\iota, p_0, p_1$ of Proposition \ref{prop:pathobject} all respect the loop order grading.
\begin{proof}
We claim that the morphisms $\iota, p_0, p_1$ induce bijections between the gauge equivalence classes of Maurer-Cartan elements 
\[
\begin{tikzcd}
  \pi_0 \MC(\mF^1\fGCc_{n}\hotimes H(BG)_{E_{n-m}}) 
\ar{r}{\iota}
&
\pi_0 \MC(\mF^1\GC_{m,n}\hotimes H(BG)_{E_{n-m}}) 
\ar[shift left]{r}{p_0}
\ar[shift right]{r}[swap]{p_1}
&
\pi_0 \MC(\mF^1\fGCc_{n}\hotimes H(BG)_{E_{n-m}}).
\end{tikzcd}
\] 
This claim follows immediately from the Goldman-Millson Theorem \cite{DolRog} and the quasi-isomorphism statement of Proposition \ref{prop:pathobject}. Just note that the quasi-isomorphism statement automatically holds on each subcomplex of fixed loop order separately, since the differential does not change the loop order.
Let $[\alpha]$ denote the gauge equivalence class of a Maurer-Cartan element $\alpha$. Then
\[
p_0([\hZ_{m,n}-\hZ_{m,n}^0])
=[Z_{m,n}^n]
= p_0(\iota([Z_{m,n}^n]))
\]
The claim above implies that 
\[
  [\hZ_{m,n}-\hZ_{m,n}^0] = \iota([Z_{m,n}^n]).
\]
Similarly, we have 
\[
p_1([\hZ_{m,n}-\hZ_{m,n}^0])
=[L^{E_{n-m}} Z_{m}^m]
= p_1(\iota([L^{E_{n-m}} Z_{m}^m]))
\]
and hence
\[
  [\hZ_{m,n}-\hZ_{m,n}^0] = \iota([L^{E_{n-m}} Z_{m}^m]).
\]
Thus we have 
\[
  \iota([Z_{m,n}^n]) = \iota([L^{E_{n-m}} Z_{m}^m])
\]
and using our claim from the beginning once again this implies that
\[
  [Z_{m,n}^n] = [L^{E_{n-m}} Z_{m}^m].
\]
In other words we have shown that $Z_{m,n}^n$ and $L^{E_{n-m}} Z_{m}^m$ are gauge equivalent in $\mF^1\fGCc_{n}\hotimes H(BG)_{E_{n-m}}$. On the other hand, the inclusion 
\[
\GC^2_n \to \mF^1\fGCc_{n}
\]
is a quasi-isomorphism in each loop order, and hence by the Goldman-Millson Theorem we have that 
\[
  \pi_0\MC(\GC^2_n) \cong  \pi_0\MC(\mF^1\fGCc_{n}),
\]
so that $Z_{m,n}^n$ and $L^{E_{n-m}} Z_{m}^m$ are also gauge equivalent in $\GC^2_{n}\hotimes H(BG)_{E_{n-m}}$ as desired.
\end{proof}

\begin{cor}\label{cor:hZprime}
There is a Maurer-Cartan element $\hZ_{m,n}'\in (\mF^2\GC_{m,n}\hotimes H(BG))^{\hZ_{m,n}^0+\hZ_{m,n}^1}$ such that  
\begin{equation}\label{equ:hZp proj}
  \begin{aligned}
    p_0(\hZ_{m,n}'+\hZ_{m,n}^1) &= Z_{m,n}^n \\
    p_1(\hZ_{m,n}'+\hZ_{m,n}^1) &= L^{E_{n-m}} Z_{m}^m 
    \end{aligned}.
  \end{equation}
\end{cor}
\begin{proof}
For $m$ odd we may simply take $\hZ_{m,n}'=\hZ_{m,n}-\hZ_{m,n}^0$, which is concentrated in loop orders $\geq 2$ by Corollary \ref{cor:hZ only 2l}.

For $m$ even we have that $\hZ_{m,n}-\hZ_{m,n}^0$ has a non-trivial part of loop order $1$, and we can also not show that $X:=\hZ_{m,n}-\hZ_{m,n}^0-\hZ_{m,n}^1$ is concentrated in loop orders $\geq 2$. Let $X_1$ be the part of $X$ of loop order 1. Then we have that $p_0(X)=p_1(X)=0$ and $dX_1=0$, where $d$ denotes the (twisted) differential on $\GC_{m,n}^{\hZ_{m,n}^0}$.

Let 
$$
V=\{x\in \GC_{m,n}\hotimes H(BG)_{E_{n-m}}\mid p_0(x)=p_1(x)=0, \text{$x$ of loop order 1} \}
$$  
be the joint kernel of $p_0$ and $p_1$ in the loop order 1 subspace of $\GC_{m,n}\hotimes H(BG)_{E_{n-m}}$, equipped with the differential $d$. Since $p_0$ and $p_1$ are quasi-isomorphisms such that $p_0\oplus p_1$ is surjective, we have that the cohomology of $V$ is identified with that of $\fGCc_n^{1\text{-loop}}\hotimes H(BG)_{E_{n-m}}$ up to a degree shift,
\[
H^\bullet(V) = H^{\bullet-1}(\fGCc^{1\text{-loop}}\hotimes H(BG)_{E_{n-m}} )
=
\bigoplus_{k\equiv 1\text{ mod }4}
\R[n-k-1]\otimes H(BG)_{E_{n-m}}.
\]
In particular, the cohomology is concentrated in even degrees. But now $X_1\in V$ is an cocycle of odd degree and hence exact.
That is, there is some $Y\in V$, such that $X_1=dY$.
But then we may gauge transform $\hZ_{m,n}$ with the gauge transformation induced by $Y$ to obtain the desired Maurer-Cartan element 
\[
e^Y\cdot \hZ_{m,n} =:  \hZ_{m,n}^0+\hZ_{m,n}^1+\hZ_{m,n}',
\]
with $\hZ_{m,n}'$ concentrated in loop orders $\geq 2$. (Here $e^Y\cdot$ denotes the gauge transformation corresponding to $Y$.)
Furthermore, since $p_0(Y)=p_1(Y)=0$ we have for $j\in\{1,2\}$
\[
 p_j( e^Y\cdot \hZ_{m,n})
 =
 e^{p_j(Y)}\cdot p_j(\hZ_{m,n})
 =
 p_j(\hZ_{m,n}),
\]
so that \eqref{equ:hZp proj} is evident.
\end{proof}

Now we can show Theorem \ref{thm:mainloc}.

\begin{proof}[Proof of Theorem \ref{thm:mainloc}]
It suffices to show the statement for $l=n-k$.
We may proceed along the lines of the proof of Corollary \ref{cor:gauge equiv GC2}, just replacing $\mF^1(-)$ by $\mF^2(-)$.
Concretely, we have the diagram 
\[
\begin{tikzcd}
  \pi_0 \MC(\mF^2\fGCc_{n}\hotimes H(BG)_{E_{n-m}}^\alpha) 
\ar{r}{\iota}
&
\pi_0 \MC(\mF^2\GC_{m,n}\hotimes H(BG)_{E_{n-m}}^{\beta}) 
\ar[shift left]{r}{p_0}
\ar[shift right]{r}[swap]{p_1}
&
\pi_0 \MC(\mF^2\fGCc_{n}\hotimes H(BG)_{E_{n-m}}^\alpha).
\end{tikzcd}
\] 
with $\alpha=-E_{n-m}E_m\tadpole$ and $\beta=\hZ_{m,n}^0+\hZ_{m,n}^1$.
It again follows from Proposition \ref{prop:pathobject} and the Goldman-Millson Theorem that all three arrows $\iota,p_0,p_1$ in the diagram are bijections.
Now, proceeding as in the proof of Corollary \ref{cor:gauge equiv GC2}, just with the MC element $\hZ_{m,n}'$ of Corollary \ref{cor:hZprime}, we conclude that $Z_{m,n}^n-\alpha$ and $L^{E_{n-m}} Z_{m}^m-\alpha$ are gauge equivalent Maurer-Cartan elements of $\mF^2\fGCc_{n}\hotimes H(BG)_{E_{n-m}}^\alpha$.
But now the inclusion 
\[
\GC_{n}\to \mF^2\fGCc_{n}
\]
is a quasi-isomorphism in each loop order separately. Hence by the Goldman-Millson Theorem we have that 
\[
  \pi_0\MC(\GC_{n} \hotimes H(BG)^\alpha)
  \cong \pi_0\MC(\fGCc_{n}\hotimes H(BG)^\alpha).
\]
Hence $Z_{m,n}^n-\alpha$ and $L^{E_{n-m}} Z_{m}^m-\alpha$ are also gauge equivalent in $\GC_{n} \hotimes H(BG)^\alpha$. But this is equivalent to $Z_{m,n}^n$ and $L^{E_{n-m}} Z_{m}^m$ being gauge equivalent in $\GC_n^+$. 


\end{proof}

\begin{rem}
We remark that gauge equivalence in $\GC_n^2$ as asserted in Corollary \ref{cor:gauge equiv GC2}, instead of gauge equivalence in $\GC_n^+$ as asserted by Theorem \ref{thm:mainloc}, would be sufficient to derive the results of section \ref{sec:MC}, and hence the main results of this paper. 
However, the smaller dg Lie algebra $\GC_n^+$ is more convenient to work with, since it is finite dimensional in every loop order, and since it acts on Kontsevich's model $\Graphs_n$ for the little disks operad.
\end{rem}

\subsection{Proof of Theorem \ref{thm:main equiv}}

We may now connect all threads to show Theorem \ref{thm:main equiv}.

Let $(G,\POp)\in \sGrp\sOp$ be a simplicial model for the pair $(\SO(n),\FM_n)$, for example $(G,\POp)=(\sS\SO(n),\sS\FM_n)$.
We desire to establish a zigzag of weak equivalences in $\dgca_*(\dgca^{/})^{\cT^{op}}$
\[
    (\Omega(\bar WG), \Omega(N\POp\sslash G))
    \to \bullet \leftarrow
    (H(B\SO(n)), N\BGraphs_n^{m_n}).
\]
This will be done in several steps.

From Lemma \ref{lem:Om PA zigzag} we find that there is a zigzag of weak equivalences (quasi-isomorphisms) in $\dgca_*(\dgca^{/})^{\cT^{op}}$
\[
    (\Omega(\bar WG), \Omega(N\POp\sslash G))
    \to \bullet \leftarrow
    (\Omega^{PA}_{\SO(n)}(*), \Omega^{PA}_{\SO(n)}(N\FM_n)).
\]
By Theorem \ref{thm:equivariant model} we have a further weak equivalence 
\[
    (H(B\SO(n)), \BGraphs_n^{\tm^n})
    \to
    (\Omega^{PA}_{\SO(n)}(*), \Omega^{PA}_{\SO(n)}(N\FM_n)),
\]
using the Maurer-Cartan element $\tm^n\in \left(\GC_n^+\hotimes H(B\SO(n))\right)^{\mathbb Z_2}$ defined via configuration space integrals in section \ref{sec:MC int formula}.

Now, by Theorem \ref{conjthm:main} we know that 
$\tm^n$ is gauge equivalent to the Maurer-Cartan element $m_n=Z^n_{conj}$ of \eqref{equ:mn intro}.
That is, there exists a degree zero element $\phi\in(\GC_n^+\hotimes H(B\SO(n)))^0$ such that 
\[
\exp(\phi)\cdot m = m_n, 
\]
where $\cdot$ denotes the gauge action.
But due to the action of $\GC_n^+\hotimes H(B\SO(n))$ on $\BGraphs_n$ this implies that we have an isomorphism of dg Hopf cooperads under $H(B\SO(n))$
\begin{gather*}
\exp(\phi) : \BGraphs_n^{\tm^n} \to \BGraphs_n^{m_n} \\
\Gamma \mapsto \exp(\phi)\cdot \Gamma
= \sum_{j\geq 0}\frac{1}{j!}
(\phi\cdot)^j \Gamma,
\end{gather*}
where $\cdot$ on the right-hand side now denotes the action.
Note that the series has only finitely many terms for each $\Gamma$, since each application of $\phi\cdot$ strictly reduces the number of edges.
This means that we have an isomorphism of pairs in $\dgca_*(\dgca^{/})^{\cT^{op}}$ 
\[
    (H(B\SO(n)), \BGraphs_n^{\tm^n})
    \xrightarrow{(\mathit{id}, \exp(\phi))}
    (H(B\SO(n)), \BGraphs_n^{m_n})
\] 
so that Theorem \ref{thm:main equiv} follows.

\hfill\qed 

\begin{rem}
In the earlier version of this manuscript the Maurer-Cartan element $m_n$ for even $n$ was $E\tadpole$ instead of $-E\tadpole$, due to different conventions for the Cartan model.
However, the sign and in fact prefactor of the term $E\tadpole$ is indeed irrelevant, as can be seen as follows. For any number $\lambda\neq 0$ consider the isomorphism 
\[
\phi_\lambda: \Graphs_n \to \Graphs_n
\]
that sends a graph $\Gamma$ with $e$ edges and $v$ internal vertices to $\phi_\lambda(\Gamma)=\lambda^{e-v} \Gamma$.
Then this isomorphism intertwines the action of a graph $\gamma\in \GC_n^+$ of loop order $\ell$ with the action of $\lambda^\ell \gamma$, i.e., 
\[
 \phi_\lambda(\gamma\cdot\Gamma)
=
\lambda^\ell \gamma\cdot \phi_\lambda(\Gamma).
\]
In particular, for even $n$ we have that $\phi_{-1}$ induces isomorphisms 
\begin{align*}
\BGraphs_n^{m_n} &\cong \BGraphs_n^{-m_n}
&
\Graphs_n \rtimes_{m_n} \mU\fg_n^c \cong  \Graphs_n \rtimes_{-m_n} \mU\fg_n^c.
\end{align*}
\end{rem}

\appendix

\section{Vanishing Lemmas for graphs with (certain) bivalent and univalent vertices}\label{app:vanishing}
The goal here is to show that the integral weights of graphs involving bivalent vertices of several types vanish. This can be done by using the standard reflection argument due to Kontsevich.
\begin{lemma}\label{lem:bivalentvanish}
The form associated to the following graph under the map \eqref{equ:Kintegralequiv} vanishes:
\[
\begin{tikzpicture}
\node[int](v) at (0,0.2){};
\node[ext](v1) at (-0.5,0){1};
\node[ext](v2) at (0.5,0.5){2};
\draw (v) edge (v1) edge (v2);
\end{tikzpicture}
\xrightarrow{\omega} 0.
\]
\end{lemma}
\begin{proof}
We can use the argument of \cite[Lemma 2.2]{KFeynman}, which we reproduce here for completeness, and to verify that it also works in the equivariant setting.
Number the black point as 3, and denote by $\alpha_{ij}$ the propagator between points $i$ and $j$, where $i,j\in \{1,2,3\}$.
The form above is then obtained as a fiber integral, integrating out point 3,
\[
 \int_3 \alpha_{13}\alpha_{23}.
\]
Note that we have chosen our propagator (anti-)invariantly under the inversion, hence $\alpha_{ij}=(-1)^n\alpha_{ji}$.
Now apply an inversion through the midpoint between points 1 and 2 to the integration variable, i.e., reflect the position of point 3 at that midpoint.
As is quickly verified using assertion (1) of Lemma \ref{lem:prop properties}, this change of variables transforms our integral into
\[
 (-1)^n\int_3 \alpha_{32}\alpha_{31}=(-1)^n\int_3 \alpha_{23}\alpha_{13}=-\int_3 \alpha_{13}\alpha_{23}.
\]
In the last equality we have used that the two forms are of degree $n-1$. It follows that the integral equals minus itself and is hence zero.
\end{proof}

Similarly, one shows the following:

\begin{lemma}
The form associated to the following graph (see Section \ref{sec:two v graphs}) vanishes:
\[
\begin{tikzpicture}
\draw (-1,0)--(1,0);
\node[int](v) at (0,0){};
\node[ext](v1) at (-0.5,0){1};
\node[ext](v2) at (0.5,0){2};
\draw (v) edge[bend right] (v1) edge[bend left] (v2);
\end{tikzpicture}
\to 0.
\]
\end{lemma}
\begin{proof}
The edge between the two type II vertices is assigned the form 
\[
E_{n-m} \Omega_{sm}^{n-m\text{-dim}},
\]
see Remark \ref{rem:prop restrict}, where $\Omega_{sm}^{n-m\text{-dim}}$ is the $m$-dimensional propagator. Hence applying the previous Lemma (with $n$ replaced by $n-m$) gives the result.
\end{proof}

%

\begin{lemma}
The weights of all graphs containing univalent internal vertices vanish, except for the graphs occurring in $\hZ^0_{m,n}$ in \eqref{equ:hZmn} above, and except for one-valent type II vertices that may be attached to type I vertices:
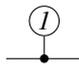
\[
\begin{tikzpicture}
\draw (-.5,0) -- (.5,0);
\node[int](w) at (0,0){};
\node[ext](v2) at (0,.5){1};
\draw (v2) edge (w);
\end{tikzpicture}
\]
\end{lemma}
\begin{proof}
Consider a graph with a univalent vertex $v$. We distinguish several cases.
First, suppose $v$ is of type I, and the graph has at least 2 other vertices.
The vanishing of the configuration space integral is purely due to degree reasons: Consider the points in the configuration other than that (say $x$) corresponding to $v$ fixed.
Then $x$ traces out an $n$-dimensional space, but there are at most $n-1$ form degrees along $x$, hence the integral is zero.
The same argument works for the case that there is one other type I vertex (and the graph is of the type with a baseline).
This settles the case of $v$ of type I.

Next suppose $v$ is of type II, with the single edge connecting it to another type II vertex.
If there is at least one more vertex in the graph, the integral vanishes by analogous reasoning as before, just in lower dimension.
If not, we have a graph occurring in \eqref{equ:hZmn}.
\end{proof}

\bibliographystyle{plain}

\end{document}